\documentclass[10pt,a4paper]{amsart}
 \usepackage{amsmath}
\usepackage{amssymb}
\usepackage[english]{babel}
\usepackage[arrow, matrix, curve]{xy}
\usepackage{amsmath,amscd}
\numberwithin{equation}{section}

\usepackage{hyperref,srcltx}
\usepackage[T1]{fontenc}
\usepackage[latin1]{inputenc}

\DeclareMathOperator\GL{GL}

\DeclareMathOperator\Spec{Spec}

\DeclareMathOperator\Lie{Lie}
\DeclareMathOperator\cind{c-ind}

\DeclareMathOperator\pr{pr}

\renewcommand{\phi}{\varphi}

\newcommand{\Fcal}{\mathcal{F}}
\newcommand{\Gcal}{\mathcal{G}}
\newcommand{\Hcal}{\mathcal{H}}

\newcommand{\Mcal}{\mathcal{M}}
\newcommand{\Ncal}{\mathcal{N}}
\newcommand{\Ocal}{\mathcal{O}}

\newcommand{\lieg}{\mathfrak{g}}

\newcommand{\Q}{\mathbb{Q}}

\newcommand{\Abb}{\mathbb{A}}

\newcommand{\Gbb}{\mathbb{G}}

\newtheorem{theo}{Theorem}[section]
\newtheorem{lem}[theo]{Lemma}
\newtheorem{prop}[theo]{Proposition}
\newtheorem{cor}[theo]{Corollary}
\newtheorem{conj}[theo]{Conjecture}
\theoremstyle{remark}
\newtheorem{rem}[theo]{Remark}
\theoremstyle{remark}
\newtheorem{expl}[theo]{Example}
\theoremstyle{definition}

\usepackage{color}
\newcounter{commentcounter}

\newcounter{commentcounterS}

\def\?{\ 
{\bf\color{red}???}\ 
\immediate\write16{}
\immediate\write16{Warning: There was still a question mark . . . }
\immediate\write16{}}

\author{Eugen Hellmann}

\address{Eugen Hellmann\\
Mathematisches Institut\\
Universit\"at M\"unster\\
Einsteinstrasse 62\\
D-48149 M\"unster\\
Germany\\
e.hellmann@uni-muenster.de}

\title{On the derived category of the Iwahori-Hecke algebra}

\begin{document}

\begin{abstract}
We state a conjecture that relates the derived category of smooth representations of a $p$-adic split reductive group with the derived category of (quasi-)coherent sheaves on a stack of L-parameters. We investigate the conjecture in the case of the principal block of ${\rm GL}_n$ by showing that the functor should be given by the derived tensor product with the family of representations interpolating the modified Langlands correspondence over the stack of L-parameters that is suggested by the work of Helm and Emerton-Helm. 
\end{abstract}

\maketitle
\setcounter{tocdepth}{1}
\tableofcontents

\section{Introduction}
We study the smooth representation theory of a split reductive group $G$ over a non-archimedean local field $F$. The classification of the irreducible smooth $G$-representations is one of the main objectives of the local Langlands program. 
One aims to parametrize these representations by so called L-parameters, together with some additional datum (a representation of a finite group associated to the L-parameter). 
Such a parametrization has been established in the case of ${\rm GL}_n(F)$. For split reductive groups it has been established by Kazhdan and Lusztig for those irreducible smooth representations of $G$ that have a non-trivial fixed vector under an Iwahori subgroup $I\subset G$, see \cite{KL}. In this case an L-parameter just becomes a conjugacy class of $(\phi,N)$, where $\phi$ is a semi-simple element of the Langlands dual group $\check G$, and $N\in\Lie\check G$, satisfying ${\rm Ad}(\phi)(N)=q^{-1}N$. Here $q$ is the number of elements of the residue field of $F$. This parametrization depends on an additional choice, called a Whittaker datum. 

In this paper we formulate a conjecture that lifts the Langlands classification to a fully faithful embedding of the category ${\rm Rep}(G)$ of smooth $G$-representations (on vector spaces over a field $C$ of characteristic zero) to the category of quasi-coherent sheaves on the stack of L-parameters. It turns out that this conjecture has to be formulated on the level of derived categories. As one of the main tools in the study of smooth representations is parabolic induction, we ask this fully faithful embedding to be compatible with parabolic induction in a precise sense. Moreover, the conjectured functor should depend on the choice of a Whittaker datum.
Similar conjectures, and results, where obtained by Ben-Zvi--Chen-Helm-Nadler \cite{BZCHN} and Zhu \cite{Zhu}. The conjectures stated here can also be regarded as a special case of the conjectural geometrization of the local Langlands correspondence of Fargues-Scholze \cite{FS}. 

The conjecture can be made more precise in the case of the principal Bernstein Block ${\rm Rep}_{[T,1]}(G)$ of ${\rm Rep}(G)$, i.e.~the block containing the trivial representation. 
This block coincides with the full subcategory ${\rm Rep}^IG$ of smooth $G$-representations generated by their $I$-fixed vectors for a choice of an Iwahori-subgroup $I\subset G$. As ${\rm Rep}^IG$ is equivalent to the category of modules over the Iwahori-Hecke algebra the conjecture comes down to a conjecture about the derived category of the Iwahori-Hecke algebra. 

In the main part of the paper we investigate the conjecture in the case of $G={\rm GL}_n(F)$ and the principal block by relating it to the construction of a family of $G$-representations interpolating the (modified) local Langlands correspondence, following the work of Emerton-Helm \cite{EmertonHelm}.

We describe the conjecture and our results in more detail. 
Fix a finite extension $F$ of $\Q_p$, or of $\mathbb{F}_p(\!(t)\!)$.
Let $\mathbb{G}$ be a split reductive group over $F$ and write $G=\mathbb{G}(F)$. 
We fix a field $C$ of characteristic zero and shall always assume that $C$ contains a square root $q^{1/2}$ of $q$. We denote by $\check G$ the dual group of $G$, considered as a reductive group over $C$. 
More generally, for every parabolic (or Levi) $\mathbb{P}$ (or $\mathbb{M}$) of $\mathbb{G}$ we will write $P=\mathbb{P}(F)$ (respectively $M=\mathbb{M}(F)$) for its group of $F$-valued points and $\check P$ (respectively $\check M$) for its dual group over $C$.
For each parabolic subgroup $P\subset G$ with Levi $M$ (normalized) parabolic induction defines a functor $\iota_{\overline{P}}^G$ from $M$-representations to $G$-representations.

On the other hand we denote by $X_{\check G}^{\rm WD}$ the space of Weil-Deligne representations with values in $\check G$, that is, the space whose $C$-valued points are pairs $(\rho,N)$ consisting of a smooth representation $W_F\rightarrow \check G(C)$ of the Weil group $W_F$ of $F$ and $N\in \Lie \check G$ satisfying the usual relation $${\rm Ad}(\rho(\sigma))(N)=q^{-||\sigma||}N,$$ where $||-||:W_F\rightarrow\mathbb{Z}$ is the projection. 
We shall write $[X_{\check G}^{\rm WD}/\check G]$ for the stack quotient by the obvious $\check G$-action. 

Let us write $\mathcal{Z}(\check G)$ for the global sections of the structure sheaf on $[X_{\check G}^{\rm WD}/\check G]$, or equivalently the coordinate ring of the GIT quotient $X_{\check G}^{\rm WD}/\hspace{-.1cm}/\check G$. Moreover, we write $\mathfrak{Z}(G)$ for the Bernstein center of ${\rm Rep}(G)$.
With these notations we state the following conjecture. For the sake of brevity we state the conjecture in a vague form and refer to the body of the paper for a more precise formulation of the conjecture.

\begin{conj}\label{mainconjecturegeneralintro}
There exists the following data:
\begin{enumerate}
\item[(i)] For each $(\mathbb{G},\mathbb{B},\mathbb{T},\psi)$ consisting of a reductive group $\mathbb{G}$, a Borel subgroup $\mathbb{B}$, a split maximal torus $\mathbb{T}\subset\mathbb{B}$, and a (conjugacy class of a) generic character $\psi:N\rightarrow C^\times$ there exists
an exact and fully faithful functor
\[\mathcal{R}_G^\psi:{\bf D}^+({\rm Rep}(G))\longrightarrow {\bf D}^+_{\rm QCoh}([X^{\rm WD}_{\check G}/\check G]),\]
\item[(ii)] for $(\mathbb{G},\mathbb{B},\mathbb{T},\psi)$ as in {\rm (i)} and a parabolic subgroup $\mathbb{P}\subset \mathbb{G}$ containing $\mathbb{B}$ we denote by $\mathbb{M}$ the Levi-quotient of $\mathbb{P}$. Then the functors $\mathcal{R}^\psi_G$ and $\mathcal{R}_M^{\psi_M}$ satisfy a compatibility with parabolic induction $\iota_{\overline{P}}^G$. Here $\psi_M$ is the restriction of $\psi$ to the unipotent radical of the Borel $B_M$ of $M$ via the splitting $\mathbb{M}\rightarrow \mathbb{P}$ defined by $\mathbb{T}$.
\end{enumerate}
These data satisfy the following conditions:
\begin{enumerate}
\item[(a)] If $\mathbb{G}=\mathbb{T}$ is a split torus, then $\mathcal{R}_T=\mathcal{R}_T^{\psi}$ is induced by the equivalence $${\rm Rep}(T)\cong {\rm QCoh}(X_{\check T}^{\rm WD})$$ given by  local class field theory.
\item[(b)] Let $(\mathbb{G},\mathbb{B},\mathbb{T},\psi)$ be as in {\rm (i)}. The morphism $\mathcal{Z}(\check G)\rightarrow \mathfrak{Z}(G)$ defined by fully faithfulness of $\mathcal{R}_G^\psi$ is independent of the choice of $\psi$ and induces a surjection
\[\omega_G:\left\{\begin{array}{*{20}c} \text{Bernstein components}\\  \text{of}\ {\rm Rep} (G)\end{array}\right\}\longrightarrow\left\{\begin{array}{*{20}c}\text{connected components}\\ \text{of}\ X_{\check G}^{\rm WD}\end{array}\right\}.\]
\item[(c)] For $(\mathbb{G},\mathbb{B},\mathbb{T},\psi)$ as in {\rm (i)} there is an isomorphism
\[R_G^\psi((\cind_N^G\psi))\cong \Ocal_{[X^{\rm WD}_{\check G}/\check G]}.\]
\end{enumerate}
\end{conj}

In this paper we mainly focus on the conjecture in the case of the principal block of ${\rm Rep}(G)$.
If $T\subset G$ is a split maximal torus, we write ${\rm Rep}_{[T,1]}(G)$ for the Bernstein block of those representations $\pi$ such that all irreducible subquotients of $\pi$ are subquotients of a representations induced from an unramified $T$-representation. 
Then parabolic induction restricts to a functor $$\iota_{\overline{P}}^G:{\rm Rep}_{[T_M,1]}(M)\rightarrow {\rm Rep}_{[T,1]}(G)$$ for any choice of maximal split tori $T\subset G$ and $T_M\subset M$ (as the categories do not depend on these choices).

On the other hand we denote by $X_{\check G}=\{(\phi,N)\in\check G\times\Lie\check G\mid {\rm Ad}(\phi)(N)=q^{-1}N\}$ the space of L-parameters (corresponding to the representations in the principal block) and write $[X_{\check G}/\check G]$ for the stack quotient by the action of $\check G$ induced by conjugation. We obtain similar spaces $[X_{\check P}/\check P]$ etc.~for parabolic subgroups $P\subset G$ (or their Levi quotients). If $T$ is a (maximal split) torus, then $X_{\check T}$ is just the dual torus $\check T$. 

In this case the relation between the the Bernstein center $\mathfrak{Z}_G=\mathfrak{Z}_{[T,1]}(G)$ of the category ${\rm Rep}_{[T,1]}(G)$ and the GIT quotient $X_{\check G}/\hspace{-.1cm}/\check G$ can be made precise: the center $\mathfrak{Z}_G$ can naturally be identified with the functions on the adjoint quotient of $\check G$ and hence $\mathfrak{Z}_G$ acts on categories of modules over $X_{\check G}$ as well as on ${\rm Rep}_{[T,1]}(G)$.
The following conjecture is a slightly more precise version of Conjecture \ref{mainconjecturegeneralintro} in the case of the block ${\rm Rep}_{[T,1]}(G)$. Equivalently, the conjecture can be interpreted as a conjecture about the derived category of the Iwahori-Hecke algebra, and we shall take this point of view in the last part of the paper when we discuss the case of ${\rm GL}_n$. In the case of modules over an affine Hecke algebra (where $q$ is an invertible indeterminate) a similar conjecture\footnote{The author was not aware of their project when coming up with the conjecture and with the results in this paper.} is due to Ben-Zvi--Chen--Helm-Nadler \cite{BZCHN} and Zhu \cite{Zhu}.

\begin{conj}\label{mainconjectureintro}
There exists the following data:
\begin{enumerate}
\item[(i)] For each $(\mathbb{G},\mathbb{B},\mathbb{T},\psi)$ consisting of a reductive group $\mathbb{G}$ a Borel subgroup $\mathbb{B}$, a maximal split torus $\mathbb{T}\subset\mathbb{B}$ and a (conjugacy class of a) generic character $\psi:N\rightarrow C^\times$ there exists
an exact and fully faithful $\mathfrak{Z}_G$-linear functor
\[R_G^\psi:{\bf D}^+({\rm Rep}_{[T,1]}(G))\longrightarrow {\bf D}^+_{\rm QCoh}([X_{\check G}/\check G]),\]
\item[(ii)] for $(\mathbb{G},\mathbb{B},\mathbb{T},\psi)$ as in {\rm (i)} and each parabolic subgroup $\mathbb{P}\subset \mathbb{G}$ containing $\mathbb{B}$ 
there exists a natural $\mathfrak{Z}_{G}$-linear isomorphism
\[\xi_{P}^{G}:R_{G}^\psi\circ \iota_{\overline{P}}^{G}\longrightarrow (R\beta_{\ast}\circ L\alpha^\ast)\circ R_{M}^{\psi_M}\]
of functors ${\bf D}^+({\rm Rep}_{[T_M,1]}M)\rightarrow {\bf D}^+_{\rm QCoh}([X_{\check G}/\check G])$ such that the various $\xi_{P}^{G}$ are compatible (in a precise sense).
Here $\mathbb{M}$ is the Levi quotient of $\mathbb{P}$, the character $\psi_M$ is the restriction of $\psi$ to the unipotent radical of $B_M=B\cap M$ (using a splitting $M\hookrightarrow P$ of $P\rightarrow M$), and 
\begin{align*}
\alpha:&[X_{\check P}/\check P]\longrightarrow [X_{\check M}/\check M] \\
\beta:&[X_{\check P}/\check P]\longrightarrow [X_{\check G}/\check G] 
\end{align*}
are the morphisms on stacks induced by the natural maps $\check P\rightarrow \check M$ and $\check P\rightarrow \check G$.
\end{enumerate}
For a maximal split torus $T$ the functor $R_T=R_T^{\psi_T}$ is induced by the identification
\[{\rm Rep}_{[T,1]}(T)\cong C[T/T^\circ]\text{-}{\rm mod}\cong{\rm QCoh}(\check T),\]
were $T^\circ\subset T$ is the maximal compact subgroup. Moreover, for $(\mathbb{G},\mathbb{B},\mathbb{T},\psi)$ as in {\rm (i)} there is an isomorphism 
\[R_G^\psi((\cind_N^G\psi)_{[T,1]})\cong \Ocal_{[X_{\check G}/\check G]}.\]
\end{conj}
In fact it turns out that in the formulation of the conjecture the stack $[X_{\check P}/\check P]$ has to be replaced by a derived variant. Again, we refer to the body of the paper for details and a more precise formulation of the conjecture. 

In the case $G={\rm GL}_n(F)$ we consider a candidate for the conjectured functor. Emerton and Helm \cite{EmertonHelm} have suggested (in the context of $\ell$-adic deformation rings rather than the stack $[X_{\check G}/\check G]$) the existence of a family $\mathcal{V}_G$ of smooth $G$-representations on $[X_{\check G}/\check G]$ that interpolates the \emph{modified} local Langlands correspondence.
A candidate for the family $\mathcal{V}_G$ was constructed by Helm in \cite{Helm1}.
The modified local Langlands correspondence assigns to $(\phi,N)\in X_{\check G}(C)$ a certain representation ${\rm LL^{mod}}(\phi,N)$ that is indecomposable, induced from a parabolic subgroup, has a unique irreducible subrepresentation, which is a generic representation, and its unique irreducible quotient is the representation ${\rm LL}(\phi,N)$ associated to $(\phi,N)$ by the local Langlands correspondence.
In the context of modules over the Iwahori-Hecke algebra $\Hcal_G$ the $\Hcal_G$-modules corresponding to the representations ${\rm LL^{mod}}(\phi,N)$ are often referred to as the \emph{standard modules}.

We conjecture that, in the ${\rm GL}_n$-case, the functor $R_G=R_G^\psi$ should be given by the derived tensor product with $\mathcal{V}_G$ (we omit the superscript $\psi$ from the notation as in the case of ${\rm GL}_n$ there is a unique Whittaker datum). For the precise formulation it is more convenient to pass from $G$-representations to modules over the Iwahori-Hecke algebra $\Hcal_G$. The family of $\Hcal_G$-modules associated to $\mathcal{V}_G$ by taking $I$-invariants is in fact a $\Hcal_G\otimes_{\mathfrak{Z}_G}\Ocal_{[X_{\check G}/\check G]}$-module $\mathcal{M}_G$ that is coherent as an $\Ocal_{[X_{\check G}/\check G]}$-module.

We consider the functor 
\begin{equation}\label{EHfunctorintro}
R_G:{\bf D}^+(\Hcal_G\text{-}{\rm mod}) \longrightarrow {\bf D}^+_{\rm QCoh}([X_{\check G}/\check G])
\end{equation}
mapping $\pi$ to ${}^t\pi\otimes^L_{\Hcal_G}\mathcal{M}_G$. Here ${}^t\pi$ is $\pi$ considered as a right module over $\Hcal_G$ by means of the standard involution $\Hcal_G\cong \Hcal_G^{\rm op}$, and we point out that the derived tensor product can easily be made explicit, as $\Hcal_G$ has finite global dimension.
Every (standard) Levi subgroup of ${\rm GL}_n(F)$ is a product of some ${\rm GL}_m(F)$, and we hence can construct similar functors 
\begin{equation}\label{EHfunctorintro2}
R_M:{\bf D}^+(\Hcal_M\text{-}{\rm mod}) \longrightarrow {\bf D}^+_{\rm QCoh}([X_{\check M}/\check M])
\end{equation}
for every Levi $M$.
Over a certain (open and dense) regular locus $X_{\check G}^{\rm reg}$ of $X_{\check G}$, see section \ref{sectionbasicprop} for the definition, we can relate the functor $R_G$ to Conjecture \ref{mainconjectureintro} as follows.

\begin{theo} Let $\mathbb{G}={\rm GL}_n$. 
For each parabolic $\mathbb{P}\subset\mathbb{G}$ with Levi $\mathbb{M}$ the restriction of $(\ref{EHfunctorintro2})$ to the regular locus is a $\mathfrak{Z}_M$-linear functor
\[R_M^{\rm reg}:{\bf D}^+(\Hcal_M\text{-}{\rm mod}) \longrightarrow {\bf D}^+_{\rm QCoh}([X^{\rm reg}_{\check M}/\check M]).
\]
satisfying compatibility with parabolic induction as in Conjecture $\ref{mainconjectureintro}$. Moreover, 
\[R_G((\cind_N^G\psi)_{[T,1]}^I)\cong \Ocal_{[X_{\check G}/\check G]}\]
for any choice of a generic character $\psi:N\rightarrow C^\times$ of the unipotent radical $N$ of a Borel subgroup $B\subset {\rm GL}_n(F)$.
\end{theo}
In the case ${\rm GL}_2(F)$ we can also control the situation for non-regular $(\phi,N)$ and prove fully faithfulness:
\begin{theo}
Let $G={\rm GL}_2(F)$ and $T\subset B\subset G$ denote the standard maximal torus respectively the standard Borel. The functors $R_G$ and $R_T$ defined by $(\ref{EHfunctorintro})$ are fully faithful and there is a natural $\mathfrak{Z}_G$-linear isomorphism 
\[\xi_{B}^{G}:R_{G}\circ \iota_{\overline{B}}^{G}\longrightarrow (R\beta_{\ast}\circ L\alpha^\ast)\circ R_{T},\]
where $\alpha$ and $\beta$ are defined as in Conjecture \ref{mainconjectureintro} {\rm (ii)}.
\end{theo} 

We finally  return to ${\rm GL}_n$ for arbitrary $n$, but restrict to the case of $(\phi,N)$ with $\phi$ regular semi-simple. 
Over the regular semi-simple locus the situation in fact can be controlled very explicitly and we are able to compute examples. 
Given $(\phi,N)\in X_{\check G}(C)$ with regular semi-simple $\phi$ we write $X_{\check G,[\phi,N]}$ for the Zariski-closure of its $\check G$-orbit. 
\begin{theo}
Let $(\phi,N)\in X_{\check G}(C)$ and assume that $\phi$ is regular semi-simple. Then
\[R_G({\rm LL^{mod}}(\phi,N))=\Ocal_{[X_{\check G,[\phi,N]}/\check G]}.\]
\end{theo}
Moreover (still in the ${\rm GL}_n$ case), the conjectured functor $R_G$ should be uniquely determined by the conditions in Conjecture \ref{mainconjectureintro}, see section \ref{remarksaboutthevariousconjectures} for more details. After formal completion we can prove a result in that direction.
For a character $\chi:\mathfrak{Z}_G\rightarrow C$ we write $\hat\Hcal_{G,\chi}$ for the completion of the Iwahori-Hecke algebra $\Hcal_G$ with respect to the kernel of $\chi$. Similarly we can consider the formal completion $\hat X_{\check G,\chi}$ of $X_{\check G}$ with respect to the pre-image of (the closed point of the adjoint quotient defined by) $\chi$ in $X_{\check G}$. Then $(\ref{EHfunctorintro})$ extends to a functor
\[\hat R_{G,\chi}:{\bf D}^+(\hat\Hcal_{G,\chi}\text{-mod})\longrightarrow{\bf D}^+_{\rm QCoh}([\hat X_{\check G,\chi}/\check G]),\]
and similarly for (standard) Levi subgroups $M\subset G$.
\begin{theo}
Let $\phi\in\check G(C)$ be regular semi-simple and let $\chi:\mathfrak{Z}_G\rightarrow C$ denote the character defined by the image of $\phi$ in the adjoint quotient. The set of functors 
\[\hat R_{M,\chi}:\mathbf{D}^{+}(\hat \Hcal_{M,\chi}\text{-}{\rm mod})\longrightarrow \mathbf{D}^{+}_{\rm QCoh}([\hat X_{\check M,\chi}/\check M])\]
for standard Levi subgroups $M\subset G$, is uniquely determined (up to isomorphism) by requiring that they are $\hat{\mathfrak{Z}}_{M,\chi}$-linear, compatible with parabolic induction, and that $\hat R_{T,\chi}$ is induced by the identification 
\[\hat\Hcal_{T,\chi}\text{-}{\rm mod}={\rm QCoh}(\hat X_{\check T,\chi}).\]
\end{theo}

Finally, I would like to mention that I was led to Conjecture \ref{mainconjectureintro} by considerations about $p$-adic automorphic forms and moduli spaces of $p$-adic Galois representations. In fact we hope that the conjecture extends (in a yet rather vague sense) to a $p$-adic picture, which should have implications on the computation of locally algebraic vectors in the $p$-adic Langlands program, as in work of Pyvovarov \cite{Pyvovarov3}, which in fact inspired the computation in section \ref{sectioncind}. 
We do not pursue this direction here, but will come back to this in the future.

{\bf Acknowledgments:} I very much like to thank Christophe Breuil, Michael Rapoport, Timo Richarz, Peter Schneider, Jakob Scholbach, Benjamin Schraen and Matthias Strauch for many helpful discussions and for their interest.
Moreover, I would like to thank Johannes Ansch\"utz and Arthur-C\'esar Le-Bras for asking many helpful questions and pointing out some mistakes in an earlier version of the paper.
 Special thanks go to Peter Scholze for his constant interest and encouragement after I explained him my first computations.
 The author was supported by Germany's Excellence Strategy EXC 2044-390685587 "Mathematics M\"unster: Dynamics--Geometry--Structure" and by the CRC 1442 Geometry: Deformations and Rigidity of the DFG.

\section{Spaces of L-parameters}\label{spacesofLparam}
We fix a field $C$ of characteristic $0$ and a prime $p$ with power $q=p^r$.
Let $G$ be a linear algebraic group over $C$ and let $\lieg$ denote its Lie-Algebra, considered as a $C$-scheme. 
We define the $C$-scheme $X_G$ as the scheme representing the functor
\[R\longmapsto \{(\phi,N)\in (G\times\lieg)(R)\mid {\rm Ad}(\phi)(N)=q^{-1}N \}\]
on the category of $C$-algebras. 

The scheme $X_G$ comes with a canonical $G$-action, by conjugation on $G$ and by the adjoint action on $\lieg$. We write $[X_G/G]$ for the stack quotient of $X_G$ by this action. For obvious reasons this is an algebraic stack (or Artin stack).
Given a homomorphism $\alpha:G\rightarrow H$ of linear algebraic groups, we obtain canonical morphisms $X_G\rightarrow X_H$ of schemes and $[X_G/G]\rightarrow [X_H/H]$ of stacks.

\subsection{Basic properties}\label{sectionbasicprop}
 We study the basic properties of the spaces $X_G$ and $[X_G/G]$. Some of the results in this section were also obtained, in the more general situation of stacks of L-parameters, by Dat-Helm-Kurinczuk-Moss \cite{DHKM} and by Zhu \cite{Zhu}.

\begin{prop}\label{equidimreductive}
(i) Assume that $G$ is reductive. Then $X_{G}$ is a complete intersection inside $G\times \lieg$ and has dimension $\dim G$.\\
(ii) If $G={\rm GL}_n$, then $X_G$ is reduced and the irreducible components are in bijection with the set of $G$-orbits in the nilpotent cone $\mathcal{N}_G\subset \lieg$ of $G$. Moreover let $\eta=(\phi_\eta,N_\eta)\in X_G$ be a generic point of an irreducible component. Then $\phi_\eta$ is regular semi-simple. 
\end{prop}
\begin{proof}
\noindent (ii) This is \cite[Theorem 3.2.]{HartlHellmann} resp.~\cite[Proposition 4.2]{Helm2}. 

\noindent (i) Helm's argument from \cite{Helm2} directly generalizes to the case of a reductive group:
as $X_G\subset G\times \lieg$ is cut out by $\dim G$ equations, it is enough to show that $X_G$ is equi-dimensional of dimension $\dim G$. 

Let us write $f:X_G\rightarrow \lieg$ for the projection to the Lie algebra. We first claim that $f$ set-theoretically factors over the nilpotent cone $\Ncal_G\subset \Lie G$. In order to do so we choose an embedding $G\hookrightarrow \GL_m$ for some $m$. Then $X_G$ embeds into $X_{\GL_m}$ and given $(\phi,N)\in X_G$ \cite[Lemma 2.3]{HartlHellmann} implies that $N$ is mapped to a nilpotent element of $\mathfrak{gl}_m$. This implies $N\in \Ncal_G$.

The scheme $\Ncal_G$ is irreducible and a finite union of (locally closed) $G$-orbits for the adjoint action, as $G$ is reductive. 
Let $Z\subset \Ncal_G$ be such a $G$-orbit and let $\Gcal_Z\subset G\times Z$ be the $Z$-group scheme of centralizers of the points in $Z$, i.e.~the fiber $G_z$ of $\Gcal_Z$ over $z\in Z$ is the centralizer of $z$ in $G$. Then the translation action makes $f^{-1}(Z)$ (if non-empty) into a right $\Gcal_Z$-torsor. In particular in this case we have $\dim f^{-1}(Z)=\dim Z+\dim_Z \Gcal_Z=\dim Z+\dim G_z=\dim G$, where $z\in Z$ is any (closed) point. 
The scheme $X_G$ now is the union of the locally closed subsets $f^{-1}(Z)$, where $Z$ runs over all the $G$-orbits in $\Ncal_G$. As all these locally closed subsets (if $\neq\emptyset$) have dimension $\dim G$, their closures are precisely the irreducible components of $X_G$.  It follows that $X_G$ is equi-dimensional of dimension $\dim G$ as claimed. 
\end{proof}

\begin{rem}\label{remarkequidim}
\noindent (a) The proof implies that the irreducible components of $X_G$ are indexed by a subset of the $G$-orbits in $\Ncal_G$. 
We expect that the conclusion of (ii) holds true for a general reductive group, i.e.~the scheme $X_G$ should be reduced and a complete intersection.  Its irreducible components should be in bijection with the $G$-orbits in the nilpotent cone and at the generic points of the irreducible components the element $\phi$ should be regular semi-simple.
The first part ($X_G$ is reduced and a complete intersection) can be deduced from \cite[Corollary 2.4, Proposition 2.7]{DHKM} (compare also \cite[Proposition 3.1.6]{Zhu}) after noting that $X_G$ is a connected component of the fiber over $C$ of the scheme $\underline{Z}^1(W_F^\circ/P_F,G)$ of loc.~cit..
\\
\noindent (b) The only ingredient in the proof of (i) that uses the assumption that $G$ is reductive is the fact that $G$ acts with only finitely many orbits on its nilpotent cone. More precisely, let $G$ be an arbitrary linear algebraic group and let $G\hookrightarrow {\rm GL}_m$ be a faithful representation. Then the proof of (i) works if $G$ acts with finitely many orbits on $\Lie G\cap \Ncal_{\GL_m}$. This is not true in general, even if $G$ is a parabolic subgroup in $\GL_m$, see \cite{BoosM}: it follows from loc.~cit.~that this fails in the case of a Borel subgroup in $\GL_m$ for $m\geq 6$.
The following example shows that also the statement of the proposition fails for Borel subgroups of ${\rm GL}_n$ for $n\geq 9$. We did not check that this is the optimal bound. It is very likely possible that $X_B$ is not equi-dimensional if $B$ is a Borel subgroup in ${\rm GL}_6$. 
\end{rem}
\begin{expl}\label{explnonequi}
Let $r,d>0$ and $n=rd$. Let $B\subset \GL_n$ be the Borel subgroup of upper triangular matrices and let \[\varphi_0={\rm diag}(1,\dots,1, q,\dots, q,\dots, q^{d-1},\dots, q^{d-1})\in B(C),\]
where each entry $q^i$ appears $r$ times. Then a given element $N=(n_{ij})_{ij}\in {\Lie B}$ satisfies $N\phi_0=q\phi_0 N$ if and only if 
$$n_{ij}=0 \ \text{for}\ j\notin \{ir+1,\dots, (i+1)r\}.$$ 
Scaling $\phi_0$ by multiplication with elements of the center $Z\cong \Gbb_m$, we obtain a closed embedding $\Gbb_m\times \prod_{i=1}^{d-1}\Abb^{r^2}\hookrightarrow X_B$. The $B$-orbit of this closed subscheme is irreducible and of dimension 
\begin{align*}
\dim \big(\Gbb_m\times \prod_{i=1}^{d-1}\Abb^{r^2}\big)+\dim B-{\rm Stab}_B(\phi_0)&=1+\dim B+\big((d-1)r^2-d\tfrac{r(r+1)}{2}\big)\\&=1+\dim B+\tfrac{r}{2}\big(dr-(2r+d)\big).
\end{align*}
In particular we find that $X_B$ has an irreducible component of dimension strictly larger than $\dim B$ if $dr\geq 2r+d$.
On the other hand $X_B$ has always an irreducible component of dimension $\dim B$, namely $B\times\{0\}\subset B\times \Lie B$. 
\end{expl}

Let $G$ be a reductive group and let $P$ be a parabolic subgroup. We will write $\tilde X_P$ for the scheme representing the sheafification of the functor:
\begin{equation}\label{Gequivmodel}
R\mapsto \left\{(\phi,N,g)\in (G\times\lieg)(R)\times G(R)/P(R)\left | 
{\begin{array}{*{20}c}
(\phi,N)\in X_G(R) \ \text{and}\\ \phi\in g^{-1}Pg,\\ N\in {\rm Ad}(g^{-1})(\Lie P) 
\end{array}}
\right\}\right.\end{equation}
This is a closed $G$-invariant subscheme of $X_G\times G/P$ (where $G$ acts on $G/P$ by left translation). Then $(\phi,N)\mapsto (\phi,N,1)$ induces a closed embedding $X_P\hookrightarrow \tilde X_P$ which descends to an isomorphism
\begin{equation}\label{GmodelXP}
[X_P/P]\xrightarrow{\cong} [\tilde X_P/G].
\end{equation}
Moreover, the canonical projection $\tilde X_P \rightarrow X_G$ is $G$-equivariant and the induced morphism $[\tilde X_P/G]\rightarrow [X_G/G]$ agrees under the isomorphism $(\ref{GmodelXP})$ with the morphism $[X_P/P]\rightarrow [X_G/G]$ induced by $P\hookrightarrow G$. The following lemma is a direct consequence of this discussion. 
\begin{lem}\label{XPproper}
Let $G$ be a reductive group and $P\subset G$ be a parabolic subgroup. Then the canonical map $[X_P/P]\rightarrow [X_G/G]$ induced by the inclusion $P\hookrightarrow G$ is proper. 
\end{lem}

We continue to assume that $G$ is reductive. 
We say that a point $(\varphi,N)\in G\times \lieg$ is \emph{regular}, if there are only finitely many Borel subgroups $B'\subset G$ such that $\varphi\in B'$ and $N\in \Lie B'$, i.e.~if (for one fixed choice of a Borel $B$) the point $(\varphi,N)$ has only finitely many pre-images under 
\[\left\{(\phi,N,gB)\in G\times \lieg \times G/B\left | \begin{array}{*{20}c} \phi \in g^{-1}Bg,\\ N\in {\rm Ad}(g^{-1})(\Lie B)\end{array}\right.\right\}\longrightarrow G\times \lieg.\] 
As this morphism is proper and the fiber dimension is upper semi-continuous on the source, the regular elements form a Zariski-open subset $(G\times \lieg)^{\rm reg}\subset G\times \lieg$. 
Similarly we can define a Zariski-open subset $$X_G^{\rm reg}=X_G\cap (G\times\lieg)^{\rm reg}\subset X_G.$$ 
If $P\subset G$ is a parabolic subgroup, we write $$(P\times \Lie P)^{\rm reg}=(G\times \Lie G)^{\rm reg}\cap(P\times \Lie P)$$ and  $X_P^{\rm reg}=X_P\cap X_G^{\rm reg}$. Moreover, we write $\tilde X_P^{\rm reg}$ for the pre-image of $X_G^{\rm reg}$ under $\tilde X_P\rightarrow X_G$. Then $[X_P^{\rm reg}/P]=[\tilde X_P^{\rm reg}/G]$ as stacks and the morphism $\tilde X_P^{\rm reg}\rightarrow X_G^{\rm reg}$ is by construction a finite morphism. 
Moreover, if we write $M$ for the Levi quotient of $P$, it is a direct consequence of the definition that the canonical projection $[X_P/P]\rightarrow [X_M/M]$ restricts to $[X_P^{\rm reg}/P]\rightarrow [X_M^{\rm reg}/M]$.

\begin{lem}\label{equidimparabolicregular}
The scheme $X_P^{\rm reg}$ is equi-dimensional of dimension $\dim P$ and a complete intersection inside $(P\times \Lie P)^{\rm reg}$. Moreover, the map $\tilde X_P^{\rm reg}\rightarrow X_G^{\rm reg}$ is surjective and each irreducible component of $\tilde X_P^{\rm reg}$ dominates an irreducible component of $X_G^{\rm reg}$.
\end{lem}
\begin{proof}
Following the first lines of the proof of Proposition \ref{equidimreductive}, the first claim follows if we show that every irreducible component of $X_P^{\rm reg}$ has dimension at most $\dim P$. Equivalently, we can show that every irreducible component of $\tilde X_P^{\rm reg}$ has dimension at most $\dim G$. This is a direct consequence of the fact that $\tilde X_P^{\rm reg}\rightarrow X_G^{\rm reg}$ is finite. 
It follows that every irreducible component of $\tilde X_P^{\rm reg}$ has dimension equal to $\dim G$. As $\tilde X_P^{\rm reg}\rightarrow X_G^{\rm reg}$ is finite and $X_G$ is equi-dimensional of dimension $\dim G$ it follows that every irreducible component of $\tilde X_P^{\rm reg}$ dominantes an irreducible component of $X_G^{\rm reg}$. 

It remains to show that $\tilde X_P^{\rm reg}\rightarrow X_G^{\rm reg}$ is surjective. In fact we even show that $\tilde X_P\rightarrow X_G$ is surjective:
we easily reduce to the case $P=B$ a Borel subgroup, and, choosing an embedding $G\hookrightarrow \GL_n$, we can reduce to the case of $\GL_n$. There we can check the claim on $k$-valued points for algebraically closed fields $k$, where it easily follows by looking at the Jordan canonical forms of $\phi$ and $N$. 
\end{proof}
\begin{rem}
We remark that $X_P^{\rm reg}\subset X_P$ is open, but not dense in general, as can be deduced from Lemma \ref{equidimparabolicregular} and Example \ref{explnonequi}.
If $G$ is reductive then we expect that $X_G^{\rm reg}$ is dense in $X_G$. In the case of $\GL_n$ this is a consequence of Proposition \ref{equidimreductive}.
\end{rem}

If $G=\GL_n$ and $P\subset G$ is a parabolic subgroup, then $G/P$ can be identified with the variety of flags of type $P$. In particular we can identify $\tilde X_P$ with the variety of triples $(\phi,N,\mathcal{F})$ consisting of $(\phi,N)\in X_G$ and a $(\phi,N)$-stable flag of type $P$. From now on we will often use this identification.

\begin{lem}\label{reducedparaboliccase}
Let $G={\rm GL}_n$ and let $P\subset G$ be a parabolic. Then $\tilde X_P^{\rm reg}$ is reduced.
\end{lem}
\begin{proof}
To prove that $\tilde X_P^{\rm reg}$ is reduced, it remains to show that is generically reduced. 
Let $\xi=(\phi_\xi,N_\xi,\mathcal{F}_\xi)\in \tilde X_P^{\rm reg}$ be a generic point. Under $\tilde\beta_P:\tilde X_P^{\rm reg}\rightarrow X_G$ the point $\xi$ maps to a generic point $\eta=(\phi_\xi,N_\xi)$ of $X_G$ and hence $\phi_\xi$ is regular semi-simple. 
It is enough to show that $\tilde\beta_P^{-1}(\eta)$ is reduced. But as $\phi_\xi$ is regular semi-simple the space of $\phi_\xi$-stable flags is a finite disjoint union of reduced points. Hence its closed subspace of flags that are in addition stable under $N_\xi$ has to be reduced as well.
\end{proof}
\begin{rem}\label{XBfor GLsmalln}
Let $n\leq 5$ and $P\subset \GL_n$ a parabolic subgroup. We point out that the argument in Remark \ref{remarkequidim} (b) implies that $X_P$ is a complete intersection in $P\times \Lie P$. 
But if $n\geq 4$, it is not true that every irreducible component of $\tilde X_P$ dominates an irreducible component of $X_{\GL_n}$. Indeed, one can compute that if $n=4$ and $P=B$ is a Borel, then there is an irreducible component of $\tilde X_B$ on which the Frobenius $\phi$ is semi-simple with eigenvalues $\lambda,q\lambda,q\lambda,q^2\lambda$ for some indeterminate $\lambda$. This component clearly can not dominate an irreducible component of $X_{\GL_4}$.
However, for $n\leq 3$ one can compute that every irreducible component of $X_P$ is the closure of an irreducible component of $X_P^{\rm reg}$. In particular we deduce that $X_P$ is reduced if $n\leq 3$. 
In the general case ($P\subset G$ a parabolic subgroup of a reductive group) we do not know whether $X_P$ is reduced.  
\end{rem}


\begin{lem}\label{fiberproducttorindep}
Let $G$ be reductive and $P\subset G$ be a parabolic with Levi quotient $M$.\\
\noindent (i) The morphism $X_P^{\rm reg}\rightarrow X_M^{\rm reg}$ has finite Tor-dimension.\\
\noindent (ii) Let $P'\subset P$ be a second parabolic subgroup. Let $M'$ denote the Levi quotient of $P'$ and $P'_{M}\subset M$ denote the image of $P'$ in $M$. Then the diagrams
\begin{equation}\label{cartesiansquares}
\begin{aligned}
\begin{xy}
\xymatrix{
X_{P'} \ar[r]\ar[d] & X_P \ar[d]&\text{and}& X_{P'}^{\rm reg} \ar[r] \ar[d]& X_P^{\rm reg}\ar[d]\\
X_{P'_M}\ar[r] & X_M && X_{P'_M}^{\rm reg}\ar[r] & X_M^{\rm reg}
}
\end{xy}\end{aligned}\end{equation}
are cartesian and the fiber product on the right hand side is Tor-independent.
\end{lem}
\begin{proof}
\noindent (i) Let $U\subset P$ denote the unipotent radical of $P$ and fix a section $M\hookrightarrow P$ of the canonical projection. We write $\mathfrak{u}\subset \mathfrak{p}$ for the Lie algebras of $U$ resp.~$P$ and $\mathfrak{m}$ for the Lie algebra of $M$. Then we obtain a commutative diagram 
\[\begin{xy}
\xymatrix{
(M\times U)\times (\mathfrak{m}\times\mathfrak{u}) \ar[rr]^\cong_{\psi} \ar[rrd]&& P\times\mathfrak{p} \ar[d]^{\pi} \\
 && M\times \mathfrak{m},
}
\end{xy}\]
where the horizontal arrow $\psi$ is induced by multiplication and the other two morphisms are the canonical projections. 
Let $r=\dim M$ and $s=\dim U$. 
Let $I\subset\Gamma(P\times\Lie P,\Ocal_{P\times \Lie P})$ be the ideal defining $X_P\hookrightarrow P\times \Lie P$, i.e.~the ideal generated by the entires of ${\rm Ad}(\phi)(N)-q^{-1}N$, where $\phi$ and $N$ are the universal elements over $P$ resp.~$\Lie P$.
Then we deduce from the diagram that $I$ can be generated by elements $f_1,\dots, f_r,g_1\dots,g_s$ such that 
\[f_1,\dots,f_r\in \Gamma(M\times \mathfrak{m},\Ocal_{M\times \mathfrak{m}})\subset \Gamma(P\times \Lie P, \Ocal_{P\times \Lie P}),\]
where $(f_1,\dots,f_r)$ is the ideal defining $X_M\subset M\times \mathfrak{m}$.
It follows that the ideal $(g_1,\dots, g_s)$ is the ideal defining $X_P$ as a closed subscheme of $\pi^{-1}(X_M)\cong \mathbb{A}^{2s}_{X_M}$. 

Let us now write $K(g_1,\dots, g_s)$ for the Koszul complex defined by $g_1,\dots, g_s$ on the open subscheme $\pi^{-1}(X_M)^{\rm reg}=\pi^{-1}(X_M)\cap (P\times\Lie P)^{\rm reg}$ of $\pi^{-1}(X_M)$. This is a finite complex of flat $\Ocal_{X_M}$-modules and we claim that it is a resolution of $\Ocal_{X_P^{\rm reg}}$.  
Indeed, $g_1,\dots,g_s$ cut out the closed subscheme $X_P^{\rm reg}\subset \pi^{-1}(X_M)^{\rm reg}$ which is of codimension $s$ by Lemma \ref{equidimparabolicregular}. As $\pi^{-1}(X_M)^{\rm reg}$ is Cohen-Macaulay (it is an open subscheme of an affine space over $X_M$ and $X_M$ is Cohen-Macaulay as a consequence of Proposition \ref{equidimreductive}) it follows from \cite[Corollaire 16.5.6]{EGA4} that $g_1\dots,g_s$ is a regular sequence and hence the Koszul complex is a resolution of its $0$-th cohomology which is $\Ocal_{X_P^{\rm reg}}$.

\noindent (ii) The fact that the squares are fiber products follows from the fact that $P'$ is the pre-image of $P'_M$ under $P\rightarrow M$. We show that the square on the right is Tor-independent.
As in (i) we have a Koszul complex $K(g_1,\dots, g_s)$ on $\pi^{-1}(X_M)^{\rm reg}$ which is a $\Ocal_{X^{\rm reg}_M}$-flat resolution of $\Ocal_{X_P^{\rm reg}}$. Consider the closed embedding \begin{equation}\label{embeddingXP'}X_{P'}^{\rm reg}\hookrightarrow \pi^{-1}(X_{P'_M})\cap (P\times \mathfrak{p})^{\rm reg}.\end{equation}
As $(\ref{cartesiansquares})$ is cartesian, the restrictions of $g_1,\dots, g_s$ to $\pi^{-1}(X_{P'_M})\cap (P\times \Lie P)^{\rm reg}$ generate the ideal defining the closed embedding $(\ref{embeddingXP'})$ and it remains to show that the pullback of the Koszul complex $K(g_1,\dots, g_s)$ along $(\ref{embeddingXP'})$ is a resolution of its $0$-th cohomology group; that is we need to show that $g_1,\dots, g_s$ is a regular sequence in 
$$\Ocal_{\pi^{-1}(X^{\rm reg}_{P'_M})\cap (P\times \Lie P)^{\rm reg},x}\ \ \text{for all}\ \ x\in X_{P'}^{\rm reg}\subset \pi^{-1}(X^{\rm reg}_{P'_M})\cap (P\times \Lie P)^{\rm reg}.$$
Now $$\pi^{-1}(X_{P'_M}^{\rm reg})\cap (P\times \Lie P)^{\rm reg}\subset \pi^{-1}(X_{P'_M}^{\rm reg})\cong \Abb^{2s}_{X_{P'_M}^{\rm reg}}$$ is an open subscheme and hence it is Cohen-Macaulay as $X_{P'_M}^{\rm reg}$ is Cohen-Macaulay by Lemma \ref{equidimparabolicregular}. 
The claim now follows again from \cite[Corollaire 16.5.6]{EGA4} and the fact that $X_{P'}^{\rm reg}$ is equi-dimensional of dimension $\dim  \pi^{-1}(X_{P'_M}^{\rm reg})-s$. 
\end{proof}
\begin{expl}
Let us point out that the left cartesian diagram of $(\ref{cartesiansquares})$ is not necessarily Tor-independent without restricting to the regular locus. Let us consider $r=d=3$ (so that $dr=2r+d$) in Example \ref{explnonequi}. Let $B\subset \GL_n$ be the Borel subgroup of upper triangular matrices, where $n=rd=9$. Then the above example shows that $X_B$ is not equi-dimensional, and hence the defining ideal is not generated by a regular sequence. Let $P\subset \GL_9$ be a standard parabolic containing $B$ with Levi $M=\GL_5\times \GL_4$. Then the classification of \cite{BoosM} shows that $P$ as well as the Borel $B_M$ of $M$ have the property that they act only via finitely many orbits on $\Lie P\cap \Ncal_{\GL_9}$ resp.~$\Lie B_M\cap \Ncal_{M}$. In particular $X_P$ and $X_{B_M}$ are complete intersections in $P\times \Lie P$ resp.~$B_M\times \Lie B_M$ by Remark \ref{remarkequidim}.
As in the proof above we can construct generators $f_1,\dots, f_{\dim B}$ of the ideal defining $X_B\subset B\times \Lie B$ such that Tor-independence of $(\ref{cartesiansquares})$ is equivalent to exactness of the Koszul complex $K(f_1,\dots, f_{\dim B})$ in negative degrees.  
Let $x\in X_B\subset B\times\Lie B$ be a point that lies on an irreducible component of $X_B$ of dimension strictly larger than $\dim B$. Then (the germs of) $f_1,\dots, f_{\dim B}$ lie in the maximal ideal $\mathfrak{m}_{B\times\Lie B,x}\subset \Ocal_{B\times\Lie B,x}$ and the Koszul complex defined by these elements is not exact, as they do not form a regular sequence (because $\Ocal_{X_B,x}$ is not equi-dimensional of dimension $\dim B$).
\end{expl}
We reformulate the first claim of the Lemma in terms of stacks.
\begin{cor}\label{caresiansquareclassicalstacks}
Let $P'\subset P\subset G$ be parabolic subgroups of $G$ and let $M$ be the Levi quotient of $P$ and $P'_M$ the image of $P'$ in M. Then the diagram
\[
\begin{xy}
\xymatrix{
[X_{P'}/P'] \ar[r]\ar[d] & [X_P/P]\ar[d] \\
[X_{P'_M}/P'_M] \ar[r] & [X_M/M]
}
\end{xy}
\]
of stacks is cartesian.
\end{cor}
\begin{proof}
Note that $P\rightarrow M$ induces an isomorphism $P/P'\cong M/P'_M$. As in $(\ref{Gequivmodel})$ we define a closed $M$-invariant subscheme $Y_{P'_M}\subset X_M\times M/P'_M$ 
as the scheme representing the sheafification of the functor
\[
R\mapsto \left\{(\phi,N,g)\in X_M(R)\times M(R)/P'_M(R)\left | 
{\begin{array}{*{20}c}
 \phi\in g^{-1}P'_Mg\ \text{and}\\ N\in {\rm Ad}(g^{-1})(\Lie P'_M) 
\end{array}}
\right\}\right..\]
Then we have a canonical isomorphism $[X_{P'_M}/P'_M]\cong[Y_{P'_M}/M]$. Similarly we define a $P$-invariant closed subscheme $Y_{P'}\subset X_{P'}\times P/P'$ such that $[X_{P'}/P']\cong[Y_{P'}/P]$. Then the diagram
\[
\begin{xy}
\xymatrix{
Y_{P'} \ar[r]\ar[d] & X_P\ar[d] \\
Y_{P'_M}\ar[r] & X_M
}
\end{xy}
\]
is cartesian by the above lemma. Let $U\subset P$ denote the unipotent radical of $P$, then it follows that
\[
\begin{xy}
\xymatrix{
[Y_{P'}/U] \ar[r]\ar[d] & [X_P/U]\ar[d] \\
Y_{P'_M}\ar[r] & X_M
}
\end{xy}
\]
is cartesian diagram of stacks with $M$-action and with $M$-equivariant morphisms. The claim follows by taking the quotient by the $M$-action everywhere. 
\end{proof}

The objects introduced above have variants in the world of derived (or dg-) schemes (see e.g.~\cite[0.6.8]{DG} and the references cited there). The category of dg-schemes over $C$ canonically contains the category of $C$-schemes as a subcategory.
For any linear algebraic group $G$ we write $\gamma_G:G\times \Lie G\rightarrow \Lie G$ for the morphism $(\phi,N)\mapsto {\rm Ad}(\phi)(N)-q^{-1}N$. We denote by $\mathbf{X}_G$ the fiber product
\[\begin{xy}
\xymatrix{
\mathbf{X}_G\ar[r]\ar[d] &G\times \Lie G \ar[d]^{\gamma_G} \\
\{0\}\ar[r] &\Lie G
}
\end{xy}\]
in the category of dg-schemes. This yields a dg-scheme $\mathbf{X}_G$ with underlying classical scheme ${}^{\rm cl}\mathbf{X}_G=X_G$. 
If $G$ is reductive then Proposition \ref{equidimreductive} implies that $\mathbf{X}_G=X_G$.
Similarly, if $P\subset G$ is a parabolic subgroup of a reductive group, we denote by $\mathbf{X}_P^{\rm reg}\subset \mathbf{X}_P$ the open sub-dg-scheme with underlying topological space $X_P^{\rm reg}$. Then Lemma \ref{equidimparabolicregular} implies that $\mathbf{X}_P^{\rm reg}=X_P^{\rm reg}$ is a classical scheme. 

For any linear algebraic group $G$ we write $[\mathbf{X}_G/G]$ for the stack quotient of $\mathbf{X}_G$ by the canonical action of $G$. 
This is an algebraic dg-stack\footnote{In \cite{DG} dg-stacks are simply called stacks. In order to distinguish between derived and non-derived variants we will always write dg-stack.} in the sense of \cite[1.1]{DG}.
Similarly to the case of schemes we can view any algebraic stack as an algebraic dg-stack. Moreover, recall that every dg-stack $S$ has an underlying classical stack ${}^{\rm cl}S$.

If $G$ is reductive and $P\subset G$ is a parabolic we also consider the stacks $[\mathbf{X}_G^{\rm reg}/G]$ and $[\mathbf{X}_P^{\rm reg}/P]$. Then 
\begin{align*}
[\mathbf{X}_G/G]&=[X_G/G],\\ 
[\mathbf{X}_G^{\rm reg}/G]&=[X_G^{\rm reg}/G]\ \text{and} \\
[\mathbf{X}_P^{\rm reg}/P]&=[X_P^{\rm reg}/P]. 
\end{align*}

We recall that a morphism $\mathbf{Y}_1\rightarrow \mathbf{Y}_2$ of dg-stacks is \emph{schematic} if for all affine dg-schemes $\mathbf{Z}$ and all morphisms $\mathbf{Z}\rightarrow \mathbf{Y}_2$ the fiber product $\mathbf{Z}\times_{\mathbf{Y}_2}\mathbf{Y}_1$ is a dg-scheme, see \cite[1.1.2]{DG}. A morphism of dg-schemes is called proper is the induced morphism of the underlying classical schemes is proper, and a morphism of algebraic dg-stacks is proper if the morphism of underlying classical stacks is proper in the sense of \cite[Definition 7.11]{LM}.
Similarly to the non-derived results above we obtain the following lemma. 
\begin{lem}
Let $G$ be a reductive group and let $P\subset G$ be a parabolic subgroup with Levi quotient $M$. \\
(i) The morphism $[\mathbf{X}_P/P]\rightarrow [\mathbf{X}_G/G]$ is schematic and proper. \\
(ii) Let $P'\subset P$ be a second parabolic subgroup and write $P'_M\subset M$ for the image of $P'$ in $M$. Then 
$$\mathbf{X}_{P'}\cong \mathbf{X}_{P'_M}\times_{\mathbf{X}_M}\mathbf{X}_P\ \ \text{and}\ \  [\mathbf{X}_{P'}/P']\cong [\mathbf{X}_{P'_M}/P'_M]\times_{[\mathbf{X}_M/M]}[\mathbf{X}_P/P].$$
\end{lem}
\begin{proof}
(i) As in the non derived case we can rewrite $[\mathbf{X}_P/P]$ as $[\tilde{\mathbf{X}}_P/G]$ where $\tilde{\mathbf{X}}_P\subset \mathbf{X}_G\times G/P$ is the closed $G$-invariant sub-dg-scheme obtained by making $$\mathbf{X}_P\subset \mathbf{X}_G\times \{P\}\subset \mathbf{X}_G\times G/P$$ invariant under the $G$-action. 
More precisely we can describe this dg-scheme as follows. Let us write  $Z\subset G\times\Lie G\times G/P$ for scheme representing (the sheafification of) the functor
\[R\longmapsto\{(\phi,N,gP)\in (G\times\Lie G\times G/P)(R)\mid \phi\in g^{-1}Pg, \ N\in{\rm Ad}(g^{-1})\Lie P\},\]
on $C$-algebras. By definition there is an isomorphism of stacks $$[Z/G]\cong [P\times\Lie P/P].$$ 
Similarly we write $Z'\subset \Lie G\times G/P$ for the scheme parametrizing pairs $(N,gP)$ such that $N\in {\rm Ad}(g^{-1})\Lie P$. Then $Z'\rightarrow G/P$ is a vector bundle (with fibers isomorphic to $\Lie P$) and we can define $\tilde{\mathbf{X}}_P$ as the derived fiber product
\[\begin{xy}
\xymatrix{
\tilde{\mathbf{X}}_P \ar[r]\ar[d] & Z\ar[d]\\
G/P \ar[r]& Z',}
\end{xy}
\]
where $G/P\rightarrow Z'$ is the zero section and the right vertical arrow is given by $(\phi,N,gP)\mapsto ({\rm Ad}(\phi)N-q^{-1}N,gP)$
Now $[\mathbf{X}_P/P]=[\tilde{\mathbf{X}}_P/G]\rightarrow [\mathbf{X}_G/G]$, and the projection
\begin{equation}\label{GmodelforbetaP}
\tilde{\mathbf{X}}_P\longrightarrow X_G
\end{equation} 
is a $G$-equivariant model for the canonical projection $[\mathbf{X}_P/P]\rightarrow [X_G/G]$ induced by $P\subset G$. 
The claim follows from this together with the observation that ${}^{\rm cl}\tilde{\mathbf{X}}_P=\tilde X_P$ and ${}^{\rm cl}[\mathbf{X}_P/P]=[X_P/P]$. \\
(ii) This is a direct consequence of the definition of the fiber product in the category of dg-schemes and the fact that $P'$ is the pre-image of $P'_M$ under the (flat) morphism $P\rightarrow M$.
\end{proof}
\begin{rem}\label{GmodelforbetaB2}
Let $B\subset G$ be a Borel subgroup and let $U\subset B$ denote the unipotent radical. Then $\mathbf{X}_B\subset B\times \Lie B$ is in fact a (derived) subscheme of $B\times\Lie U$. Indeed, computing ${\rm Ad}(\phi)N-q^{-1}N$ for the universal pair $(\phi,N)$ over $B\times\Lie B$ and restricting to the Lie algebra of (a choice of) a maximal torus in $B$, we find that $\mathbf{X}_B\subset B\times\Lie U$. 
Using this observation we mention the following variant of the construction in the proof of (i) above. Let $\mathcal{N}_G$ denote the nilpotent cone of $G$, we can define $Y=Z\cap G\times\mathcal{N}_G\times G/B$. 
In fact $[Y/G]\cong [B\times\Lie U/B]$ and if we write $\tilde{G}\rightarrow G$ (respectively $\tilde{\mathcal{N}}_G\rightarrow \mathcal{N}_G$) for the Grothendieck (respectively Springer) resolution, then
$$Y=\tilde G\times_{G/B}\tilde{\mathcal{N}}_G.$$
In particular $Y$ is smooth and affine of relative dimension $\dim G$ over $G/B$. Then we can define  $\tilde{\mathbf{X}}_B$ as a closed sub dg-scheme of $Y$ in a similar fashion as it was defined as a closed sub dg-scheme of $Z$ in the proposition. 
\end{rem}
\subsection{Derived categories of (quasi-)coherent sheaves}

Given a scheme or a stack $X$ (or a derived scheme or a derived stack) we write ${\bf D}_{\rm QCoh}(X)$ for the derived category of quasi coherent sheaves on $X$, see \cite[1.2]{DG} and denoted by ${\rm QCoh}(X)$ in loc.~cit.
We write ${\bf D}^b_{\rm Coh}(X)$ for the full subcategory of objects that only have cohomology in finitely many degrees which moreover is coherent. 

If $X$ is a noetherian scheme then ${\bf D}^b_{\rm Coh}(X)$ coincides with the full subcategory of the derived category ${\bf D}(\mathcal{O}_X\text{-mod})$ of $\mathcal{O}_X$-modules, consisting of those complexes that have coherent cohomology and whose cohomology is concentrated in finite degrees. 
Similarly, if $X$ is a (classical) algebraic stack, then ${\bf D}^b_{\rm Coh}(X)$, and more generally ${\bf D}^+_{\rm QCoh}(X)$, agrees\footnote{This is true after passing to the underlying homotopy category. The derived categories in \cite{DG} are by definition $\infty$-categories while the derived categories in \cite{LM} are classical triangulated categories.} with the definition of the bounded derived category of coherent sheaves, respectively with the bounded below derived category of quasi-coherent sheaves as defined in \cite{LM}.

\begin{lem}\label{preservesboundedbelow}
Let $G$ be a reductive group and let $P$ be a parabolic subgroup with Levi quotient $M$, and let 
\[\alpha:[\mathbf{X}_P/P]\rightarrow [\mathbf{X}_M/M]\ \text{and}\ \beta:[\mathbf{X}_P/P]\rightarrow [\mathbf{X}_G/G]\]
denote the canonical morphisms. Then the maps
\begin{align*}
L\alpha^\ast:&{\bf D}_{\rm QCoh}([\mathbf{X}_M/M])\longrightarrow {\bf D}_{\rm QCoh}([\mathbf{X}_P/P])\\
R\beta_\ast:&{\bf D}_{\rm QCoh}([\mathbf{X}_P/P])\longrightarrow{\bf D}_{\rm QCoh}([\mathbf{X}_G/G])
\end{align*}
preserve the subcategories ${\bf D}^+_{\rm QCoh}(-)$ and ${\bf D}^b_{\rm Coh}(-)$.
\end{lem}
\begin{proof}
In the case of $R\beta_\ast$ the claim directly follows from the fact that $\beta$ is proper and schematic.
We prove the claim for $L\alpha^\ast$. As the properties of belonging to ${\bf D}^+_{\rm QCoh}(-)$ or ${\bf D}^b_{\rm Coh}(-)$ can be checked over the smooth cover $\mathbf{X}_P$ of $[\mathbf{X}_P/P]$ it is enough to show that pullback along the morphism
\[\alpha':\mathbf{X}_P\longrightarrow \mathbf{X}_M\]
preserves ${\bf D}^+_{\rm QCoh}(-)$ and ${\bf D}^b_{\rm Coh}(-)$. 
Moreover, both properties may be checked after forgetting the $\mathcal{O}_{\mathbf{X}_P}$-module structure, and only remembering the $\mathcal{O}_{P\times\Lie P}$-modules structure. As in the proof of Lemma \ref{fiberproducttorindep} we find that $\mathcal{O}_{\mathbf{X}_P}$ can be represented by a finite complex $\mathcal{F}_P^\bullet$ of flat $\mathcal{O}_{\mathbf{X}_M}$-modules that are coherent as $\mathcal{O}_{P\times \Lie P}$-modules, and $L\alpha'^\ast$ is identified with the functor $-\otimes^{L}_{\mathcal{O}_{\mathbf{X}_M}}\mathcal{F}_P^\bullet$.  The claim follows from this. 
\end{proof}
\begin{rem}
\noindent (a) We point out that using the definition of the derived categories as in \cite{DG} has the advantage that there is a canonical pullback functor between the derived categories of quasi-coherent sheaves on stacks. 
At least as long as we only consider non-derived stacks (as e.g.~$[X_G/G]$ or $[X_G^{\rm reg}/G]$) there is a definition of the derived category of quasi-coherent sheaves in \cite{LM}. However, the definition of the pullback functor in loc.~cit.~meets some problems. Lemma \ref{preservesboundedbelow} essentially tells us that we could as well use the definition of \cite{LM} and only consider complexes that are bounded below.

\noindent (b) The explicit description of the pullback in the proof of Lemma \ref{preservesboundedbelow} could also be used to completely bypass the use of derived schemes, or derived stacks. 
In the end we will be interested in the composition $R\beta_\ast\circ L\alpha^\ast$ rather than in the individual functors. Hence instead of $L\alpha^\ast$ we might use the construction $-\otimes^L_{\mathcal{O}_{X_M}}\mathcal{F}^\bullet_P$ and carefully define the $\mathcal{O}_{[X_G/G]}$-action after push-forward (and after descent to the stack quotient).
However, it seems to be more natural to use derived stacks than such an explicit workaround. 
\end{rem}

Let $P_1\subset P_2$ be parabolic subgroups of a reductive group $G$ with Levi-quotients $M_i$, $i=1,2$. We write $P_{12}\subset M_2$ for the image of $P_1$ in $M_2$. Then $P_{12}\subset M_2$ is a parabolic subgroup with Levi quotient $M_1$. We obtain a diagram
\begin{equation}\label{compisitiondiagram}
\begin{aligned}
\begin{xy}
\xymatrix{
[\mathbf{X}_{P_{1}}/P_{1}] \ar[r]_{\beta}\ar[d]^{\alpha} \ar@/_ 1cm/[dd]_{\alpha_{1}} \ar@/^ .5cm/[rr]^{\beta_{1}}& [\mathbf{X}_{P_{2}}/P_{2}] \ar[r]_{\beta_{2}}\ar[d]^{\alpha_{2}} & [\mathbf{X}_{G}/G]\\
[\mathbf{X}_{P_{12}}/P_{12}] \ar[r]^{\beta_{12}}\ar[d]^{\alpha_{12}} & [\mathbf{X}_{M_2}/M_2] \\
[\mathbf{X}_{M_1}/M_1]
}
\end{xy}
\end{aligned}
\end{equation}
where the upper left diagram is cartesian (in the category of dg-stacks). 
\begin{lem}
In the above situation we have a natural isomorphism
\begin{equation}\label{compoinductionGalois}
(R\beta_{2\,\ast}\circ L\alpha_{2}^\ast) \circ (R\beta_{12\,\ast}\circ L\alpha_{12}^\ast)\xrightarrow\cong R\beta_{1\,\ast}\circ L\alpha_{1}^\ast
\end{equation}
of functors $${\bf D}_{\rm QCoh}([\mathbf{X}_{M_1}/M_1])\longrightarrow {\bf D}_{\rm QCoh}([\mathbf{X}_{G}/G]).$$  
\end{lem}
\begin{proof}
The upper right square is (derived) cartesian and $\beta_{12}$ is schematic and proper, in particular it is quasi-compact and quasi-separated. Hence \cite[Proposition 1.3.6, 1.3.10]{DG} implies that the natural base change morphism $$L\alpha_{2}^\ast \circ R\beta_{12\,\ast}\longrightarrow R\beta_\ast \circ L\alpha^\ast$$ is an isomorphism. We obtain the natural isomorphism
\begin{align*}
(R\beta_{2\,\ast}\circ L\alpha_{2}^\ast) \circ (R\beta_{12\,\ast}\circ L\alpha_{12}^\ast)=R\beta_{2\,\ast}\circ (L\alpha_{2}^\ast \circ R\beta_{12\,\ast})\circ L\alpha_{12}^\ast\\
\xrightarrow\cong R\beta_{2\,\ast}\circ (R\beta_\ast \circ L\alpha^\ast)\circ L\alpha_{12}^\ast =R\beta_{1\,\ast}\circ L\alpha_{1}^\ast.
\end{align*}
\end{proof}
We point out that working only with classical schemes, we still obtain a natural transformation between the corresponding functors: if we consider the underlying classical stacks in the diagram $(\ref{compisitiondiagram})$ and keep the same notations for the morphisms by abuse of notation, then there still is a natural base change morphism, but it is not necessarily an isomorphism as the fiber product might not be Tor-independent. However, it becomes an isomorphism when we restrict the functors to the regular locus, i.e.~we consider them as functors $${\bf D}_{\rm QCoh}([X_{M_1}^{\rm reg}/M_1])\rightarrow {\bf D}_{\rm QCoh}([X_{G}^{\rm reg}/G]),$$ as in this case the classical and the derived picture coincide. Indeed, after the restriction to the regular locus the derived fiber product equals the classical fiber product by the Tor-independence in Lemma \ref{fiberproducttorindep}.

\subsection{Duals}

We continue to assume that $G$ is a reductive group over $C$ and fix a Borel subgroup $B\subset G$. Recall that we have defined a morphism
\[\beta_B:[\mathbf{X}_B/B]\longrightarrow [X_G/G]\]
of (derived) stacks. In this subsection we will analyze the (derived) direct image of the structure sheaf under this morphism that will play an important role\footnote{This complex of sheaves is called the \emph{coherent Springer sheaf} in the work of Ben-Zvi--Chen--Helm--Nadler \cite{BZCHN} and of Zhu \cite{Zhu}.} later on. 
For a (derived) scheme or a (derived) stack $X$ we will write $\omega_X\in {\bf D}_{\rm QCoh}(X)$ for the dualizing complex and $\mathbb{D}_X(-)=R\mathcal{H}om_{\mathcal{O}_X}(-,\omega_X)$ for the usual duality functor.
The main point of this section is to give evidence to the following conjecture (compare also \cite[Conjecture 3.35]{BZCHN} and \cite[Conjecture 4.5.1, Remark 4.5.4]{Zhu}).
\begin{conj}\label{conjdegzero}
The derived direct image $R\beta_\ast(\mathcal{O}_{[\mathbf{X}_B/B]})\in {\bf D}_{\rm Coh}^b([X_G/G])$ is concentrated in degree zero and there is an isomorphism
$$R\beta_\ast(\mathcal{O}_{[\mathbf{X}_B/B]})\cong \mathbb{D}_{[X_G/G]}(R\beta_\ast(\mathcal{O}_{[\mathbf{X}_B/B]})).$$
In particular the pullback of $R\beta_\ast(\mathcal{O}_{[\mathbf{X}_B/B]})$ to $X_G$ is concentrated in degree zero and is a maximal Cohen-Macaulay module. 
\end{conj}

Let us write $\tilde\beta_B:\tilde{\mathbf{X}}_B\rightarrow X_G$ for the $G$-equivariant model of $\beta_B$, as in $(\ref{GmodelforbetaP})$. The conjecture is obviously equivalent to the claim that $R\tilde\beta_{B,\ast}\mathcal{O}_{\tilde{\mathbf{X}}_B}$ is concentrated in degree zero and that there is a $G$-equivariant isomorphism $$R\tilde\beta_{B,\ast}\mathcal{O}_{\tilde{\mathbf{X}}_B}\cong \mathbb{D}_{X_G}(R\tilde\beta_{B,\ast}\mathcal{O}_{\tilde{\mathbf{X}}_B})[-\dim G].$$
This isomorphism then implies that the dual of the sheaf $R\tilde\beta_{B,\ast}\mathcal{O}_{\tilde{\mathbf{X}}_B}$ is concentrated in a single degree and hence, by \cite[Tag 0B5A]{stacksproject}, it is a Cohen-Macaulay module. 
As its support obviously is all of $X_G$ it is then a maximal Cohen-Macaulay module.

As in the Remark \ref{GmodelforbetaB2} we write $Y\rightarrow G\times \Lie G$ for the scheme parametrizing triples $(\phi,N,gB)\in G\times\mathcal{N}_G\times G/B$ such that $\phi\in g^{-1}Bg$ and $N\in {\rm Ad}(g^{-1})\Lie U$, where $U\subset B$ is the unipotent radical.
Recall that $Y$ is smooth (of dimension $\dim G+\dim G/B$) as the projection to the flag variety $\pr:Y\rightarrow G/B$ is the composition of a $B$-bundle and a geometric vector bundle (or rank $\dim U$). 
We consider the diagram.
\[
\begin{xy}
\xymatrix{
\tilde{\mathbf{X}}_B \ar[d]_{\tilde\beta} \ar[r]^{\iota_B} & Y \ar[d]^f  \\
X_G\ar[r]^{\iota_G} & G\times\Lie G.
}
\end{xy}
\]
and write $\mathcal{F}^\bullet=R\iota_{B,\ast}\mathcal{O}_{\tilde{\mathbf{X}}_B}$. The main observation is that $\mathcal{F}^\bullet$ is represented by a Koszul complex and hence is selfdual (up to a shift). 
In order to make this precise, recall that an algebraic representation of $B$ defines a $G$-equivariant vector bundle on $G/B$. We write $\mathcal{U}^\vee$ for the $G$-equivariant vector bundle on $G/B$ defined by the canonical $B$ representation on $\mathfrak{u}^\vee$. Here $\mathfrak{u}$ denotes the Lie algebra of $U$ (considered as a $C$-vector space), and $\mathfrak{u}^\vee$ denotes its dual. 
In particular $\mathcal{U}^\vee$ admits a filtration whose graded pieces are the lines bundles $\mathcal{L}_\alpha$ on $G/B$ associated to the \emph{negative} (with respect to $B$) roots $\alpha$ of $G$. 
\begin{lem}
\noindent (i) The complex $\mathcal{F}^\bullet$ is represented by a Koszul complex
\[ \dots\longrightarrow \bigwedge\nolimits^i \pr^\ast(\mathcal{U}^\vee) \longrightarrow \bigwedge\nolimits^{i-1}\pr^\ast(\mathcal{U}^\vee)\longrightarrow\dots \]
where the term $\bigwedge\nolimits^i \pr^\ast(\mathcal{U}^\vee)$ sits in (cohomological) degree $-i$. \\
\noindent (ii) There is a $G$-equivariant isomorphism 
\[\mathcal{F}^\bullet\cong \mathbb{D}_{Y}(\mathcal{F}^\bullet)[-\dim G].\]
\end{lem}
\begin{proof}
(i) By definition the structure sheaf of the derived scheme $\mathbf{X}_B\subset B\times \Lie U$ is (as a complex of $\mathcal{O}_{B\times \Lie U}$-modules) quasi-isomorphic to the $B$-equivariant Koszul complex
\[ \dots\longrightarrow \bigwedge\nolimits^i \mathfrak{u}^\vee\otimes_C \mathcal{O}_{B\times \Lie U} \longrightarrow \bigwedge\nolimits^{i-1}\mathfrak{u}^\vee\otimes \mathcal{O}_{B\times \Lie U}\longrightarrow\dots \]
defined by the entries of ${\rm Ad}(\phi)N-q^{-1}N$ for the universal pair $(\phi,N)$ over $B\times \Lie U$ (where, by abuse of notation, we also write $\mathfrak{u}$ for the $C$-vector space underlying the scheme $\Lie U$). 
The claim follows from the identification $[Y/G]=[B\times \Lie U]$, compare also Remark \ref{GmodelforbetaB2}. \\
\noindent (ii) The claim on duality follows from the usual self-duality of Koszul complexes which is induced by the perfect pairing
\[ \bigwedge\nolimits^i \pr^\ast(\mathcal{U}^\vee)\times  \bigwedge\nolimits^{\dim G/B-i} \pr^\ast(\mathcal{U}^\vee)\longrightarrow \bigwedge\nolimits^{\dim G/B} \pr^\ast(\mathcal{U}^\vee)\]
and the identification
\[\bigwedge\nolimits^{\dim G/B} \pr^\ast{\mathcal{U}^\vee}=\pr^\ast\big(\bigwedge\nolimits^{\dim G/B}\mathcal{U}^\vee\big)=\pr^\ast \omega_{G/B}[-\dim G/B]=\omega_Y[-\dim Y].\]
\end{proof}
\begin{conj}\label{conjdeg<=0}
The complex $Rf_\ast \mathcal{F}^\bullet$ is concentrated in non-negative degrees.
\end{conj}

\begin{prop}
Conjecture \ref{conjdeg<=0} implies Conjecture \ref{conjdegzero}.
\end{prop}
\begin{proof}
The morphisms $\iota_G$ and $f$ satisfy Grothendieck duality, as $\iota_G$ is a complete intersection and $f$ factors into a complete intersection and a smooth morphism. 
Hence we have the following isomorphisms in ${\bf D}_{\rm QCoh}(G\times\lieg)$:
\begin{align*}
R\iota_{G,\ast}(R\tilde{\beta}_{B,\ast}(\mathcal{O}_{\tilde{\mathbf{X}}_B}))=Rf_\ast\mathcal{F}^\bullet &\cong Rf_\ast(\mathbb{D}_{Y}(\mathcal{F}^\bullet)[-\dim G])=Rf_\ast(\mathbb{D}_{Y}(\mathcal{F}^\bullet))[-\dim G]\\
&\cong \mathbb{D}_{G\times\lieg}(Rf_\ast\mathcal{F}^\bullet)[-\dim G]\\&=\mathbb{D}_{G\times\lieg}(R\iota_{G,\ast}R\tilde{\beta}_{B,\ast}(\mathcal{O}_{\tilde{\mathbf{X}}_B}))[-\dim G]\\
&\cong R\iota_{G,\ast}\big(\mathbb{D}_{X_G}(R\tilde{\beta}_{B,\ast}(\mathcal{O}_{\tilde{\mathbf{X}}_B}))[-\dim G]\big).
\end{align*}
Assuming Conjecture \ref{conjdeg<=0} the complex $Rf_\ast\mathcal{F}^\bullet$ is concentrated in degrees $(-\infty,0]$, and hence so is $R\tilde{\beta}_{B,\ast}(\mathcal{O}_{\tilde{\mathbf{X}}_B})$, as $\iota_{G}$ is affine. 
It follows that $\mathbb{D}_{X_G}(R\tilde{\beta}_{B,\ast}(\mathcal{O}_{\tilde{\mathbf{X}}_B}))$ is concentrated in degrees $[-\dim G,+\infty)$ and hence $$R\iota_{G,\ast}\big(\mathbb{D}_{X_G}(R\tilde{\beta}_\ast(\mathcal{O}_{\tilde{\mathbf{X}}_B}))[-\dim G]\big)$$ is concentrated in degrees $[0,+\infty)$. 
Hence $Rf_\ast\mathcal{F}^\bullet$ is concentrated in degree zero and, again using that $\iota_G$ is affine, the same holds true for $R\tilde{\beta}_{B,\ast}(\mathcal{O}_{\tilde{\mathbf{X}}_B})$.

But as $R\tilde{\beta}_{B,\ast}(\mathcal{O}_{\tilde{\mathbf{X}}_B})$ is concentrated in a single degree the isomorphism 
\[R\iota_{G,\ast}(R\tilde{\beta}_{B,\ast}(\mathcal{O}_{\tilde{\mathbf{X}}_B}))\cong R\iota_{G,\ast}\big(\mathbb{D}_{X_G}(R\tilde{\beta}_{B,\ast}(\mathcal{O}_{\tilde{\mathbf{X}}_B}))[-\dim G]\big)\]
is in fact a $G$-equivariant isomorphism of sheaves on $G\times\lieg$ (not just an isomorphism in the derived category) and hence restricts to a $G$-invariant isomorphism of sheaves
\[R\tilde{\beta}_{B,\ast}(\mathcal{O}_{\tilde{\mathbf{X}}_B})\cong \mathbb{D}_{X_G}(R\tilde{\beta}_{B,\ast}(\mathcal{O}_{\tilde{\mathbf{X}}_B}))[-\dim G].\]
on $X_G$. As we have seen above these claims are equivalent to Conjecture \ref{conjdegzero}. 
\end{proof}

Unfortunately we are not able to prove Conjecture \ref{conjdeg<=0} in general. Instead we will show that this conjecture would follow from a combinatorial statement about roots. As $G\times \lieg$ is affine it is enough to prove that $H^i(G\times \lieg,Rf_\ast\mathcal{F}^\bullet)$ vanishes for $i>0$. This cohomology is computed by a spectral sequence
\[E_1^{i,j}=H^j(Y,\mathcal{F}^i)\Longrightarrow H^{i+j}(G\times\lieg,Rf_\ast\mathcal{F}^\bullet)\]
and hence it is enough to show that 
$H^j(Y,\bigwedge\nolimits^i\pr^\ast(\mathcal{U}^\vee))=0\ \text{for}\ j>i$ (note that $\mathcal{F}^i=\bigwedge^{-i}\pr^\ast(\mathcal{U}^\vee)$).
We can compute this cohomology group after push-forward along $Y\hookrightarrow G\times \mathcal{N}_G\times G/B$ as follows.
Let us write $\bar{\mathcal{U}}$ respectively $\bar{\mathcal{B}}$ for $G$-equivariant vector bundles on $G/B$ associated to the $B$-representations on $\mathfrak{g}/\mathfrak{b}$ respectively on $\mathfrak{g}/\mathfrak{u}$, where $\mathfrak{b}=\Lie B$. 
Note that $\mathcal{O}_Y$ is, as sheaf on $G\times \mathcal{N}_G\times G/B$ quasi-isomorphic to the Koszul complex
\[\dots\longrightarrow \bigwedge\nolimits^i g^\ast(\bar{\mathcal{U}}^\vee\otimes\bar{\mathcal{B}}^\vee)\longrightarrow \bigwedge\nolimits^{i-1} g^\ast(\bar{\mathcal{U}}^\vee\otimes\bar{\mathcal{B}}^\vee)\longrightarrow \dots \]
where again the $\bigwedge\nolimits^i$ term appears in degree $-i$, and where $g:G\times\mathcal{N}_G\times G/B\rightarrow G/B$ denotes the canonical projection. 
By the projection formula (using that $g$ is affine and flat base change) we find that
\begin{align*}
H^j(Y,\mathcal{F}^{-i})&=H^j(G\times\mathcal{N}_G\times G/B,\bigwedge\nolimits^ig^\ast(\mathcal{U})^\vee\otimes \mathcal{O}_Y)\\
&=H^j(G\times \mathcal{N}_G\times G/B, \bigwedge\nolimits^ig^\ast(\mathcal{U})^\vee\otimes \bigwedge\nolimits^\bullet  g^\ast(\bar{\mathcal{U}}^\vee\otimes\bar{\mathcal{B}}^\vee))\\&=H^j(G/B,\bigwedge\nolimits^i\mathcal{U}^\vee\otimes \bigwedge\nolimits^\bullet  (\bar{\mathcal{U}}^\vee\otimes\bar{\mathcal{B}}^\vee))\otimes_C \Gamma(G\times \mathcal{N}_G,\mathcal{O}_{G\times\mathcal{N}})
\end{align*}
and again we can use a spectral sequence to compute this cohomology group. Hence we have to show that
\[H^\bullet(G/B,\bigwedge\nolimits^a\mathcal{U}^\vee\otimes\bigwedge\nolimits^b \bar{\mathcal{U}}^\vee\otimes \bigwedge\nolimits^c\bar{\mathcal{B}}^\vee)=0\ \text{for}\ \bullet>a+b+c.\]
As $\bar{\mathcal{B}}^\vee$ is an extension of $\bar{\mathcal{U}}^\vee$ and a power of the trivial line bundle it is enough to prove that 
\[H^\bullet(G/B,\bigwedge\nolimits^a\mathcal{U}^\vee\otimes\bigwedge\nolimits^b \bar{\mathcal{U}}^\vee\otimes \bigwedge\nolimits^c\bar{\mathcal{U}}^\vee)=0\ \text{for}\ \bullet>a+b+c.\]
We compute this cohomology group in terms of the cohomology of equivariant line bundles: note that the $G$-equivariant vector bundle $\bigwedge\nolimits^a\mathcal{U}^\vee\otimes\bigwedge\nolimits^b \bar{\mathcal{U}}^\vee\otimes \bigwedge\nolimits^c\bar{\mathcal{U}}^\vee$ has a filtration whose subquotients are $G$-equivariant lines bundles $\mathcal{L}_{\lambda}\otimes\mathcal{L}_\mu\otimes\mathcal{L}_\nu=\mathcal{L}_{\lambda+\mu+\nu}$, where $\lambda$ is a sum of $a$ pairwise distinct negative roots, $\mu$ is a sum of $b$ pairwise distinct positive roots and $\nu$ is a sum of $c$ pairwise distinct positive roots. It hence suffices to show that $H^\bullet(G/B,\mathcal{L}_{\lambda+\mu+\nu})$ vanishes in degrees larger than $a+b+c$ in this case. 
By the Borel-Bott-Weil theorem the cohomology of these line bundles is computed as follows: we write $w\cdot\kappa=w(\kappa+\rho)-\rho$ for the dot action of the Weyl group $W$ on $X^\ast(T)$, where $\rho$ is the half sum of the positive roots. 
Then there is either an element $1\neq w\in W$ such that $w\cdot\kappa=\kappa$, or there is a unique $w\in W$ such that $w\cdot \kappa$ is dominant. In the first case $H^\bullet (G/B,\mathcal{L}_{\kappa})=0$ and in the second case $H^\bullet(G/B,\mathcal{L}_{\kappa})$ is concentrated in degree $\ell(w)$, where, as usual, $\ell(w)$ denotes the length of the Weyl group element $w$.
We hence obtain a combinatorial claim about roots which allows us to check Conjecture \ref{conjdeg<=0}. At least for ${\rm GL}_2$ and ${\rm GL}_3$ we can explicitly check this claim about roots. 
\begin{prop}
Conjecture \ref{conjdegzero} holds true for ${\rm GL}_2$ and ${\rm GL}_3$.
\end{prop}
\begin{proof}
The above discussion implies that it is enough to prove the following claim: 
Let $a,b,c\geq 0$ and let $\lambda$ be a sum of $a$ pairwise distinct negative roots, $\mu$ be a sum of $b$ pairwise distinct positive roots and $\nu$ be a sum of $c$ pairwise distinct positive roots. If there is some (necessarily unique) $w\in W$ such that $w\cdot(\lambda+\mu+\nu)$ is dominant, then $\ell(w)\leq a+b+c$.

In the ${\rm GL}_2$ case this computation is rather trivial. In the ${\rm GL}_3$ case $G/B$ has dimension $3$ and the claim is trivially satisfied for $a+b+c\geq 3$. Moreover, if $a+b+c=0$, then $\mathcal{L}_{\lambda+\mu+\nu}=\mathcal{O}_{G/B}$ and again the claim is trivial. We are left to check the cases $a+b+c=1$ and $a+b+c=2$. We write $\alpha$ and $\beta$ for the two simple positive roots, and $s_\alpha,s_\beta\in W$ for the corresponding reflections.

If $a+b+c=1$, then $\mathcal{L}_{\lambda+\mu+\nu}\in\{\mathcal{L}_\alpha,\mathcal{L}_\beta, \mathcal{L}_{\alpha+\beta}, \mathcal{L}_{-\alpha},\mathcal{L}_{-\beta},\mathcal{L}_{-\alpha-\beta}\}$. 
The root $\alpha$ is fixed, under the dot action, by $s_\beta$,  the root $\beta$ is fixed by $s_\alpha$ and $\alpha+\beta$ is dominant. Hence $\mathcal{L}_\alpha,\mathcal{L}_\beta$ and $\mathcal{L}_{\alpha+\beta}$ have no higher cohomology. 
On the other hand $s_\alpha\cdot(-\alpha)=0$ and $s_\beta\cdot(-\beta)=0$ are dominant and hence $\mathcal{L}_{-\alpha}$ and $\mathcal{L}_{-\beta}$ have cohomology in degree $1=\ell(s_\alpha)=\ell(s_\beta)$. Moreover, $-\alpha-\beta=-\rho$ is the fix point for the dot action and hence $\mathcal{L}_{-\alpha-\beta}$ has no cohomology. 

Now assume $a+b+c=2$. We list the possible weights $\lambda+\mu+\nu$ in this case:  
\begin{enumerate}
\item[-] if $a=b=0, c=2$ (or if $a=c=0, b=2$), we have to check the weights
$\alpha+\beta,2\alpha+\beta, \alpha+2\beta$.
\item[-] if $a=2, b=c=0$, we have to check the weights $-\alpha-\beta,-2\alpha-\beta, -\alpha-2\beta$.
\item[-] if $a=0, b=c=1$, we have to check the weights $2\alpha, 2\beta, \alpha+\beta,2\alpha+\beta, \alpha+2\beta,2\alpha+2\beta$.
\item[-] if $a=b=1, c=0$ (or if $a=c=1,b=0$) we have to check the weights $0,\alpha-\beta,-\beta,\beta-\alpha,-\alpha,\beta,\alpha$.
\end{enumerate}
More precisely, we have to check that for any of these weights $\kappa$, either $\kappa$ is fixed by some $1\neq w$ under the dot action, or $w\cdot \kappa$ is dominant for some $w$ of length less or equal to $2$. This is done in the following table.

\[
\begin{array}{c|c|c}
\kappa&w\ \text{such that}\ w\cdot\kappa\ \text{is dominant}&1\neq w\ \text{fixing}\ \kappa\\
\hline
2\alpha&s_\beta&-\\
2\beta &s_\alpha&-\\
\alpha+\beta&1&-\\
2\alpha+\beta&1&-\\
\alpha+2\beta&1&-\\
2\alpha+2\beta &1&-\\
\alpha&-&s_\beta\\
\beta&-&s_\alpha\\
0&1&-\\
\alpha-\beta&s_\beta&-\\
\beta-\alpha&s_\alpha&-\\
-\alpha&s_\alpha&-\\
-\beta&s_\beta&-\\
-\alpha-\beta&-&\text{any}\ w\\
-2\alpha-\beta&s_\beta s_\alpha&-\\
-\alpha-2\beta&s_\alpha s_\beta&-
\end{array}
\]
\end{proof}
\begin{rem}\label{remarkgenericLparam}
Even though we can not prove Conjecture \ref{conjdegzero} in general, it follows that the restriction of $R\tilde\beta_{B,\ast}(\mathcal{O}_{\tilde{\mathbf{X}}_B})$ to the regular locus $X^{\rm reg}_G\subset X_G$ is a maximal Cohen-Macaulay module. 
In particular is it flat in the neighborhood of points $x=(\phi,N)\in X^{\rm reg}_G$ such that $X_G$ is smooth at $x$, by the Auslander-Buchsbaum formula. 
For $G=\GL_n$ this is the case if $(\phi^{\rm ss},N)$ is the L-parameter of a generic representation, see \cite[Lemma 1.3.2.(1)]{BLGGT}.
For a parabolic subgroup $P\subset G$ the morphism $\beta_P:\tilde X_P\rightarrow X_G$ is clearly not flat in general, as its fiber dimension can jump. And even the finite morphism $\tilde X_P^{\rm reg}\rightarrow X_G^{\rm reg}$ is not flat: at the intersection points of two irreducible components of $X_G^{\rm reg}$ the number of points in the fiber (counted with multiplicity) can jump. 
\end{rem}

\section{Smooth representations and modules over the Iwahori-Hecke algebra}

Let $F$ be a finite extension of $\Q_p$ (or of $\mathbb{F}_p(\!(t)\!)$) with residue field $k_F$ and let $q=p^r=|k_F|$. In the following let $\mathbb{G}$ be a split reductive group over $F$ and write $G=\mathbb{G}(F)$. From now on we will assume that $C$ contains a square root $q^{1/2}$ of $q$. We fix a choice of this root. 

We will always fix $\mathbb{T}\subset \mathbb{B}\subset \mathbb{G}$ a split maximal torus and a Borel subgroup. By this choice we can define the dual group $\check G$ of $G$ considered as an algebraic group over $C$. Moreover, we denote by $\check T\subset \check B\subset \check G$ the dual torus, resp.~the dual Borel.
More generally, given a parabolic subgroup $\mathbb{P}\subset \mathbb{G}$ containing $\mathbb{B}$, we denote by $\check P\subset \check G$ the corresponding parabolic subgroup of the dual group.
We write $W=W_G=W(\mathbb{G},\mathbb{T})$ for the Weyl group of $(\mathbb{G},\mathbb{T})$.
If $\mathbb{P}\subset \mathbb{G}$ is a parabolic subgroup containing $\mathbb{B}$, then the choice of $\mathbb{T}$ defines a lifting of the Levi quotient $\mathbb{M}$ of $\mathbb{P}$ to a subgroup of $\mathbb{G}$. Similarly, we regard the dual group $\check M$ of $\mathbb{M}$ as a subgroup of $\check G$ containing the maximal torus $\check T$. We write $W_M\subset W$ for the Weyl group of $(\mathbb{M},\mathbb{T})$.

Let $\check G/\hspace{-.1cm}/\check G$ denote the GIT quotient of $\check G$ with respect to its adjoint action on itself. The inclusion $\check T\hookrightarrow \check G$ induces an isomorphism $\check T/W\cong \check G/\hspace{-.1cm}/\check G$. 
The projection $X_{\check G}\rightarrow {\check G}$ induces a map
\begin{equation}\label{maptoTmodW}
\chi=\chi_G: X_{\check G}\longrightarrow \check G\longrightarrow \check G/\hspace{-.1cm}/\check G=\check T/W
\end{equation}
which is $\check G$-equivariant and hence induces a map $\bar\chi=\bar\chi_{G}:[X_{\check G}/\check G]\rightarrow \check T/W$. 
Similarly, we obtain morphisms $$\chi_M:X_{\check M}\rightarrow \check T/W_M\ \ \text{and}\ \ \bar\chi_M:[X_{\check M}/\check M]\rightarrow \check T/W_M.$$

\subsection{Categories of smooth representations} 
Let us write ${\rm Rep}(G)$ for the category of smooth representations of $G$ on $C$ vector spaces. It is well known that ${\rm Rep}(G)$ has a decomposition into Bernstein blocks
\[{\rm Rep}(G)=\prod_ {[M,\sigma]\in\Omega(G)}{\rm Rep}_{[M,\sigma]}(G),\]
where $\Omega(G)$ is a set of equivalence classes of a Levi $M$ of $G$ and a cuspidal representation $\sigma$ of $M$, see \cite[III, 2.2]{Bernstein} for example.
We will restrict our attention to the Bernstein component ${\rm Rep}_{[T,1]}(G)$, where $1$ is the trivial representation of the torus $T$. Given $\pi\in{\rm Rep}(G)$ we write $\pi_{[T,1]}$ for its image under the projection to ${\rm Rep}_{[T,1]}(G)$.
Moreover, we will write $\mathfrak{Z}_G$ for the center of the category ${\rm Rep}_{[T,1]}(G)$, then 
\[\mathfrak{Z}_G\xrightarrow{\cong}\Gamma(\check G/\hspace{-.1cm}/\check G,\Ocal_{\check G/\hspace{-.1cm}/\check G}),\]
see below for an explicit description. 
This isomorphism allows us to identify the category $\mathfrak{Z}_G\text{-mod}$ of $\mathfrak{Z}_G$-modules with the category ${\rm QCoh}(\check T/W)$ of quasi-coherent sheaves on the adjoint quotient $\check G/\hspace{-0.1cm}/\check G=\check T/W$ of $\check G$, and the category $\mathfrak{Z}_G\text{-mod}_{\rm fg}$ of finitely generated $\mathfrak{Z}_G$-modules with the category ${\rm Coh}(\check T/W)$ of coherent sheaves on $\check T/W$.
We obtain an identification of derived categories
\begin{equation}
\begin{aligned}
{\bf D}(\mathfrak{Z}_G\text{-mod})&\cong {\bf D}_{\rm QCoh}(\check T/W),\\
{\bf D}^b(\mathfrak{Z}_G\text{-mod}_{\text{fg}})&\cong {\bf D}^b_{\rm Coh}(\check T/W).
\end{aligned}
\end{equation}
We use these identifications and the morphism $\bar\chi:[X_{\check G}/\check G]\rightarrow \check T/W$ to make ${\bf D}^+_{\rm QCoh}([X_{\check G}/\check G])$ and ${\bf D}^b_{\rm Coh}([X_{\check G}/\check G])$ into $\mathfrak{Z}_G$-linear categories.

If $\mathbb{P}\subset \mathbb{G}$ is a parabolic subgroup with Levi quotient $\mathbb{M}$, we write $$\iota_P^G={\rm Ind}_P^G(\delta_P^{1/2}\otimes -):{\rm Rep}(M)\longrightarrow{\rm Rep}(G)$$ for the normalized parabolic induction, and $\iota_{\overline{P}}^G$ for normalized parabolic induction of the opposite parabolic $\overline{P}$ of $P$ (note that the normalization uses the choice of $q^{1/2}$). These functors are exact and restrict to functors $${\rm Rep}_{[T_M,1]}(M)\longrightarrow {\rm Rep}_{[T,1]}(G)$$ (for any choice of a maximal split torus $T_M$ of $M$).
Using a splitting $\mathbb{M}\hookrightarrow \mathbb{P}\subset\mathbb{G}$ to the projection we obtain a morphism
\[\check M/\hspace{-.1cm}/\check M\longrightarrow \check G/\hspace{-.1cm}/\check G\]
which is obviously independent of the choice of $\mathbb{M}\hookrightarrow\mathbb{G}$. Then the functors $\iota_P^G$ and $\iota_{\overline{P}}^G$ are linear with respect to the morphism 
\[\mathfrak{Z}_G\cong\Gamma(\check G/\hspace{-.1cm}/\check G, \Ocal_{\check G/\hspace{-.1cm}/\check G})\longrightarrow\Gamma(\check M/\hspace{-.1cm}/\check M,\Ocal_{\check M/\hspace{-.1cm}/\check M})\cong\mathfrak{Z}_M,\]
see below for details.

Let us write ${\bf D}({\rm Rep}_{[T,1]}(G))$ respectively ${\bf D}^+({\rm Rep}_{[T,1]}(G))$ for the derived category, respectively for the bounded below derived category, of ${\rm Rep}_{[T,1]}(G)$. Moreover, we write ${\bf D}^b({\rm Rep}_{[T,1],{\rm fg}}(G))$ for the full subcategory of complexes whose cohomology is concentrated in bounded degrees and is finitely generated as a $C[G]$-module. Then $\iota_P^G$ and $\iota_{\overline{P}}^G$ induce functors
\begin{align*}
{\bf D}({\rm Rep}_{[T_M,1]}(M))&\longrightarrow {\bf D}({\rm Rep}_{[T,1]}(G))\\
{\bf D}^+({\rm Rep}_{[T_M,1]}(M))&\longrightarrow {\bf D}^+({\rm Rep}_{[T,1]}(G))\\
{\bf D}^b({\rm Rep}_{[T_M,1],{\rm fg}}(M))&\longrightarrow {\bf D}^b({\rm Rep}_{[T,1],{\rm fg}}(G))
\end{align*}
which we will also denote by $\iota_P^G$ respectively $\iota_{\overline{P}}^G$.

Given two parabolic subgroup $\mathbb{P}_1\subset\mathbb{P}_2$ of $\mathbb{G}$ with Levi quotient $\mathbb{M}_1$ respectively $\mathbb{M}_2$. We write $\mathbb{P}_{12}$ for the image of $\mathbb{P}_1$ in $\mathbb{M}_2$. 
Then we have natural isomorphisms
\begin{equation}\label{compoofinduction}
\begin{aligned}
\iota_{P_{2}}^{G}\circ \iota_{P_{12}}^{M_2}&\longrightarrow \iota_{P_{1}}^{G},\\
\iota_{\overline{P}_{2}}^{G}\circ \iota_{\overline{P}_{12}}^{M_2}&\longrightarrow \iota_{\overline{P}_{1}}^{G}
\end{aligned}
\end{equation}
of functors ${\bf D}({\rm Rep}_{[T_{M_1},1]}(M_1))\rightarrow {\bf D}({\rm Rep}_{[T,1]}(G))$.

Finally, recall that a Whittaker datum is a $G$-conjugacy class of tuples $(\mathbb{B},\psi)$, where $\mathbb{B}\subset \mathbb{G}$ is a Borel subgroup and $\psi:N\rightarrow C^\times$ is a generic character of $N=\mathbb{N}(F)$, where $\mathbb{N}\subset \mathbb{B}$ is the unipotent radical. 
As above we fix the choice of a Borel subgroup $\mathbb{B}$ and a maximal split torus $\mathbb{T}\subset\mathbb{G}$. 
For a parabolic $\mathbb{P}\subset\mathbb{G}$ containing $\mathbb{B}$ with Levi quotient $\mathbb{M}$ we write $\psi_M:N_M\rightarrow C^\times$ for the restriction of $\psi$ to the unipotent radical $N_M\subset N$ of the Borel $B_M=B\cap M$ of $M$. Note that the $M$-conjugacy class of $(B_M,\psi_M)$ does not depend on the choice of $\mathbb{M}\hookrightarrow\mathbb{G}$ (i.e.~on the choice of $\mathbb{T}$).

We can describe the above categories of representations in terms of modules over Iwahori-Hecke algebras. In order to do so, let us fix a hyperspecial vertex in the apartment of the Bruhat-Tits building of $G$ defined by the maximal torus $\mathbb{T}$, i.e.~we fix $\Ocal_F$-models of $(\mathbb{G},\mathbb{T})$. The choice of a Borel $\mathbb{B}$ then defines an  Iwahori subgroup $I\subset G$. We write ${\rm Rep}^I G$ for the category of  smooth representations $\pi$ of $G$ on $C$-vector spaces that are generated by their Iwahori fixed vectors $\pi^I$ and ${\rm Rep}_{\rm fg}^I G\subset {\rm Rep}^I G$ for the full subcategory of representations that are finitely generated (as $C[G]$-modules). 
It is well known that the category ${\rm Rep}^I G$ does not depend on the choice of $I$ and agrees with the Bernstein block ${\rm Rep}_{[T,1]}(G)$.

Let $\Hcal_G=\Hcal(G,I)={\rm End}_G(\text{c-ind}_I^G\mathbf{1}_I)$ denote the Iwahori-Hecke algebra. Then 
\[\pi\longmapsto \pi^I={\rm Hom}_G(\text{c-ind}_I^G\mathbf{1}_I,\pi)\]
induces an equivalence of categories between ${\rm Rep}_{[T,1]}(G)={\rm Rep}^I G$ and the category $\Hcal_G\text{-mod}$ of $\Hcal_G$-modules. This equivalence identifies ${\rm Rep}^I_{\rm fg} G$ and the full subcategory $\Hcal_G\text{-mod}_{\text{fg}}\subset \Hcal_G\text{-mod}$ of finitely generated $\Hcal_G$-modules. 
Moreover, it identifies the center $\mathfrak{Z}_G$ of ${\rm Rep}_{[T,1]}(G)$ with the center of the Iwahori-Hecke algebra $\Hcal_G$. 
Then we have an isomorphism 
\[\mathfrak{Z}_G\cong C[X_\ast(\mathbb{T})]^W=C[X^\ast(\check T)]^W=\Gamma(\check T/W,\Ocal_{\check T/W})=\Gamma(\check G/\hspace{-0.1cm}/\check G,\Ocal_{\check G/\hspace{-0.1cm}/\check G})\]
(see for example \cite[Lemma 2.3.1]{HainesKottwitzPrasad}), which is in fact independent of the choice of the Iwahori $I$.

Given a representation $\pi\in {\rm Rep}^IG$ and a $\mathfrak{Z}_G$-module $\rho$ we will sometimes (by abuse of notation) write $\pi\otimes_{\mathfrak{Z}_G}\rho$ for the pre-image of the $\mathcal{H}_G$-module $\pi^I\otimes_{\mathfrak{Z}_G}\rho$ under the equivalence ${\rm Rep}^I G\cong \Hcal_G\text{-mod}$ (and similarly for corresponding derived functors). 

\begin{rem}
Note that if $\mathbb{G}=\mathbb{T}$ is a split torus, then $I=I_T=T^\circ$ is the unique maximal compact subgroup of $T$ and we have canonical identifications 
\begin{equation}\label{HeckealgTorus}
C[X_\ast(\mathbb{T})]\cong C[T/T^\circ]=\Hcal_T.
\end{equation}
where the first isomorphism is given by $\mu\mapsto \mu(\varpi)$ for the choice of a uniformizer $\varpi$ of $F$ (note that this isomorphism is independent of this choice).
We often use this isomorphism to identify unramified characters and $\Hcal_T$-modules.
\end{rem}

 

Let $\mathbb{P}\subset \mathbb{G}$ be a parabolic subgroup containing $\mathbb{B}$ with Levi quotient $\mathbb{M}$ and write $P=\mathbb{P}(F)$ and $M=\mathbb{M}(F)$. 
Set $I_M=I_G\cap M$, which is an Iwahori-subgroup of $M$, in particular ${\rm Rep}_{[T_M,1]}(M)={\rm Rep}^{I_M}M$.
There is a canonical embedding $\mathcal{H}_M\hookrightarrow \mathcal{H}_G$ such that the diagrams
\begin{equation}\label{inductiontensor}
\begin{aligned}
\begin{xy}
\xymatrix{
{\rm Rep}^{I_M}M \ar[r]^{(-)^{I_M}}\ar[d]_{\iota_P^G} & \mathcal{H}_M\text{-mod}\ar[d]^{{\rm Hom}_{\mathcal{H}_M}(\mathcal{H}_G,-)} & \text{and} & {\rm Rep}^{I_M}M \ar[r]^{(-)^{I_M}}\ar[d]_{\iota_{\overline P}^G} & \mathcal{H}_M\text{-mod}\ar[d]^{\mathcal{H}_G\otimes_{\mathcal{H}_M}-}\\
{\rm Rep}^{I_G}G \ar[r]^{(-)^{I}} & \mathcal{H}_G\text{-mod} && {\rm Rep}^{I_G}G \ar[r]^{(-)^{I}} & \mathcal{H}_G\text{-mod}.
}
\end{xy}
\end{aligned}
\end{equation}
commute. 
Note that this is equivalent to the commutativity of the diagram
\begin{equation}\label{Jacquetforget}
\begin{aligned}
\begin{xy}
\xymatrix{
{\rm Rep}^{I_M}M \ar[r]^{(-)^{I_M}} & \mathcal{H}_M\text{-mod}\\
{\rm Rep}^{I_G}G \ar[r]^{(-)^{I}} \ar[u]^{r_P^G(-)}& \mathcal{H}_G\text{-mod}\ar[u]_{\text{forget}}.
}
\end{xy}
\end{aligned}\end{equation}
Here $r^G_P(-)$ is the normalized Jacquet-module which is the left adjoint functor to $\iota_P^G(-)$. It is also the right adjoint functor to $\iota_{\overline{P}}^G(-)$ by Bernstein's second adjointness theorem.
By abuse of notation we will often write $\iota_{P}^G$ respectively $\iota_{\overline{P}}^G$ for the functors ${\rm Hom}_{\mathcal{H}_M}(\mathcal{H}_G,-)$  respectively $\mathcal{H}_G\otimes_{\mathcal{H}_M}-$ on Hecke modules.

The embedding $\Hcal_M\subset \Hcal_G$ induces an embedding $\mathfrak{Z}_G\subset \mathfrak{Z}_M$, where $\mathfrak{Z}_M$ is the center of ${\rm Rep}_{[T_M,1]}(M)$ which is identified with the center of $\mathcal{H}_M$, such that the canonical diagram
\[\begin{xy}
\xymatrix{
\mathfrak{Z}_G\ar[r] \ar[d] & \Gamma(\check T/W,\mathcal{O}_{\check T/W})\ar[d]\\
\mathfrak{Z}_M\ar[r]  & \Gamma(\check T/W_M,\mathcal{O}_{\check T/W_M})
}
\end{xy}\]
commutes. We deduce that $\iota_P^G$ and $\iota_{\overline{P}}^G$ are $\mathfrak{Z}_G$-linear. 
In particular, for a $\mathfrak{Z}_G$-module $\rho$ we obtain natural isomorphisms
\begin{equation}\label{inductionandcenter}
\begin{aligned}
\iota_P^G(-\otimes_{\mathfrak{Z}_G}\rho)&\longrightarrow \iota_P^G(-)\otimes_{\mathfrak{Z}_G}\rho,\\
\iota_{\overline{P}}^G(-\otimes_{\mathfrak{Z}_G}\rho)&\longrightarrow \iota_{\overline{P}}^G(-)\otimes_{\mathfrak{Z}_G}\rho,
\end{aligned}
\end{equation}
and similarly for the corresponding functors on the derived category. 

\subsection{The main conjecture}
Using the notations introduced above we state the following conjecture.
Variants of the conjecture have been around in representation theory in the past years. A proof of the conjecture is announced in the work of Ben-Zvi--Nadler--Helm \cite{BZCHN}, and by Zhu \cite{Zhu}.
\begin{conj}\label{mainconjecture}
There exists the following data:
\begin{enumerate}
\item[(i)] For each $(\mathbb{G},\mathbb{B},\mathbb{T},\psi)$ consisting of a reductive group $\mathbb{G}$, a Borel subgroup $\mathbb{B}$, a split maximal torus $\mathbb{T}\subset\mathbb{B}$, and a (conjugacy class of a) generic character $\psi:N\rightarrow C^\times$ there exists
an exact and fully faithful $\mathfrak{Z}_G$-linear functor
\[R_G^\psi:{\bf D}^+({\rm Rep}_{[T,1]}(G))\longrightarrow {\bf D}^+_{\rm QCoh}([X_{\check G}/\check G]),\]
\item[(ii)] for $(\mathbb{G},\mathbb{B},\mathbb{T},\psi)$ as in {\rm (i)} and each parabolic subgroup $\mathbb{P}\subset \mathbb{G}$ containing $\mathbb{B}$ 
there exists a natural $\mathfrak{Z}_{G}$-linear isomorphism
\[\xi_{P}^{G}:R_{G}^\psi\circ \iota_{\overline{P}}^{G}\longrightarrow (R\beta_{\ast}\circ L\alpha^\ast)\circ R_{M}^{\psi_M}\]
of functors ${\bf D}^+({\rm Rep}_{[T_M,1]}M)\rightarrow {\bf D}^+_{\rm QCoh}([X_{\check G}/\check G])$.
Here $\mathbb{M}$ is the Levi quotient of $\mathbb{P}$ and 
\begin{align*}
\alpha:[\mathbf{X}_{\check P}/\check P]&\longrightarrow[X_{\check M}/\check M],\\
\beta:[\mathbf{X}_{\check P}/\check P]&\longrightarrow[X_{\check G}/\check G]
\end{align*} 
are the morphisms on stacks induced by the natural maps $\check P\rightarrow \check M$ and $\check P\rightarrow \check G$.
\end{enumerate}
These data satisfy the following conditions:
\begin{enumerate}
\item[(a)] If $\mathbb{G}=\mathbb{T}$ is a split torus, then $R_T=R_T^{\psi}$ is induced by the identification $(\ref{HeckealgTorus})$ and viewing a sheaf on $\check T$ as an $\check T$-equivariant sheaf with the trivial $\check T$-action (note that $\check T$ acts trivially on $\check T=X_{\check T}$).
\item[(b)] Let $(\mathbb{G},\mathbb{B},\mathbb{T},\psi)$ as in {\rm (i)} and let $\mathbb{P}_1\subset\mathbb{P}_2\subset \mathbb{G}$ be parabolic subgroups containing $\mathbb{B}$ with Levi quotients $\mathbb{M}_1$ and $\mathbb{M}_2$. Let $\mathbb{P}_{12}$ denote the image of $\mathbb{P}_1$ in $\mathbb{M}_2$. 
Then, with the notations from $(\ref{compisitiondiagram})$, the diagram 
\[\begin{xy}
\xymatrix{
&R_{G}^\psi\circ \iota_{\overline{P}_{1}}^{G} \ar[dl]_{\xi_{P_1}^{G}}\ar[dr]^{(\ref{compoofinduction})} & \\
R\beta_{1,\ast}L\alpha_{1}^\ast \circ R^{\psi_{M_1}}_{M_1}\ar[d]_{(\ref{compoinductionGalois})}&& R_{G}^\psi\circ \iota_{\overline{P}_{2}}^{G}\circ  \iota_{\overline{P}_{12}}^{M_2}\ar[d]^{\xi_{P_2}^{G}}\\
(R\beta_{2,\ast}L\alpha_{2}^\ast)\circ (R\beta_{12,\ast}L\alpha_{12}^\ast) \ar[rr]^{\xi_{P_{12}}^{M_2}}\circ R^{\psi_{M_1}}_{M_1} && (R\beta_{2,\ast}L\alpha_{2}^\ast)\circ R^{\psi_{M_2}}_{M_2}\circ  \iota_{\overline{P}_{12}}^{M_2}
}
\end{xy}\]
is a commutative diagram of functors $${\bf D}^+({\rm Rep}_{[T_{M_1},1]}(M_1))\longrightarrow {\bf D}^+_{\rm QCoh}([X_{\check G}/\check G]).$$
\item[(c)] For any $(\mathbb{G},\mathbb{B},\mathbb{T},\psi)$ as in {\rm (i)} let $(\cind_N^G\psi)_{[T,1]}$ denote the projection of the compactly induced representation $\cind_N^G\psi$ to ${\rm Rep}_{[T,1]}(G)$. Then
\[R_G^\psi((\cind_N^G\psi)_{[T,1]})\cong \Ocal_{[X_{\check G}/\check G]}.\]
\end{enumerate}
\end{conj}
Let us point out that the $\mathfrak{Z}_G$-linearity of the conjectured functor $R_G^\psi$ implies that for each $\rho\in {\bf D}^+(\mathfrak{Z}_G\text{-}{\rm mod})$ there is a natural isomorphism
\[\psi_{G,\rho}: R_G^\psi(-\otimes_{\mathfrak{Z}_G}^L\rho)\xrightarrow{\cong} R_G^\psi(-)\otimes^L_{\Ocal_{[X_{\check G}/\check G]}}L\bar\chi_{G}^\ast\rho\]
of functors ${\bf D}^+({\rm Rep}_{[T,1]}(G))\rightarrow {\bf D}^+_{\rm QCoh}([X_{\check G}/\check G])$ which is functorial in $\rho$ (in the obvious sense). 
Moreover, given $\mathbb{P}\subset \mathbb{G}$ as in (ii), the $\mathfrak{Z}_{G}$-linearity of the natural isomorphism $\xi_{P}^{G}$ implies that the natural transformations $\psi_{M,\rho}$ and $\psi_{G,\rho}$ are compatible with parabolic induction. We do not spell this out explicitly in terms of commutative diagrams. 

\begin{rem}\label{remaboutconj}
\noindent (a) We expect that the conjectured functor $R_G^\psi$ induces a functor $${\bf D}^b({\rm Rep}_{[T,1],{\rm fg}}(G))\longrightarrow {\bf D}^b_{\rm Coh}([X_{\check G}/\check G]).$$
This would allow to extend the functor to the full derived category ${\bf D}({\rm Rep}_{[T,1]}(G))$: as ${\rm Rep}_{[T,1],{\rm fg}}(G)\cong \mathcal{H}_G\text{-mod}_{\text{fg}}$ and as $\mathcal{H}_G$ has finite global dimension, see \cite[4. Theorem 29]{Bernstein}, the full derived category ${\bf D}({\rm Rep}_{[T,1]}(G))$ is the ind-completion of ${\bf D}^b({\rm Rep}_{[T,1],{\rm fg}}(G))$.
Hence the conjectured functor would extend to a fully faithful and exact functor
\[{\bf D}({\rm Rep}_{[T,1]}(G))\longrightarrow {\rm IndCoh}([X_{\check G}/\check G]),\]
where ${\rm IndCoh}([X_{\check G}/\check G])$ is the ind-completion of ${\bf D}^b_{\rm Coh}([X_{\check G}/\check G])$. Note that this category differs from ${\bf D}_{\rm QCoh}([X_{\check G}/\check G])$, as $X_{\check G}$ is singular. However, there is a canonical equivalence $${\rm IndCoh}^+([X_{\check G}/\check G])\xrightarrow{\cong}{\bf D}^+_{\rm QCoh}([X_{\check G}/\check G]),$$ see e.g.~\cite[3.2.4]{DG}. In particular, restricting to bounded below objects, the conjecture that the (yet hypothetical) functor $R_G^\psi$ is fully faithful does not depend on whether we consider it as a functor with values in ${\rm IndCoh}^+([X_{\check G}/\check G])$ or with values in ${\bf D}^+_{\rm QCoh}([X_{\check G}/\check G])$.
We hence arrive with a conjecture that parallels the formulation of the geometric Langlands program, see \cite{Gaitsgory}. Also the conjectured compatibility with parabolic induction agrees with the compatibility with parabolic induction in loc.~cit..
See also the formulation given in \cite[Conjecture 4.4.5]{Zhu}.

\noindent (b) Recall that an L-parameter for $G$ that is trivial on inertia is a $\check G$-conjugacy class $[\phi,N]$ of $(\phi,N)\in X_{\check G}(C)$ with $\phi$ semi-simple. 
We write $S_{[\phi,N]}=C_{[\phi,N]}/C^\circ_{[\phi,N]}$ for the quotient of the centralizer of $(\phi,N)$ by its connected component of the identity.
By the classification of Kazhdan-Lusztig \cite[Theorem 7.12]{KL} the irreducible representations in ${\rm Rep}^IG$ (respectively the simple objects in $\Hcal_G\text{-mod}$) are in bijection with pairs $([\phi,N],\rho)$, where $[\phi,N]$ is an L-parameter and $\rho$ runs through a certain set of irreducible representation of $S_{[\phi,N]}$. This parametrization depends on an additional choice that corresponds to the choice of a Whittaker datum $(\mathbb{B},\psi)$.
More precisely, the classification in \cite{KL} (which in the case of $\GL_n$ coincides with the Bernstein-Zelevinsky classification \cite{BernsteinZele}) associates to $([\phi,N],\rho)$ an indecomposable representation (respectively Hecke module) $\pi^\psi_{[\phi,N],\rho}$ which has a unique irreducible quotient. 
Conjecture \ref{mainconjecture} should have the following relation with this classification. For simplicity we only treat the case of regular semi-simple $\phi$, the general case seems to be much more involved. 
 
Given $[\phi,N]$ let us write $$X_{\check G,[\phi,N]}^\circ\subset X_{\check G}$$ for the $\check G$-orbit of $(\phi,N)$\footnote{If $\phi$ is not regular semi-simple we have to replace this space by the locally closed subscheme whose $C$-valued points are given by those $(\phi',N')$ such that $[\phi,N]$ is the $\check G$-conjugacy class of $(\phi'^{\rm ss},N)$, where $\phi'^{\rm ss}$ is the semi-simplification of $\phi'$. 
Also the definition of $X_{\check G,[\phi,N]}$, the expected support of the sheaf $\mathcal{F}_{[\phi,N],\rho}$, should be more complicated: it should be given by the union of those $X_{\check G,[\phi',N']}^\circ$ such that $(\phi'^{\rm ss},N')$ lies in the Zariski closure of the $\check G$-orbit of $(\phi^{\rm ss},N)$ . This closed subscheme is in fact larger than the Zariski closure of $X_{\check G,[\phi,N]}^\circ$.
Finally the description of the equivariant sheaf $\mathcal{F}_{[\phi,N],\rho}$ should be more involved as well.}. 
Moreover, we denote by $$X_{\check G,[\phi,N]}=\overline{X_{\check G,[\phi,N]}^\circ}$$ its Zariski closure. As we assume that $\phi$ is regular semi-simple we can, given an irreducible representation $\rho$ of $S_{[\phi,N]}$ on a finite dimensional $C$-vector space, use $\rho$ to define a $\check G$-equivariant coherent sheaf $$\tilde{\mathcal{F}}_{[\phi,N],\rho}\in{\rm Coh}(X_{\check G,[\phi,N]})$$ which hence defines a coherent sheaf $\mathcal{F}_{[\phi,N],\rho}$ on the closed substack $$[X_{\check G,[\phi,N]}/\check G]\subset [X_{\check G}/\check G].$$ 
We then expect that the conjectured functor $R_G$ has the property
\[R_G^\psi(\pi_{[\phi,N],\rho}^\psi)=\mathcal{F}_{[\phi,N],\rho}.\]
If the L-parameter $[\phi,N]$ is generic, there is a unique $\psi$-generic representation $\pi$ in the L-packet defined by $[\phi,N]$. With the above notations this representation is the representation $$\pi=\pi^\psi_{[\phi,N],{\rm trivial}}.$$
Then, the expected formula above specializes to
\[R_G^\psi(\pi)=\mathcal{O}_{[X_{\check G,[\phi,N]}/\check G]}\]

\noindent (c) We point out that the conjectured functor $R_G^\psi$ will not be essentially surjective. 
In fact this is already obvious in the case $G=T$ a split torus. Here $R_T=R_T^\psi$ is the derived version of the functor
\[\Hcal_T\text{-}{\rm mod}\cong{\rm QCoh}(\check T)\longrightarrow {\rm QCoh}([\check T/\check T]).\]
The morphism on the right hand side is the embedding given by equipping a quasi-coherent sheaf with the trivial $\check T$-action. Obviously $\check T$-equivariant sheaves with non-trivial $\check T$-action are not in the essential image.

There is also a second obstruction for essential surjectivity for general $G$ (i.e.~$G$ is not assumed to be a torus).  
Let $[\phi,N]$ be an L-parameter such that $\phi$ is semi-simple but not regular semi-simple. Then (using the notation of (b) and its footnote) the structure sheaf of the closed substack $$[X_{\check G,[\phi,N]}^{\rm ss}/\check G]\subset [X_{\check G,[\phi,N]}/\check G]$$ of pairs $(\phi',N')$ where $\phi'$ is (pointwise) semi-simple should not be in the essential image of the functor $R_G$.

Following Fargues-Scholze \cite{FS} and Zhu \cite[4.6]{Zhu}, the failure of essential surjectivity should be fixed by replacing the category of smooth representations by a larger category. 

\noindent (d) Finally we point out that in the conjecture it is necessary to pass to derived categories. Heuristically this can be explained by the fact that flat morphisms on the representation theory side correspond to non-flat morphisms on the side of stacks: for example $\Hcal_G$ is flat over its center, whereas the canonical morphism $$\bar\chi_G:[X_{\check G}/\check G]\longrightarrow \check T/W$$ is not flat (as it maps some irreducible components to proper closed subschemes of $\check T/W$).
Moreover, we will see below that in the case of ${\rm GL}_n(F)$ the trivial representation will be mapped to a complex concentrated in cohomological degree $1-n$, see Remark \ref{remimageofLL} below. Hence, without passing to derived categories, the functor can not be fully faithful. 
The canonical t-structures on the source (respectively target) should correspond to an exotic t-structure on the other side. However, we have no idea how this t-structure could be described intrinsically.
Moreover, the formulation of the conjecture needs the passage to derived schemes respectively derived stacks: as parabolic induction is transitive (in the sense that $(\ref{compoofinduction})$ is an isomorphism), the base change morphism $(\ref{compoinductionGalois})$ has to be an isomorphism as well. However, in the world of classical schemes the corresponding cartesian diagram is not Tor-independent in general. 
\end{rem}

\subsection{A generalization of the conjecture} Conjecture \ref{mainconjecture} in fact is a special case of a more general conjecture about the category ${\rm Rep}(G)$, instead of the Bernstein block ${\rm Rep}_{[T,1]}(G)$. 
Let us describe this generalization. A similar generalization is conjectured by Zhu \cite[Conjecture 4.5.1]{Zhu}. The generalization stated here can also be viewed as a special case of the main conjecture \cite[Conjecture I.10.2]{FS} of Fargues-Scholze. 

We continue to assume that $\mathbb{G}$ is a split reductive group with dual group $\check G$. 
Let us write $W_F$ for the Weil group of $F$ and $I_F\subset W_F$ for the inertia group. 
We define the space of $\check G$-valued Weil-Deligne representations to be the scheme $X_{\check G}^{\rm WD}$ representing the functor
\[R\longmapsto\left\{\rho:W_F\rightarrow \check G(C),\ N\in \Lie\check G\left| \begin{array}{*{20}c} \rho|_{J}\ \text{is trivial for some}\ J\subset I_F\ \text{open}\\ {\rm Ad}(\rho(\sigma))(N)=q^{-||\sigma||}N \end{array}
 \right.\right\}\]
on $C$-algebras $R$.  Here $||-||:W_F\rightarrow\mathbb{Z}$ is the usual projection. It is easy to see that $X_{\check G}^{\rm WD}$ is an infinite disjoint union of affine schemes and is equipped with a $\check G$-action via conjugation on $\rho$ and via the adjoint action on $N$.
The space of Weil-Deligne representations $X_{\check G}^{\rm WD}$ in fact agrees with the fiber over $C$ of the moduli space of L-parameters $\underline{Z}^1(W_F^\circ,\check G)$ studied in work of Dat-Helm-Kurinczuk-Moss \cite{DHKM} and is defined and studied as well in \cite[3.1]{Zhu}.

Similarly, for every parabolic subgroup $\check P\subset \check G$ we can define the scheme $X_{\check P}^{\rm WD}$ and the derived scheme ${\bf X}_{\check P}^{\rm WD}$ that come equipped with $\check P$-actions. 

The inclusion $\check P\hookrightarrow \check G$ and the projection $\check P\rightarrow \check M$ onto the Levi-quotient $\check M$ of $\check P$ induce morphisms
\begin{equation}\label{alphabetaWD}
\begin{aligned}
\beta_{\check P}^{\rm WD}:[{\bf X}_{\check P}^{\rm WD}/\check P]&\longrightarrow [X_{\check G}^{\rm WD}/\check G],\\
\alpha_{\check P}^{\rm WD}:[{\bf X}_{\check P}^{\rm WD}/\check P]&\longrightarrow [X_{\check M}^{\rm WD}/\check M]
\end{aligned}
\end{equation}
of the respective stack quotients. Moreover, we will write $X_{\check G}^{\rm WD}/\hspace{-0.1cm}/\check G$ for the GIT quotient of $X_{\check G}^{\rm WD}$ by the $\check G$-action. 
As in the case of the space of $(\phi,N)$-modules $X_{\check G}$ it is easy to show that $\beta_{\check P}^{\rm WD}$ is proper. The following summarizes properties of the spaces just introduced (which are proved using similar methods as in section \ref{spacesofLparam}).

Let $\check P\subset \check G$ be a parabolic subgroup with Levi-quotient $\check M$.
\begin{enumerate}
\item[(i)] The space $X_{\check G}^{\rm WD}$ is reduced and a local complete intersection. \\
(This follows from \cite[Theorem 4.1]{DHKM}. See also \cite[Proposition 3.1.6]{Zhu}.)
\item[(ii)] The morphism $\alpha_{\check P}^{\rm WD}:[{\bf X}_{\check P}^{\rm WD}/\check P]\rightarrow [X_{\check M}^{\rm WD}/\check M]$ has finite Tor-dimension.\\
(This follows from \cite[Lemma 3.3.1]{Zhu}.)
\item[(iii)] There is a morphism $X_{\check M}^{\rm WD}/\hspace{-.1cm}/\check M\rightarrow X_{\check G}^{\rm WD}/\hspace{-.1cm}/\check G$ making the diagram
\begin{tiny}
\[\begin{xy}
\xymatrix{
& [{\bf X}_{\check P}^{\rm WD}/\check P]\ar[dl]_{\beta_{\check P}^{\rm WD}}\ar[dr]^{\alpha_{\check P}^{\rm WD}} & \\
[X_{\check G}^{\rm WD}/\check G] \ar[d]&& [X_{\check M}^{\rm WD}\check M]\ar[d]\\
X_{\check G}^{\rm WD}/\hspace{-.1cm}/\check G && X_{\check M}^{\rm WD}/\hspace{-.1cm}/\check M \ar@{-->}[ll]
}
\end{xy}\]
\end{tiny}
commutative.\\
(This is the commutative diagram  \cite[(3.10)]{Zhu}. The morphism can easily be constructed using the morphism $X_{\check M}^{\rm WD}\rightarrow \mathbf{X}^{\rm WD}_{\check P}$ induced by the choice of a splitting of $\check P\rightarrow \check M$.)
\end{enumerate}

\begin{rem}
In the case of the space of $(\phi,N)$-modules $X_{\check G}$ all these properties have been verified in section \ref{spacesofLparam}.
In relation with (iii) we remark that the morphism
\[[X_{\check G}/\check G]\longrightarrow X_{\check G}/\hspace{-.1cm}/\check G\]
is just the morphism $\bar\chi$ from $(\ref{maptoTmodW})$, i.e.~the GIT quotient $X_{\check G}/\hspace{-.1cm}/\check G$ agrees with the adjoint quotient $\check G/\hspace{-.1cm}/\check G$. This can be seen as follows: the morphism $\phi\mapsto (\phi,0)$ defines a closed embedding $\check G\hookrightarrow X_{\check G}$ which is the inclusion of an irreducible component. As $\check G$ is reductive and $C$ has characteristic $0$ the category of $\check G$-representations is semi-simple and we obtain a closed embedding $$\check G/\hspace{-.1cm}/\check G\longrightarrow X_{\check G}/\hspace{-.1cm}/\check G.$$
As source and target are reduced (as $\check G$ and $X_{\check G}$ are) it is enough to show that the morphism is bijective. This comes down to proving that for $(\phi,N)\in X_{\check G}(k)$, for an algebraically closed field $k$, there exists $\phi'\in\check G(k)$ such that 
$$\overline{\check G\cdot (\phi',0)}\cap \overline{\check G\cdot (\phi,N)}\neq \emptyset.$$
By (the proof of) Lemma \ref{equidimparabolicregular} we may assume that $\phi\in \check B$ and $N\in \Lie\check B$ for some Borel $\check B\subset \check G$. Let $\mathbb{G}_m$ act on $X_{\check G}$ by the sum of the positive roots, then the closure of $\mathbb{G}_m\cdot (\phi,N)$ contains in addition the point $(\phi',0)$ for some $\phi'\in \check G$ such that $\phi$ and $\phi'$ have the same image in the adjoint quotient $\check G/\hspace{-.1cm}/\check G$. 
\end{rem}

Let us write $\mathfrak{Z}(G)$ for the Bernstein center of the category ${\rm Rep}(G)$. Given a Bernstein component $\Omega$ of ${\rm Rep}(G)$ we denote its center by $\mathfrak{Z}_\Omega(G)$. Moreover, we denote by $$\mathcal{Z}(\check G)=\Gamma(X_{\check G}^{\rm WD}/\hspace{-.1cm}/\check G,\Ocal_{X_{\check G}^{\rm WD}/\hspace{-.1cm}/\check G})$$ the ring of functions on the GIT quotient $X_{\check G}^{\rm WD}/\hspace{-.1cm}/\check G$. If $X\subset X_{\check G}^{\rm WD}$ is a connected component, we write $\mathcal{Z}_X(\check G)$ for the ring of functions on the GIT quotient $X/\hspace{-.1cm}/\check G$.

\begin{rem}
If $\mathbb{G}=\mathbb{T}$ is a split torus, then the isomorphism $F^\times \rightarrow W_F^{\rm ab}$ of local class field theory identifies $X_{\check T}^{\rm WD}$ with the scheme representing the functor
\[R\longmapsto \{\rho:F^\times\longrightarrow \check T(R)\ \text{smooth character}\}\]
on the category of $C$-algebras. 
In particular, the scheme $X_{\check T}^{\rm WD}$ decomposes into a disjoint union of copies of $\check T$ indexed by the smooth characters $\Ocal_F^{\times}\rightarrow \check T(C)$. This decomposition induces an equivalence of categories
\begin{equation}\label{equivlcft}
{\rm Rep}(T)\cong{\rm QCoh}(X_{\check T}^{\rm WD}).
\end{equation}
\end{rem}

We state a generalization of Conjecture \ref{mainconjecture}.
\begin{conj}\label{mainconjecturegeneral}
There exists the following data:
\begin{enumerate}
\item[(i)] For each $(\mathbb{G},\mathbb{B},\mathbb{T},\psi)$ consisting of a reductive group $\mathbb{G}$, a Borel subgroup $\mathbb{B}$, a split maximal torus $\mathbb{T}\subset\mathbb{B}$, and a (conjugacy class of a) generic character $\psi:N\rightarrow C^\times$ there exists
an exact and fully faithful functor
\[\mathcal{R}_G^\psi:{\bf D}^+({\rm Rep}(G))\longrightarrow {\bf D}^+_{\rm QCoh}([X^{\rm WD}_{\check G}/\check G]),\]
\item[(ii)] for $(\mathbb{G},\mathbb{B},\mathbb{T},\psi)$ as in {\rm (i)} and each parabolic subgroup $\mathbb{P}\subset \mathbb{G}$ containing $\mathbb{B}$ 
there exists a natural isomorphism
\[\xi_{P}^{G}:\mathcal{R}_{G}^\psi\circ \iota_{\overline{P}}^{G}\longrightarrow (R\beta^{\rm WD}_{\check P,\ast}\circ L\alpha_{\check P}^{{\rm WD},\ast})\circ \mathcal{R}_{M}^{\psi_M}\]
of functors ${\bf D}^+({\rm Rep}(M))\rightarrow {\bf D}^+_{\rm QCoh}([X^{\rm WD}_{\check G}/\check G])$.
Here $\mathbb{M}$ is the Levi quotient of $\mathbb{P}$ and $\alpha_{\check P}^{\rm WD}$ and $\beta_{\check P}^{\rm WD}$ are the morphisms defined in $(\ref{alphabetaWD})$.
\end{enumerate}
These data satisfy the following conditions:
\begin{enumerate}
\item[(a)] If $\mathbb{G}=\mathbb{T}$ is a split torus, then $\mathcal{R}_T=\mathcal{R}_T^{\psi}$ is induced by the equivalence $(\ref{equivlcft})$ given by  local class field theory.
\item[(b)] Let $(\mathbb{G},\mathbb{B},\mathbb{T},\psi)$ be as in {\rm (i)}. The morphism $\mathcal{Z}(\check G)\rightarrow \mathfrak{Z}(G)$ defined by fully faithfulness of $\mathcal{R}_G^\psi$ is independent of the choice of $\psi$ and induces a surjection
\[\omega_G:\left\{\begin{array}{*{20}c} \text{Bernstein components}\\  \text{of}\ {\rm Rep} (G)\end{array}\right\}\longrightarrow\left\{\begin{array}{*{20}c}\text{connected components}\\ \text{of}\ X_{\check G}^{\rm WD}\end{array}\right\}.\]
\item[(c)] Let $(\mathbb{G},\mathbb{B},\mathbb{T},\psi)$  and $\mathbb{P}$ be as in {\rm (ii)}. Then the natural isomorphism $\xi_P^G$ is $\mathcal{Z}(\check G)$-linear for the $\mathcal{Z}(\check G)$-linear structure on ${\rm Rep}(M)$ defined by the morphism $$\mathcal{Z}(\check G)\longrightarrow\mathcal{Z}(\check M)\longrightarrow \mathfrak{Z}(M)$$ 
that is given by the composition of the morphism $\mathcal{Z}(\check G)\rightarrow\mathcal{Z}(\check M)$ induced by $X_{\check M}^{\rm WD}/\hspace{-.1cm}/\check M\rightarrow X_{\check G}^{\rm WD}/\hspace{-.1cm}/\check G$ with the morphism $\mathcal{Z}(\check M)\rightarrow \mathfrak{Z}(M)$ of {\rm (b)}.
\item[(d)] Let $(\mathbb{G},\mathbb{B},\mathbb{T},\psi)$ as in {\rm (i)} and let $\mathbb{P}_1\subset\mathbb{P}_2\subset \mathbb{G}$ be parabolic subgroups containing $\mathbb{B}$ with Levi quotients $\mathbb{M}_1$ and $\mathbb{M}_2$. Let $\mathbb{P}_{12}$ denote the image of $\mathbb{P}_1$ in $\mathbb{M}_2$. 
Then the diagram \begin{tiny}
\[\begin{xy}
\xymatrix{
&\mathcal{R}_{G}^\psi\circ \iota_{\overline{P}_{1}}^{G} \ar[dl]_{\xi_{P_1}^{G}}\ar[dr]^{(\ast)} & \\
R\beta_{{\check P_1},\ast}^{\rm WD}L\alpha_{\check P_1}^{{\rm WD},\ast} \circ \mathcal{R}^{\psi_{M_1}}_{M_1}\ar[d]_{(\ast\ast)}&& \mathcal{R}_{G}^\psi\circ \iota_{\overline{P}_{2}}^{G}\circ  \iota_{\overline{P}_{12}}^{M_2}\ar[d]^{\xi_{P_2}^{G}}\\
(R\beta^{\rm WD}_{\check P_2,\ast}L\alpha_{\check P_2}^{{\rm WD},\ast})\circ (R\beta_{\check P_{12},\ast}^{\rm WD}L\alpha_{\check P_{12}}^{{\rm WD},\ast}) \ar[rr]^{\xi_{P_{12}}^{M_2}}\circ \mathcal{R}^{\psi_{M_1}}_{M_1} && (R\beta_{\check P_2,\ast}^{\rm WD}L\alpha_{\check P_2}^{{\rm WD},\ast})\circ \mathcal{R}^{\psi_{M_2}}_{M_2}\circ  \iota_{\overline{P}_{12}}^{M_2}
}
\end{xy}\]
\end{tiny}
is a commutative diagram of functors $${\bf D}^+({\rm Rep}(M_1))\longrightarrow {\bf D}^+_{\rm QCoh}([X^{\rm WD}_{\check G}/\check G]).$$ Here $(\ast)$ is the natural isomorphism given by transitivity of parabolic induction and $(\ast\ast)$ is a base change isomorphism defined by the analogous diagram as in $(\ref{compisitiondiagram})$.
\item[(e)] For $(\mathbb{G},\mathbb{B},\mathbb{T},\psi)$ as in {\rm (i)} there is an isomorphism
\[\mathcal{R}_G^\psi(\cind_N^G\psi)\cong \Ocal_{[X^{\rm WD}_{\check G}/\check G]}.\]
\end{enumerate}
\end{conj}

\begin{rem}
(a) It should be possible to construct the expected morphism $$\mathcal{Z}(\check G)\longrightarrow \mathfrak{Z}(G)$$ of (b) in the conjecture, without referring to the conjectured functor $\mathcal{R}_G^\psi$. In the case of ${\rm GL}_n$ a result like this has been established by Helm and Moss \cite{HelmMoss} (even with $\mathbb{Z}_{\ell}$-coefficients). More generally Fargues and Scholze \cite[Proposition I.9.3]{FS} give a construction of such a morphism. The construction in \cite{FS} uses the spectral action constructed in loc.cit.. While the morphism is an isomorphism in the ${\rm GL}_n$-case of \cite{HelmMoss}, this is not true in the general case.\\
(b) In fact $\mathcal{Z}(\check G)$ coincides with the \emph{stable Bernstein center} as defined by Haines in \cite[5.3.]{Haines}. This is a consequence of \cite[Theorem 6.10]{DHKM}. 
With this identification the morphism $\mathcal{Z}(\check G)\rightarrow\mathfrak{Z}(G)$ of (b) in the Conjecture should coincide with the morphism constructed in \cite[Proposition 5.5.1]{Haines} assuming the local Langlands correspondence. 
\end{rem}
Let us point out that the morphism $\omega_G$ from (b) can not be expected to be a bijection in general, as, for a given Whittaker datum $\psi$, not every Bernstein component $\Omega$ is $\psi$-generic\footnote{In fact there are groups with Bernstein components that are not $\psi$-generic for any choice of a Whittaker datum $\psi$.} in the sense of \cite[4.3]{BushnellHenniart} (note that the notions of being \emph{$\psi$-generic} and being \emph{simply $\psi$-generic} of \cite{BushnellHenniart} agree by Example 4.5 (1) of loc.~cit., as $\mathbb{G}$ is assumed to be (quasi-)split).
More precisely, Conjecture \ref{mainconjecturegeneral} predicts that the restriction of $\omega_G$ induces a bijection
\[\left\{\begin{array}{*{20}c} \psi\text{-generic Bernstein} \\ \text{components of}\ {\rm Rep} (G)\end{array}\right\}\longrightarrow\left\{\begin{array}{*{20}c}\text{connected} \\ \text{components of}\ X_{\check G}^{\rm WD}\end{array}\right\},\]
and that for a $\psi$-generic Bernstein component $\Omega$ the induced morphism
\begin{equation}\label{centervsstablecenter}
\mathcal{Z}_{\omega_G(\Omega)}(\check G)\longrightarrow \mathfrak{Z}_\Omega(G)
\end{equation}
is an isomorphism.
Indeed, combining 4.2.~Corollary and 4.3.~Theorem of \cite{BushnellHenniart} we deduce that the $\psi$-generic components are precisely those components $\Omega$ such that $(\cind_N^G\psi)_\Omega\neq 0$. 
Moreover, the morphism $(\ref{centervsstablecenter})$ fits in the commutative diagram
\[\begin{xy}
\xymatrix{
{\rm End}_G((\cind_N^G\psi)_\Omega) \ar[r]^{\cong} & {\rm End}_{[X_{\check G,\Omega}^{\rm WD}/\check G]}(\Ocal_{[X_{\check G,\Omega}^{\rm WD}/\check G]})\\
\mathfrak{Z}_\Omega(G)\ar[u]_{\cong} & \mathcal{Z}_{\omega_G(\Omega)}(\check G), \ar[u]^{\cong}\ar[l]
}
\end{xy}\]
where $X_{\check G,\Omega}^{\rm WD}\subset X_{\check G}^{\rm WD}$ denotes the connected component defined by $\omega_G(\Omega)$.
Here the upper horizontal arrow is an isomorphism by (e) and fully faithfulness in the conjecture, the right vertical arrow is an isomorphism by definition and the left vertical arrow is an isomorphism by \cite[4.3. Theorem]{BushnellHenniart}.

\begin{rem}
Conjecture \ref{mainconjecture} is concerned with the principal component ${\rm Rep}_{[T,1]}(G)$. This Bernstein component is $\psi$-generic for any choice of $\psi$. 
\end{rem}

%

\begin{rem}
In the case $\mathbb{G}={\rm GL}_n$ there is (up to conjugation) a unique choice of $(\mathbb{B},\psi)$ and every Bernstein component of ${\rm Rep}({\rm GL}_n(F))$ is $\psi$-generic, see e.g.~\cite[4.5, Examples (2)]{BushnellHenniart}.
Moreover, in this case one can show that $X_{\check G}^{\rm WD}$ decomposes into a disjoint union 
\[X_{{\rm GL}_n}^{\rm WD}=\coprod\nolimits_{\underline{n}} X_{\underline{n}},\]
where $\underline{n}=(n_{[\tau]})$ is a tuple of non-negative integers $n_{[\tau]}$ indexed by the $W_F$-conjugacy classes $[\tau]$ of irreducible $I_F$-representations $\tau:I_F\rightarrow{\rm GL}_{d_\tau}(C)$
such that  $$n=\sum_{[\tau]} [W_F:W_\tau]\cdot n_\tau d_\tau.$$
Here $W_\tau\subset W_F$ is the $W_F$-stabilizer of a representative $\tau$ of $[\tau]$. Moreover, each $X_{\underline{n}}$ is connected and decomposes into a product where each factor is a space of $(\phi,N)$-modules for a finite extension $F'$ of $F$.
On the other hand, the local Langlands correspondence for ${\rm GL}_n(F)$ induces a bijection 
\[\left\{\begin{array}{*{20}c} W_F\text{-conjugacy classes of}\\ \text{irreducible smooth representations}\\ \tau:I_F\rightarrow {\rm GL}_m(C),\ m\geq 1\end{array}\right\}\longleftrightarrow \left\{\begin{array}{*{20}c}\text{equivalence classes of}\\ \text{cuspidal representations}\\ {\rm GL}_r(F),\ r\geq 1\end{array}\right\}\]
where two cuspidal representations are said to be equivalent if they differ by the twist by an unramified character.
Hence we obtain a bijection between the Bernstein components of ${\rm Rep}({\rm GL}_n(F))$ and the connected components of $X_{{\rm GL}_n}^{\rm WD}$.
By results of Bushnell-Kutzko \cite{BushnellKutzko} every Bernstein component of ${\rm Rep}({\rm GL}_n(F))$ can be described by a semi-simple type and the corresponding Hecke-algebra is in fact isomorphic to a tensor product of Iwahori-Hecke algebras. This corresponds to the decomposition of the connected components $X_{\underline{n}}$ of $X_{{\rm GL}_n}^{\rm WD}$ into a product of spaces of $(\phi,N)$-modules. 
In fact, in the case of ${\rm GL}_n$ type theory and a closer inspection of these decompositions should reduce Conjecture \ref{mainconjecturegeneral} to Conjecture \ref{mainconjecture} (in the case of ${\rm GL}_r$ for various $r$).
In particular it should be possible to generalize all results proven in the following section for the block ${\rm Rep}_{[T,1]}({\rm GL}_n(F))$ to the whole category ${\rm Rep}({\rm GL}_n(F))$, compare also \cite[5]{BZCHN}.
\end{rem}

\section{The case of ${\rm GL}_n$}

In this section we consider the group $G={\rm GL}_n(F)$ and make Conjecture \ref{mainconjecture} more explicit in this case. We will provide a candidate for the conjectured functor and prove that it satisfies compatibility with parabolic induction on the dense open subset of regular elements. In the case of $\GL_2$ we give a full proof of the conjecture. 

We fix $\mathbb{G}=\GL_n$ and choose the canonical integral model of $\mathbb{G}$ over $\Ocal_F$ corresponding to the maximal compact subgroup $K=\GL_n(\Ocal_F)$ of $G$. 
In particular we assume that the hyperspecial vertex defined by $K$ is contained in the apartment defined by the maximal split torus $\mathbb{T}\subset \GL_n$, and $I\subset K$.
We use this to obtain canonical integral models for the choice of a Borel $\mathbb{B}\supset \mathbb{T}$ and for parabolic subgroups $\mathbb{P}\supset \mathbb{B}$ as well as for their Levi quotients. We will use the same symbols for these integral models.
We will often simply write $\mathfrak{Z}=\mathfrak{Z}_G$ for the Bernstein center of the category ${\rm Rep}_{[T,1]}(G)$.

\subsection{The modified Langlands correspondence}\label{modifiedLL}
We recall the construction of the \emph{modified local Langlands correspondence} defined by Breuil and Schneider in \cite[4]{BreuilSchneider}, see also \cite[4.2]{EmertonHelm}. We restrict ourselves to the Bernstein block ${\rm Rep}_{[T,1]}(G)$.

Let $\varpi$ be a uniformizer of $F$. For any field extension $L$ of $C$ and $\lambda\in L^\times$ we write ${\rm unr}_\lambda:F^\times\rightarrow L^\times$ for the unramified character mapping $\varpi$ to $\lambda$. 
More generally, for $\underline{\lambda}=(\lambda_1,\dots,\lambda_n)\in (L^\times)^n$ we write ${\rm unr}_{\underline{\lambda}}={\rm unr}_{\lambda_1}\otimes\dots\otimes{\rm unr}_{\lambda_n}:T\rightarrow L^\times$ for the unramified character of the torus $T=(F^\times)^n$ whose restriction to the $i$-th coordinate is ${\rm unr}_{\lambda_i}$.

Write $|-|={\rm unr}_{q^{-1}}:F^\times\rightarrow C^\times$ for the unramified character such that $|\varpi|=q^{-1}$.
Let $L$ be a field extension of $C$ and let $(\phi,N)\in X_{\check G}(L)$ be a $(\phi,N)$-module such that $\phi$ is semi-simple. 
Then Breuil and Schneider associate to $(\phi,N)$ a smooth, absolutely indecomposable representation ${\rm LL^{mod}}(\phi,N)$ of ${\rm GL}_n(F)$ with coefficients in $L$ as follows:

Fix an algebraic closure $\bar L$ of $L$. Given a scalar $\lambda\in \bar L^\times$ and $r\geq 0$ let ${\rm Sp}(\lambda,r)$ denote as usual the $(\phi,N)$-module structure on $\bar L^r=\bar Le_0\oplus \dots \bar Le_{r-1}$ defined by
\begin{equation}\label{standardphiNmodule}
\begin{aligned}
\phi(e_i)&=q^{-i}\lambda\\
N(e_i)&=\begin{cases}e_{i+1}, &i<r-1\\ 0, &i=r-1.\end{cases}
\end{aligned}
\end{equation}
Let ${\rm St}(\lambda,r)$ denote the generalized Steinberg representation of ${\rm GL}_r(F)$ with coefficients in $\bar L$, i.e.~the unique simple quotient of $\iota_B^G({\rm unr}_\lambda\otimes{\rm unr}_\lambda |-|\otimes\dots\otimes{\rm unr}_\lambda |-|^{n-1})$. 

Given some $(\phi,N)\in X_{\check G}(L)$ with $\phi$ semi-simple we decompose (after enlarging $L$ if necessary)
\[(L^n,\phi,N)\cong\bigoplus_{i=1}^s {\rm Sp}(\lambda_i,r_i)\]
and define ${\rm LL}^{\rm mod}(\phi,N)$ as the unique $L$-model of the $\bar L$ representation 
\begin{equation}\label{defnmodifiedLL}
\iota_P^G \big({\rm St}(\lambda_1,r_1)\otimes\dots\otimes{\rm St}(\lambda_{s},r_s)\big).
\end{equation}
Here $P$ is the block upper triangular parabolic whose Levi is the block diagonal subgroup ${\rm GL}_{r_1}\times\dots\times{\rm GL}_{r_s}$ and the $\lambda_i$ are ordered so that they satisfy the condition of \cite[Definition 1.2.4]{Kudla}.

\begin{rem}
Note that the normalization we use differs from the one in \cite{BreuilSchneider} and \cite{EmertonHelm}. 
In loc.~cit.~the representation ${\rm LL^{mod}}(\phi,N)$ is modified by the twist by $|\det|^{-(n-1)/2} $.
This has the advantage that the resulting $\GL_n(F)$ representation has a unique model over $L$, without assuming the existence (or fixing a choice) of $q^{1/2}$, as proven in \cite[Lemma 4.2]{BreuilSchneider}.
As we have fixed a choice $q^{1/2}$ in the base field $C$, and hence a choice of $|\det|^{-(n-1)/2} $, their argument also implies that our representation ${\rm LL^{mod}}(\phi,N)$ has a (unique) model over $L$. 
The reason for these two different normalizations is the following: 
In \cite{BreuilSchneider} the representation should be canonically defined over $L$, without choosing $q^{1/2}$, and moreover, in \cite{EmertonHelm} the representations should (conjecturally) satisfy some local-global compatibility. 
In our case we work purely locally and we are aiming for a compatibility with normalized parabolic induction. More precisely, we need Lemma \ref{JBofLLmod} below to be true as stated (i.e.~not a twisted version of it). Anyway, the definition of normalized parabolic induction forces us to choose a square root $q^{1/2}$.
\end{rem}

If $(\phi,N)\in X_{\check G}$ with non semi-simple $\varphi$, we write ${\rm LL^{\rm mod}}(\phi,N)={\rm LL^{mod}}(\phi^{\rm ss},N)$. Moreover, if $(\phi,N)$ is such that ${\rm LL}^{\rm mod}(\phi,N)$ is absolutely irreducible (that is if the $\check G$-conjugacy class $[\phi^{\rm ss},N]$ is a generic L-parameter), we usually just write ${\rm LL}(\phi,N)$ instead of ${\rm LL^{mod}}(\phi,N)$. 
Note that in this case ${\rm LL}(\phi,N)^\vee\cong {\rm LL}((\phi,N)^\vee)$, as normalized parabolic induction commutes with contragredients and as in this case the parabolic induction of the contragredient representation still satisfies the condition of \cite[Definition 1.2.4]{Kudla}. 

\begin{lem}\label{compatibilityofcenteraction}
Let $x=(\phi_x,N_x)\in X_{\check G}$. Then, using the notation of $(\ref{maptoTmodW})$ the center $\mathfrak{Z}$ acts on the representation ${\rm LL^{mod}}((\phi,N)^\vee)^\vee$ via the character $\chi_x:\mathfrak{Z}\rightarrow k(x)$ defined by $\chi(x)\in\check T/W=\Spec \mathfrak{Z}$
\end{lem}
\begin{proof}
By definition of ${\rm LL^{mod}}$ we may assume that $\phi$ is semi-simple. The representation ${\rm LL^{mod}}((\phi,N)^\vee)^\vee$ embeds into $$\iota_{\overline B}^G({\rm unr}_{\lambda_1}\otimes\dots\otimes{\rm unr}_{\lambda_n})$$ for some ordering $\underline{\lambda}=(\lambda_1,\dots,\lambda_n)$ of the eigenvalues of $\phi$.
Hence it follows that $\big({\rm LL^{mod}}((\phi,N)^\vee)^\vee\big)^I$ embeds into $\Hcal_G\otimes_{\Hcal_T}{\rm unr}_{\underline{\lambda}}$ and it is enough to prove that $\mathfrak{Z}\subset\Hcal_G$ acts on $\Hcal_G\otimes_{\Hcal_T}{\rm unr}_{\underline{\lambda}}$ as asserted. 
But as $\mathfrak{Z}\subset \Hcal_T$ is the center of $\Hcal_G$, it acts on $\Hcal_G\otimes_{\Hcal_T}{\rm unr}_{\lambda}$ via the same character as on ${\rm unr}_{\underline{\lambda}}$. The claim follows from this. 
\end{proof}

Recall that for a regular semi-simple endomorphism $\phi$ of an $L$-vector space $L^n$ with eigenvalues in $L$ there is a canonical bijection
\begin{equation}\label{orderingsversusflags}
\{\phi\text{-stable complete flags}\ \mathcal{F}\ \text{of}\ L^n\}\longleftrightarrow \{\text{orderings of the eigenvalues of}\ \phi\}.
\end{equation}
If $\mathcal{F}$ is a flag corresponding to an ordering $\underline{\lambda}=(\lambda_1,\dots,\lambda_n)$ of the eigenvalues of $\phi$, we denote by ${\rm unr}_{\mathcal{F}}={\rm unr}_{\underline{\lambda}}$ the $L$-valued unramified character defined by this ordering.

\begin{lem}\label{JBofLLmod}
Let $x=(\phi_x,N_x)\in X_{\check G}$ with $\phi_x$ regular semi-simple and let $L$ be an (algebraic) extension of $k(x)$ containing the eigenvalues of $\phi_x$. 
Then $$r^G_B({\rm LL^{mod}}((\phi_x,N_x)^{\vee})^\vee\otimes_{k(x)}L)=\bigoplus_{\mathcal{F}}{\rm unr}_{\mathcal{F}},$$
where the direct sum runs over all flags of $L^n$ stable under $\phi_x$ and $N_x$. 
\end{lem}
\begin{proof}
The lemma is an application of the geometrical Lemma \cite[2.11, p.~448]{BernsteinZele} describing the composition of parabolic induction with the Jacquet functor.

Assume first that $(\phi,N)\otimes_{k(x)}L={\rm Sp}(\lambda,r)$, see $(\ref{standardphiNmodule})$. Then we need to compute the Jacquet-module of the generalized Steinberg representation 
$${\rm St}(\lambda,r)=\iota_B^G\big(\delta_B^{-1/2}\otimes {\rm unr}_\lambda|-|^{(n-1)/2}\big)\big/\sum_{B\subsetneq P\subseteq G}\iota_P^G\big(\delta_P^{-1/2}\otimes{\rm unr}_\lambda|-|^{(n-1)/2}\big).$$
Here we view ${\rm unr}_\lambda|-|^{(n-1)/2}$ as a character of $M$ for any (standard) Levi $M$. 
Computing $r_B^G(\iota_P^G(\delta^{-1/2}\otimes {\rm unr}_\lambda|-|^{(n-1)/2}))$ using the geometrical lemma of \cite{BernsteinZele} it follows that $$ r_B^G({\rm St}(\lambda,r))={\rm unr}_{q^{-(r-1)}\lambda}\otimes\dots\otimes {\rm unr}_{q^{-1}\lambda}\otimes {\rm unr}_{\lambda}.$$
This is the character corresponding to the ordering  $q^{-(r-1)} \lambda ,\dots,\lambda q^{-1},\lambda$, i.e.~ to the ordering defined by the unique $(\varphi,N)$-stable flag of ${\rm Sp}(\lambda,r)$.

In the general case we decompose $(\phi,N)\otimes_{k(x)}L=\bigoplus_{i=1}^s {\rm St}(\lambda_i,r_i)$ and 
write $${\rm LL^{mod}}(\phi_x,N_x)\otimes_{k(x)}L=\iota_P^G\big({\rm St}(\lambda_1,r_1)\otimes\dots\otimes{\rm St}(\lambda_{s},r_s)\big)$$ as in $(\ref{defnmodifiedLL})$. Then again the geometrical lemma of \cite{BernsteinZele} computes that its Jacquet-module is the desired one, and the claim follows from compatibility with contragredients. 
\end{proof}
\begin{rem}
For $C=\mathbb{C}$ the lemma can be interpreted as a consequence of the classification of Kazhdan-Lusztig \cite{KL} using equivariant $K$-theory, or its formulation using Borel-Moore homology in \cite[8.1]{ChrissGinz}.
For $(\phi,N)\in X_{\check G}(C)$ as in the lemma the fiber $\tilde\beta_{\check B}^{-1}(\phi,N)$ of $\tilde\beta_{\check B}:\tilde X_{\check B}\rightarrow X_{\check G}$ is identified with the $\phi$-fixed points $\mathfrak{B}_N^\phi$ of the variety $\mathfrak{B}_N$ of $N$-stable complete flags, compare \cite[8.1]{ChrissGinz}. Using the induction theorem \cite[6]{KL} one can deduce that the $\Hcal_G$-module ${\rm LL^{mod}}(\phi,N)$ is precisely the standard module constructed in \cite[Definition 8.1.9]{ChrissGinz}(note that the group $C(\phi,N)$ of loc.~cit.~is trivial in the ${\rm GL}_n$-case). However, as $\phi$ is regular semi-simple the variety $\mathfrak{B}_N^\phi$ is a finite union of points, namely the complete $(\phi,N)$-stable flags. Hence its Borel-Moore homology is the direct sum of copies of $C$ indexed by these points. By construction the Hecke algebra $\Hcal_T$ acts on this direct sum as asserted in the lemma.
\end{rem}

\subsection{The work of Helm and Emerton-Helm}\label{EHfamily}
Emerton and Helm \cite{EmertonHelm} proposed the existence of a family of $G$-representations over a deformation space of $\ell$-adic Galois (or Weil-Deligne) representations that interpolates the modified local Langlands correspondence in a certain sense. A candidate for such a family was constructed in subsequent work of Helm\footnote{in Helm's integral $\ell$-adic set up, the construction of the candidate in \cite{Helm1} is not complete, but depends on a conjecture about the action of the Bernstein center (Conjecture 7.5. of \cite{Helm1}). This conjecture was proven by Helm and Moss \cite{HelmMoss}. In out set up of representations on characteristic $0$, and only considering the Bernstein block defined by $[T,1]$ this conjecture becomes much easier and boils down to Lemma \ref{compatibilityofcenteraction} above.} \cite{Helm1}. 
Rather than working over $\ell$-adic deformation rings we want to work with the stacks of L-parameters defined above. We review the work of Emerton-Helm and Helm in this set up in order to construct a family of $G$-representation on the stack $[X_{\check G}/\check G]$.

In this section we need to work with families of admissible smooth representations, compare \cite[2.1.]{EmertonHelm}. We make precise what we mean by this. Let $A$ be a noetherian $C$-algebra and let $V$ be a finitely generated $A[G]$-module. We say that $V$ is an admissible smooth family of $G$ representations over $A$, if the $G$-representation on $V$ is smooth and if $V^{K'}$ is a finitely generated $A$-module for every compact open subgroup $K'\subset G$. 

Denote by $N\subset B$ the unipotent radical and let $\psi:N\rightarrow C^\times$ be a generic character. 
Recall that an irreducible $G$-representation $\pi$ is called \emph{generic} if there exists an embedding $\pi\hookrightarrow {\rm Ind}_N^G\psi$. Equivalently, $\pi$ is generic if there is a surjection $\cind_N^G\psi \twoheadrightarrow \pi$.

We write $(\cind_N^G\psi)_{[T,1]}$ for the image of the compactly induced representation $\cind_N^G\psi$ in the Bernstein component ${\rm Rep}_{[T,1]}(G)={\rm Rep}^IG$. As in the case of ${\rm GL}_n$ a Whittaker datum is unique up to isomorphism, this representation (up to isomorphism) does not depend on the Whittaker datum $(B,\psi)$. 

Recall that we have fixed $K=\GL_n(\Ocal_F)\supset I$ and consider the induced representation ${\rm Ind}_I^K \mathbf{1}_I$. By \cite{SchneiderZink} this induced representation decomposes into a direct sum 
\begin{equation}\label{SchneiderZinkdecompo}
{\rm Ind}_I^K \mathbf{1}_I=\bigoplus_{\mathcal{P}}\sigma_{\mathcal{P}}^{\oplus m_{\mathcal{P}}}
\end{equation}
indexed by the set of partitions $\mathcal{P}$ of the positive integer $n$ which is partially ordered, see \cite[p.~169]{SchneiderZink}, and has a unique minimal\footnote{Note that the partial ordering used here is the opposite to the standard ordering of partitions, compare \cite[3.]{SchneiderZink}. Here the maximal element is given by $1+1+\dots+1$ and the minimal element is $n$.} element $\mathcal{P}_{\rm min}$ and a unique maximal element $\mathcal{P}_{\rm max}$. Let ${\rm st}={\rm st}_G=\sigma_{\mathcal{P}_{\rm min}}$ denote the finite dimensional Steinberg representation. This representation occurs with multiplicity $m_{\mathcal{P}_{\rm min}}=1$. 
As $\cind_K^G{\rm st}$ lies in the Bernstein component $[T,1]$ it carries a natural action of $\mathfrak{Z}$.

We further recall from \cite[3.2, 3.5]{BernsteinZele} the definition of the $r$-th derivative $V^{(r)}$ of a ${\rm GL}_n(F)$-representation $V$ which is a smooth representation of ${\rm GL}_{n-r}(F)$. In particular $V^{(n)}$ is just a $C$-vector space. 
By \cite[p.5, (2)]{Helm1} there is natural isomorphism 
\begin{equation}\label{(n)adjointness}
{\rm Hom}_G(\cind_N^G\psi, V)\cong  V^{(n)}.
\end{equation}
If $0\neq v\in V^{(n)}$ and $V$ lies in the Bernstein component $[T,1]$, then the morphism defined by $v$ obviously factors through $(\cind_I^G\psi)_{[T,1]}$. 

The following theorem is a summary of the results in \cite[\S\S 3,4]{Helm1} (translated to the easier situation considered here).
\begin{theo}\label{summaryHelm}
Let $\pi$ be one of the representations $(\cind_N^G\psi)_{[T,1]}$ and $\cind_K^G{\rm st}$. Then $\pi$ is a smooth $\mathfrak{Z}$-representation and the $n$-th derivative $\pi^{(n)}$ is a free $\mathfrak{Z}$-modules of rank $1$. 
Moreover, let $\mathfrak{p}\in {\rm Spec}\,\mathfrak{Z}$ then
\begin{itemize}
\item[(a)] the representation $\pi\otimes k(\mathfrak{p})$ is a direct sum of finite length representations.
\item[(b)] the cosocle ${\rm cosoc}(\mathfrak{p})$ of $\pi\otimes k(\mathfrak{p})$ is absolutely irreducible and generic.
\item[(c)] the representation $\ker(\pi\otimes k(\mathfrak{p})\rightarrow {\rm cosoc}(\mathfrak{p}))$ does not contain any generic subquotient.
\end{itemize}
Finally, the representation $\cind_K^G{\rm st}$ is admissible as a $\mathfrak{Z}$-representation.
\end{theo}
\begin{proof}
We cite the proof from \cite{Helm1}. All references in this proof refer to loc.~cit. In Helm's situation the coefficients are $W(k)$ for a finite field $k$, instead of the characteristic zero field $C$ in our case. The arguments literally do not change in our set-up; except for one argument, where the classification of irreducible, smooth mod $\ell$ representations in terms of parabolic induction has to be replaced by the corresponding classification of irreducible, smooth representations in characteristic zero.

The case of $(\cind_I^G\psi)_{[T,1]}$ follows from Lemma 3.2 and Lemma 3.4.
In the case $\pi=\cind_K^G{\rm st}$ admissibility follows from Theorem 4.1, and part (a) is Lemma 4.2.
Properties (b) and (c) are proven in Proposition 4.9. Finally the claim on $\pi^{(n)}$ is Corollary 4.10.

The proof of Helm's Proposition 4.9 uses the classification of irreducible smooth mod $\ell$ representations of ${\rm GL}_n(F)$ in terms of parabolic induction, and has to be replaced by the usual Bernstein-Zelevinsky classification of irreducible, smooth representations in characteristic zero \cite{Zele}. With this change of reference the proof in \cite{Helm1} literally does not change. 
\end{proof}

\begin{cor}\label{indetifycindpsi}
There is an isomorphism of $\mathfrak{Z}[G]$-modules. \[(\cind_N^G\psi)_{[T,1]}\cong \cind_K^G{\rm st},\]
unique up to a scalar in $\mathfrak{Z}^\times$.   
\end{cor}
\begin{proof}
By Theorem \ref{summaryHelm} the $n$-th derivative $(\cind_K^G{\rm st})^{(n)}$ is locally free of rank $1$ over $\mathfrak{Z}$. As $\mathfrak{Z}\cong C[X_1,\dots, X_{n-1},X_n^{\pm 1}]$ every line bundle on $\Spec\mathfrak{Z}$ is trivial and hence $(\cind_K^G{\rm st})^{(n)}\cong \mathfrak{Z}$. By the discussion preceding Theorem \ref{summaryHelm}, a choice of a basis vector (which is unique up to a scalar in $\mathfrak{Z}^\times$) gives rise to a morphism 
 \[\alpha:(\cind_N^G\psi)_{[T,1]}\longrightarrow \cind_K^G{\rm st}.\]
We claim that $\alpha$ is an isomorphism.  

First we show that $\alpha$ is surjective: let $W$ denote the cokernel of $\alpha$. Then $W$ is generated by its Iwahori fixed vectors $W^I$ and, by admissibility of $\cind_K^G{\rm st}$, the $\mathfrak{Z}$-module $W^I$ is finitely generated.\\
As $(-)^I$ is an exact functor $W^I\otimes k(\mathfrak{p})=(W\otimes k(\mathfrak{p}))^I$ and hence $W=0$ if and only if $W\otimes k(\mathfrak{p})=0$ for all $\mathfrak{p}\in {\rm Spec}\, \mathfrak{Z}$. 

As $\alpha$ by definition induces an isomorphism $$\alpha^{(n)}:(\cind_N^G\psi)_{[T,1]}^{(n)}\longrightarrow (\cind_K^G{\rm st})^{(n)}$$ and as the functor $(-)^{(n)}$ is exact (see e.g.~\cite[3.2, Proposition]{BernsteinZele}) it follows that $W^{(n)}=0$ and $(W\otimes k(\mathfrak{p}))^{(n)}=0$ for all $\mathfrak{p}\in{\rm Spec}\, \mathfrak{Z}$. 
Assume that $W\otimes k(\mathfrak{p})\neq 0$. As $W\otimes k(\mathfrak{p})$ is a quotient of $\cind_K^G{\rm st}\otimes k(\mathfrak{p})$, Theorem \ref{summaryHelm} (b),(c) implies that there exists a non-zero morphism
\[\cind_N^G\psi\longrightarrow W\otimes k(\mathfrak{p}),\]
contradicting $(W\otimes k(\mathfrak{p}))^{(n)}=0$

Now $\cind_K^G {\rm st}$ is projective as a $G$-representation and hence the surjection $\alpha$ has a splitting
\[(\cind_N^G\psi)_{[T,1]}\cong \cind_K^G{\rm st}\oplus W'.\]
As $\alpha$ induces an isomorphism after applying the $n$-th derivative $(-)^{(n)}$ it follows that $(W')^{(n)}=0$. By the adjointness property $(\ref{(n)adjointness})$ is follows that the canonical projection
\[\beta:(\cind_N^G\psi)_{[T,1]}\longrightarrow W'\]
is zero and hence $W'=0$, as $\beta$ is surjective.
\end{proof}

Following \cite{Helm1} we construct a family $\mathcal{V}_G$ of $G$-representations on $[X_{\check G}/\check G]$ that conjecturally interpolates the modified local Langlands correspondence (see Conjecture \ref{Conjfibers} below for the precise meaning). 
Rather than constructing $\mathcal{V}_G$ directly on $[X_{\check G}/\check G]$ we construct a family $\tilde{\mathcal{V}}_G$ on $X_{\check G}=:{\rm Spec}\, A_{\check G}$ that is $\check G$-equivariant and hence descents to $[X_{\check G}/\check G]$.

\begin{lem}\label{surjecttoLLmod}
Let $x=(\phi_x,N_x)\in X_{\check G}$. There exists a canonical surjection
\[(\cind_N^G\psi)_{[T,1]}\otimes_{\mathfrak{Z}}k(x)\longrightarrow {\rm LL^{mod}}((\phi_x^{\rm ss},N_x)^{\vee})^\vee\]
that is unique up to scalar.
\end{lem}
\begin{proof}
This follows from the argument in the proof of \cite[Theorem 7.9]{Helm1}, using Lemma \ref{compatibilityofcenteraction} instead of Conjecture 7.5.~of loc.~cit..
\end{proof}

Let $\eta=(\phi_\eta,N_\eta)\in X_{\check G}$ be a generic point. Then $${\rm LL}(\phi_\eta,N_\eta)={\rm LL^{mod}}(\phi_\eta,N_\eta)={\rm LL^{mod}}((\phi_\eta,N_\eta)^\vee)^\vee$$ is an irreducible generic representation. We obtain a morphism
\[(\cind_N^G\psi)_{[T,1]}\otimes_{\mathfrak{Z}}A_{\check G}\longrightarrow (\cind_N^G\psi)_{[T,1]}\otimes_{\mathfrak{Z}}k(\eta)\longrightarrow {\rm LL}(\phi_\eta,N_\eta),\]
where the second morphism is the choice of a surjection as in Lemma \ref{surjecttoLLmod}.

Let $\eta_i$, $i\in I$ denote the generic points of $X_{\check G}=\Spec A_{\check G}$. We define $\tilde{\mathcal{V}}_{G}$ to be the (admissible smooth) family of $G$-representations over $A_{\check G}$ that is the image of the morphism
\begin{equation}\label{defnofuniversalfamily}
(\cind_N^G\psi)_{[T,1]}\otimes_{\mathfrak{Z}}A_{\check G}\longrightarrow \prod_{i\in I}{\rm LL}(\phi_{\eta_i},N_{\eta_i}).
\end{equation}
Up to isomorphism, this image does not depend on the choice of the surjection $(\cind_N^G\psi)_{[T,1]}\otimes_{\mathfrak{Z}}k(\eta_i)\rightarrow {\rm LL}(\phi_{\eta_i},N_{\eta_i})$.
By abuse of notation we also write $\tilde{\mathcal{V}}_G$ for the corresponding sheaf on $X_{\check G}$.  

It can easily be seen that $\tilde{\mathcal{V}}_G$ is a $\check G$-equivariant quotient of $(\cind_N^G\psi)_{[T,1]}\otimes_{\mathfrak{Z}}A_{\check G}$ (equipped with the obvious $\check G$-equivariant structure). Hence $\tilde{\mathcal{V}}_G$ descents to a quasi-coherent sheaf $\mathcal{V}_G$ on $[X_{\check G}/\check G]$ that carries an action of $G$. We will often refer to this family of $G$-representation as the Emerton-Helm family.
Conjecturally this family interpolates the modified local Langlands correspondence:
\begin{conj}[compare \cite{EmertonHelm}]\label{Conjfibers}
Let $x=(\phi,N)\in X_{\check G}$ be any point, then
\[(\tilde{\mathcal{V}}_G\otimes k(x))^\vee\cong {\rm LL}^{\rm mod}((\phi,N)^\vee).\]
\end{conj}

\subsection{Idempotents in the Iwahori-Hecke algebra}
We will describe the family of Hecke modules associated to the Emerton-Helm family $\mathcal{V}_G$ in the next subsection, and relate this construction to Conjecture \ref{mainconjecture}. Before we do so, we need some preparation about idempotent elements in the Iwahori-Hecke algebra. 

Let $J\subset G$ be a compact open subgroup and $(\lambda,W)$ be a smooth representation of $J$ on a finite dimensional $C$-vector space with contragredient representation $(\lambda^\vee, W^\vee)$. 
Then we have a natural identification of $C$-algebras
\begin{equation}\label{Heckealgcompsuppfct}
\begin{aligned}
{\rm End}_G(\cind_J^G\lambda)\cong\left\{\begin{array}{*{20}c} \text{compactly supported}\ f:G\rightarrow {\rm End}_C(W^\vee)\\ \text{such that}\ f(j_1gj_2)=\lambda^\vee(j_1)\circ f(g)\circ \lambda^\vee(j_2)\\ \text{for all}\ g\in G,\, j_1,j_2\in J\end{array}\right\},
\end{aligned}
\end{equation}
where, as usual, the algebra structure on the right hand side is given by convolution. 
Given $f\in \Hcal(G,\lambda)$ one defines $\check f:g\mapsto f(g^{-1})^\vee\in{\rm End}_C(W)$. Then $f\mapsto \check f$ induces an isomorphism of $C$-algebras
\[\Hcal(G,\lambda)\cong \Hcal(G,\lambda^\vee)^{\rm op}.\]

Recall that $\mathcal{H}_G={\rm End}_G(\cind_I^G\mathbf{1}_I)={\rm End}_G(\cind_K^G V)$, where $V={\rm Ind}_I^K\mathbf{1}_I$. From now on we write $\lambda$ for the $K$-representation on $V$. 
As in $(\ref{SchneiderZinkdecompo})$ the representation $V$ decomposes as a direct sum of the representations isomorphic to $\sigma_\mathcal{P}$. Note that $$V={\rm Ind}_{B(k)}^{\rm GL_n(k)}\mathbf{1}_{B(k)},$$
where $B(k)\subset {\rm GL}_n(k)$ is the special fiber of the Borel subgroup and $K$ acts via the quotient map $K\rightarrow\GL_n(k)$.
For a partition $\mathcal{P}$ we write $\Sigma_\mathcal{P}\subset V$ for the $\sigma_\mathcal{P}$-isotypical component of $V$. In particular we have $\Sigma_\mathcal{P}\cong \sigma_\mathcal{P}^{m_\mathcal{P}}$.
The direct summand $\cind_K^G\Sigma_\mathcal{P}$ of $\cind_K^GV=\cind_I^G\mathbf{1}_I$ defines an idempotent element $e_\mathcal{P}\in \Hcal_G$. 
If $\mathcal{P}=\mathcal{P}_{\rm min}$ we will usually write $e_{\rm st}$ (or $e_{G,\rm st}$ if we need to refer to the group $G$) instead of $e_{\mathcal{P}_{\rm min}}$. Further we usually write $e_K=e_{\mathcal{P}_{\rm max}}$, which is identified with the characteristic function of $K$. 

Using the description of the Hecke algebra $(\ref{Heckealgcompsuppfct})$ the idempotent elements $e_\mathcal{P}$ can be described as follows.
Let $f_\mathcal{P}:V^\vee\rightarrow V^\vee$ denote the endomorphism that is the identity on $\Sigma_\mathcal{P}^\vee$ and zero on $\Sigma_{\mathcal{P}'}^\vee$ for $\mathcal{P}'\neq\mathcal{P}$.
Then the idempotent element $e_\mathcal{P}$ is defined by
\[e_\mathcal{P}:g\longmapsto\begin{cases}0,& g\notin K\\ \lambda^\vee(g)\circ f_\mathcal{P}=f_\mathcal{P}\circ \lambda^\vee(g),& g\in K\end{cases}.\]

Note that the representation $V={\rm Ind}_K^G\mathbf{1}_I$  and the irreducible representations $\sigma_\mathcal{P}$ are self-dual.
In the case of $V$ this follows from the computation of the smooth dual of an induced representation. In particular, the canonical identification $\mathbf{1}_I=(\mathbf{1}_I)^\vee$ gives a canonical isomorphism $\alpha:V\rightarrow V^\vee$. 
In the case of $\sigma_\mathcal{P}$ we proceed by descending induction: the claim is obviously true for $\mathbf{1}_K=\sigma_{\mathcal{P}_{\rm max}}$ and for each $\mathcal{P}$ we can find some (integral model of a) parabolic subgroup $\mathbb{P}\subset {\rm GL}_n$ such that 
$${\rm Ind}_{\mathbb{P}(k)}^{{\rm GL}_n(k)}\mathbf{1}\cong \sigma_\mathcal{P}\oplus\bigoplus_{\mathcal{P}\preceq\mathcal{P}'\neq \mathcal{P}}\sigma_{\mathcal{P}'}^{\oplus a_{\mathcal{P}'}}$$
for some integers $a_{\mathcal{P}'}$. As the induced representation on the left hand is self-dual so must be $\sigma_\mathcal{P}$.

It follows that we can identify $\Hcal_G=\Hcal_G(V,\lambda)$ with $\Hcal_G(V,\lambda^\vee)$. In particular we obtain a canonical isomorphism $\Hcal_G\cong\Hcal_G^{\rm op}$.
\begin{lem}\label{dualofidempotents}
Let $\mathcal{P}$ be a partition. Then $\check e_{\mathcal{P}}=e_{\mathcal{P}}$. 
\end{lem}
\begin{proof}
The canonical isomorphism $\alpha$ allows us to identify ${\rm End}_C(V^\vee,V^\vee)$ with ${\rm End}_C(V,V)$ and $\Hcal_G=\Hcal(G,\lambda)$ with $\Hcal(G,\lambda^\vee)$. 
By definition  $\check e_\mathcal{P}$ is the element $(g\mapsto e_\mathcal{P}(g^{-1})^\vee)\in\Hcal(G,\lambda^\vee)=\Hcal(G,\lambda)$ under this identification. 
We calculate that
\[\check e_\mathcal{P}(g)=\begin{cases}0,&g\notin K\\ f_\mathcal{P}^\vee\circ (\lambda^\vee(k^{-1}))^\vee=f_\mathcal{P}^\vee\circ \lambda(k),&g\in K.\end{cases}\]
Here $f_\mathcal{P}^\vee$ is the idempotent endomorphism of $V$ defined by the direct summand $\Sigma_\mathcal{P}$.
As the $\sigma_\mathcal{P}$ are self-dual the isomorphism $\alpha$ maps $\Sigma_{\mathcal{P}}$ to $\Sigma^\vee_{\mathcal{P}}$. 
Hence we conclude that (under the identification ${\rm End}_C(V^\vee,V^\vee)={\rm End}_C(V,V)$ using $\alpha$) the element $\check e_\mathcal{P}$ equals $e_\mathcal{P}$.
\end{proof}

Recall that $\Hcal_G$ contains the finite Hecke algebra 
\begin{align*}
\Hcal_{G,0}=\mathcal{C}_c^\infty(I\backslash K/I)&=\{f:\mathbb{B}(k)\backslash {\rm GL}_n(k)/\mathbb{B}(k)\rightarrow C\}\\ &=\left\{f:K\rightarrow {\rm End}_CV\left| \begin{array}{*{20}c} f(k_1kk_2)=\lambda(k_1)\circ f(k)\circ \lambda(k_2)\\ \text{for all}\ k,k_1,k_2\in K\end{array}\right.\right\}\\&={\rm End}_K(V)
\end{align*}
as a subalgebra. This algebra contains the idempotent elements $e_\mathcal{P}$. 
Further recall that for a parabolic subgroup $\mathbb{P}\subset \mathbb{G}$ containing $\mathbb{B}$ we have an embedding $\Hcal_M\hookrightarrow \mathcal{H}_G$ of Hecke algebras, where $M=\mathbb{M}(F)$ is the Levi of $\mathbb{P}$.
If $\mathbb{P}=\mathbb{B}$ this gives an embedding 
\[C[X_\ast(\mathbb{T})]=\Hcal_T\hookrightarrow \mathcal{H}_G.\]
By \cite[Lemma 1.7.1]{HainesKottwitzPrasad} the morphism 
\begin{equation}\label{HTbasis}
\Hcal_T\otimes_C\Hcal_{G,0}\longrightarrow \Hcal_G
\end{equation}
induced by multiplication is an isomorphism of $C$-vector spaces. 
\begin{lem}\label{HGestfree}
(i) The canonical inclusion $\Hcal_T e_{G,\rm st}\subset \Hcal_G e_{G,\rm st}$ is an equality. Moreover this module is free of rank $1$ with basis $e_{G,\rm st}$ as an $\Hcal_T$-module.\\
\noindent (ii) Let $\mathbb{P}\subset \mathbb{G}$ be a parabolic as above and let $M=\mathbb{M}(F)\subset G$ be the corresponding Levi subgroup. The isomorphism 
\[\Hcal_Me_{M,\rm st}=\Hcal_T e_{M,\rm st}\longrightarrow \Hcal_Ge_{G,\rm st}=\Hcal_Te_{G,\rm st}\] of free $\Hcal_T$-modules of rank $1$ defined by $e_{M,\rm st}\mapsto e_{G,\rm st}$ is an $\Hcal_M$-module homomorphism.
\end{lem}

\begin{proof}
(i) It directly follows from $(\ref{HTbasis})$ that $\Hcal_Te_{G,\rm st}$ is free of rank $1$ as an $\Hcal_T$-module.
Moreover, note that 
\[\Hcal_{G,0}e_{G,\rm st}=({\rm st}_G)^{\mathbb{B}(k)}\]
is a $1$-dimensional $C$-vector space. 
This implies that $f\in\Hcal_G$ can be written as $f=f_0e_{\rm st} +f_1(1-e_{\rm st})$ with $f_0\in \Hcal_T$ and $f_1\in \Hcal_G$. 
It follows that $$f e_{G,\rm st}=f_0e_{G,\rm st}+f_1(1-e_{G,\rm st})e_{G,\rm st}=f_0e_{G,\rm st}\in\Hcal_Te_{G,\rm st}.$$
(ii) As the inclusion $\Hcal_Te_{G,\rm st}\subset \Hcal_Ge_{G,\rm st} $ is an equality, we also have an equality $\Hcal_Me_{G,\rm st}=\Hcal_Ge_{G,\rm st}$. Therefore it is enough to show that the $\Hcal_M$-module homomorphism
\[\Hcal_M\longrightarrow \Hcal_M e_{G,\rm st}\] mapping $1$ to $e_{G,\rm st}$ factors through $\Hcal_M\rightarrow \Hcal_Me_{M,\rm st}$. That is, we need to show $(1-e_{M, \rm st})e_{G,\rm st}=0$ in $\Hcal_G$. 
We can check this equality in the subalgebra $\Hcal_{G,0}$.
Translating the claim back to representation theory it comes down to the claim that 
$${\rm Ind}_{\mathbb{P}(k)}^{{\rm GL}_n(k)}{\rm st}_M\subset {\rm Ind}_{\mathbb{P}(k)}^{{\rm GL}_n(k)}\big({\rm Ind}_{\mathbb{B}_M(k)}^{\mathbb{M}(k)}\mathbf{1}\big)={\rm Ind}_{\mathbb{B}(k)}^{{\rm GL}_n(k)}\mathbf{1} $$
contains the direct summand ${\rm st}_G$, where $\mathbb{B}_M=\mathbb{B}\cap\mathbb{M}$ is a Borel in $\mathbb{M}$. This is true, as ${\rm st}_G$ is the only constituent of the right hand side that does not occur in any parabolically induced representation for a parabolic strictly larger than $\mathbb{B}$.
\end{proof}

\begin{cor}\label{identifyquotofcindNG}
Let $x=(\phi_x,N_x)\in X_{\check G}$ with $\phi_x$ regular semi-simple and let $L$ be an extension of $k(x)$ containing all the eigenvalues of $\phi_x$. Then
\[\big((\cind_N^G\psi)_{[T,1]}\otimes_{\mathfrak{Z}}k(x)\big)^I\cong(\Hcal_Ge_{G,\rm st})\otimes_\mathfrak{Z}k(x)\]
and after extending scalars to $L$ its Jacquet-module is given by
\[r^G_B((\cind_N^G\psi)_{[T,1]}\otimes_{\mathfrak{Z}}L)=\bigoplus_\mathcal{F} {\rm unr}_{\mathcal{F}},\]
where the sum is indexed by the $\phi_x$-stable flags $\mathcal{F}$ of $L^n$.
Moreover, the kernel of the quotient map of $\Hcal_T$-modules
\begin{equation}\label{eq1identifyLLquotient}
(\Hcal_Ge_{G,\rm st})\otimes_\mathfrak{Z}L=\bigoplus_{\phi_x\text{-}{\rm stable}\ \mathcal{F}}{\rm unr}_{\mathcal{F}}\longrightarrow \bigoplus_{(\phi_x,N_x)\text{-}{\rm stable}\ \mathcal{F}}{\rm unr}_{\mathcal{F}}
\end{equation}
is $\Hcal_G$-stable and the induced $\Hcal_G$-module structure on the quotient identifies the right hand side with the $I$-invariants of (the scalar extension to $L$ of) the quotient ${\rm LL^{mod}}((\phi_x,N_x)^\vee)^\vee$ in Lemma \ref{surjecttoLLmod}.
\end{cor}
\begin{proof}
The first claim is a direct consequence of $(\cind_N^G\psi)_{[T,1]}\cong \cind_K^G{\rm st}_G$ and the identification $$(\cind_K^G{\rm st}_G)^I=\Hcal_Ge_{G,\rm st}.$$ The claim on the Jacquet-module follows from $\Hcal_Ge_{G,\rm st}=\Hcal_Te_{G,\rm st}$ and $(\ref{Jacquetforget})$.

For the second part, note that the right hand side in $(\ref{eq1identifyLLquotient})$ is uniquely determined as an $\Hcal_T$-module, as the characters ${\rm unr}_{\mathcal{F}}$ are pairwise distinct.  
Hence it is enough to prove that the quotient
\[\big((\cind_N^G\psi)_{[T,1]}\otimes_{\mathfrak{Z}}k(x)\big)^I\longrightarrow \big({\rm LL^{mod}}((\phi_x,N_x)^\vee)^\vee\big)^I\] 
given by Lemma \ref{surjecttoLLmod} induces this quotient map on the underlying $\Hcal_T$-modules (and after extending scalars to $L$). 
This is a consequence of the computation of $r^G_B({\rm LL^{mod}}((\phi_x,N_x)^\vee)^\vee\otimes_{k(x)}L)$, see Lemma \ref{JBofLLmod}.
\end{proof}

We finish this subsection by recalling some easy facts about the passage from left to right modules over $\Hcal_G$. 
Given a left $\Hcal_G$-module $\pi$ one can view $\pi$ as a right $\Hcal_G$-module via the isomorphism $\Hcal_G\cong\Hcal_G^{\rm op}$. We write ${}^t\pi$ for this right module structure on $\pi$.
\begin{lem}\label{lefttorightinduction}
Let $M\subset G$ be a Levi and let $\pi$ be a left $\Hcal_M$-module. Then there is a canonical and functorial isomorphism of right $\Hcal_G$-modules
\[{}^t(\Hcal_G\otimes_{\Hcal_M}\sigma)\cong {}^t\sigma\otimes_{\Hcal_M}\Hcal_G,\]
where the $\Hcal_G$-module structure on the right hand side is given by right multiplication.
\end{lem}
\begin{proof}
It is easily checked that $\phi\otimes v\longmapsto v\otimes \check \phi$ defines the desired isomorphism.
\end{proof}

\begin{lem}\label{technicalidentities}
Let $\pi$ be an $\Hcal_G$-module and let $e\in\Hcal_G$ be an idempotent element. \\
(i) There is a canonical equality of $\mathfrak{Z}$-modules
\[{\rm Hom}_{\Hcal_G}(\Hcal_Ge,\pi)=e\pi=e\Hcal_G\otimes_{\Hcal_G}\pi.\]
(ii) There is a canonical identification ${}^t(\Hcal_Ge)=\check e\Hcal$ as $\Hcal_G$ right modules. \\
(iii) Let $\mathcal{P}$ be a partition. Then $${\rm Hom}_{\Hcal_G}(\Hcal_Ge_\mathcal{P},\pi)={}^t(\Hcal_Ge_\mathcal{P})\otimes_{\Hcal_G}\pi.$$
(iv) For two partitions $\mathcal{P},\mathcal{P}'$ we have 
\[e_{\mathcal{P}}\Hcal_Ge_{\mathcal{P}'}\cong \mathfrak{Z}^{m_{\mathcal{P}}^2m_{\mathcal{P}'}^2}.\]
\end{lem}
\begin{proof}
Part (i) and (ii) are obvious, and (iii) is a direct consequence of (i), (ii) and $\check e_{\mathcal{P}}=e_{\mathcal{P}}$. 
Finally we find
\begin{align*}
e_{\mathcal{P}}\Hcal_Ge_{\mathcal{P}'}={\rm Hom}_{\Hcal_G}(\Hcal_G e_{\mathcal{P}},\Hcal_Ge_{\mathcal{P}'})&={\rm Hom}_G(\cind_K^G\Sigma_\mathcal{P},\cind_K^G\Sigma_{\mathcal{P}'})\\
&={\rm Hom}_G(\cind_K^G\sigma_\mathcal{P},\cind_K^G\sigma_{\mathcal{P}'})^{m_{\mathcal{P}}m_{\mathcal{P}'}}.
\end{align*}
Now (iv) follows from \cite[Theorem 1.4]{Pyvovarov1}.
\end{proof}

\subsection{The Hecke-module of the interpolating family}
In subsection \ref{EHfamily} we constructed a family of $G$-representations $\mathcal{V}_G$ on the stack $[X_{\check G}/\check G]$. Let $$\mathcal{M}_G=(\mathcal{V}_G)^I$$ denote the corresponding module over the Iwahori-Hecke algebra. 
We write $\tilde{\mathcal{M}}_G$ for the corresponding $\check G$-equivariant sheaf of $\Ocal_{X_{\check G}}\otimes_{\mathfrak{Z}}\Hcal_G$-modules on $X_{\check G}$. 

We write $A_{\check G}$ for the coordinate ring of $X_{\check G}$. Similarly, given a Levi-subgroup $\mathbb{M}\subset \mathbb{G}$ we write $A_{\check M}$ for the coordinate ring of $X_{\check M}$.
Recall that we have embeddings $\Hcal_M\hookrightarrow \Hcal_G$ and a canonical isomorphism $\Hcal_T=\mathfrak{Z}_T\cong A_{\check T}$.

Recall the following diagram of $C$-schemes from section \ref{spacesofLparam}.
\[
\begin{xy}
\xymatrix{
&&\tilde{\mathbf{X}}_{\check B} \ar[dll]_{\tilde \beta_B} \ar[dl]^{\gamma}\ar[ddl]\\ 
X_{\check G} \ar[d]  & X_{\check G}\times_{\check T/W}\check T\ar[d] \ar[l]_{\beta'}\\
\check T/W & \check T. \ar[l]}
\end{xy}
\]
Assuming Conjecture $\ref{conjdegzero}$ the complex $$R\tilde\beta_{B,\ast}\Ocal_{\tilde{\mathbf{X}}_{\check B}}$$ is concentrated in degree zero and, as $\beta'$ is affine, so is $R\tilde\gamma_\ast \Ocal_{\tilde{\mathbf{X}}_{\check B}}$. Hence the formulation of the following conjecture makes sense.
\begin{conj}\label{GLnsurjection}
Let $\mathbb{G}={\rm GL}_n$ and assume Conjecture $\ref{conjdegzero}$. Then the canonical map 
\[\Ocal_{X_{\check G}\times_{\check T/W}\check T} \longrightarrow R\gamma_\ast \Ocal_{\tilde{\mathbf{X}}_{\check B}}\] is a surjection.  
\end{conj}
Note that Conjecture \ref{GLnsurjection} would imply that we have a canonical surjection
\begin{equation}\label{surjectiontoRbetaast}
\Ocal_{X_{\check G}}\otimes_{\mathfrak{Z}}\Hcal_T=\Ocal_{X_{\check G}}\otimes_{\mathfrak{Z}}A_{\check T}=\beta'_\ast(\Ocal_{X_{\check G}\times_{\check T/W}\check T})\longrightarrow R\tilde\beta_{B,\ast}\Ocal_{\tilde{\mathbb{X}}_{\check B}}.
\end{equation}

\begin{rem}
We point out that Conjecture $\ref{GLnsurjection}$ is a conjecture for the group ${\rm GL}_n$ and will fail for other groups. In fact it already fails for ${\rm SL}_2$ and its failure seems to be related to the existence of non-trivial L-packets. 
We refer to Example \ref{exampleSL2} below for a discussion of this point. 
\end{rem}

The restriction of the above diagram to the regular locus yields the diagram
\[\begin{xy}
\xymatrix{
 && \tilde X_{\check B}^{\rm reg}\ar[dl]_\gamma\ar[ddl]\ar[dll]_{\tilde\beta_B}\\
X_{\check G}^{\rm reg} \ar[d]& X_{\check G}^{\rm reg}\times_{\check T/W}\check T\ar[d]\ar[l]_{\beta'}\\
\check T/W & \ar[l]\check T.
}
\end{xy}\]

\begin{theo}
Conjecture \ref{GLnsurjection} is true over the regular locus. That is the canonical morphism
\[\Ocal_{X_{\check G}^{\rm reg}\times_{\check T/W}\check T} \longrightarrow R\gamma_\ast \Ocal_{\tilde{{X}}^{\rm reg}_{\check B}}\] is a surjection.  
Moreover, the restriction of the induced surjection $(\ref{surjectiontoRbetaast})$ to the regular locus and the identification
\[\Ocal_{X_{\check G}}\otimes_{\mathfrak{Z}}\Hcal_T=\Ocal_{X_{\check G}}\otimes_{\mathfrak{Z}}\Hcal_T e_{G,{\rm st}}=\Ocal_{X_{\check G}}\otimes_{\mathfrak{Z}}\Hcal_G e_{G,{\rm st}}\] 
equips $R\tilde\beta_{B,\ast}\Ocal_{\tilde X_{\check B}^{\rm reg}}$ with the structure of an $\mathcal{O}_{X_{\check G}^{\rm reg}}\otimes_{\mathfrak{Z}}\Hcal_G$-module
that identifies this sheaf with the the restriction $\tilde{\mathcal{M}}_G|_{X_{\check G}^{\rm reg}}$ of the Hecke-module defined by the Emerton-Helm family. 
\end{theo}
\begin{proof}
This follows from Lemma \ref{lemmagammaclosedimmersion} and Proposition \ref{identifyMandOB} below. 
\end{proof}
\begin{rem}
We point out that proving Conjecture $\ref{GLnsurjection}$ would automatically imply that the identification of $R\tilde\beta_{B,\ast}\Ocal_{\tilde{\mathbf{X}}_{\check B}}$ with the Hecke module underlying the Emerton-Helm family holds true without restricting to the regular locus. 
\end{rem}

\begin{lem}\label{lemmagammaclosedimmersion}
The morphism 
\[\gamma:\tilde X_{\check B}^{\rm reg}\longrightarrow X_{\check G}^{\rm reg}\times_{\check T/W}\check T\]
is a closed immersion. 
\end{lem}
\begin{proof}
Clearly $\gamma$ is a finite morphism, by the very definition of the regular locus. Hence it is enough to show that $\gamma$ induces an injection on $k$-valued points, for algebraically closed fields $k$, and a surjection on complete local rings. 

Let $k$ be an algebraically closed extension of $C$ and let $(A,\mathfrak{m})$ be a local Artinian $C$-algebra with residue field $k$. 
Let $(\phi,N)\in X_{\check G}^{\rm reg}(A)$ and let $\lambda_1,\dots,\lambda_n\in A$.  Then we have to show that there is at most one complete flag $\mathcal{F}_\bullet$ of $A^n$ stable under $\phi$ and $N$ such that $\phi$ acts on $\Fcal_i/\Fcal_{i-1}$ by multiplication with $\lambda_i$.

Assume first $A=k$. We prove the claim by induction on $n$. The case $n=1$ is trivial.  
Assume the claim is true for $n-1$. Then it is enough to show that there is a unique $(\phi,N)$-stable line $\mathcal{F}_1$ in $k^n$ on which $\phi$ acts by multiplication with $\lambda_1$. Obviously this forces $\mathcal{F}\subset \ker N$ and we need to show that the $\phi$-eigenspace in $\ker N$ of eigenvalue $\lambda_1$ is one dimensional. However, if this is not the case then there are infinitely many pairwise distinct $(\phi,N)$-stable lines in $k^n$, and each can be completed to a complete $(\phi,N)$-stable flag. This contradicts the regularity of  $(\phi,N)$. 

Now assume that $(A,\mathfrak{m})$ is a general Artinian $C$-algebra with residue field $k$. 
Again it suffices to show that there is a unique $(\phi,N)$-stable $A$-line in $A^n$, such that the quotient of $A^n$ by this line is free, on which $\phi$ acts as multiplication by $\lambda_1$.
By induction on the length of $A$ we can reduce to the following situation:
there exists $f\in A$ such that $\mathfrak{m}f=0$, and if $A'=A/(f)$ and $(\phi',N')$ is the image of $(\phi,N)$ in $X_{\check G}^{\rm reg}(A')$, then there is a unique $(\phi',N')$ stable $A'$-line in $A'^n$ on which $\phi'$ acts by multiplication with $\lambda_1\ {\rm mod}\,(f)$. Let $(\bar\phi,\bar N)\in X_{\check G}^{\rm reg}(k)$ denote the reduction of $(\phi,N)$ modulo $\mathfrak{m}$ and let $\bar\lambda_1\in k$ denote the reduction of $\lambda_1$. 
Then the multiplication with $f$ induces an embedding of $k^n\hookrightarrow A^n$ of $(\phi,N)$-modules with cokernel $A'^n$.
Assume that $\mathcal{F}_1=Ae_1$ and $\mathcal{F}'_1=Ae'_1$ are two $(\phi,N)$-stable $A$-lines on which $\phi$ acts by multiplication with $\lambda_1$.  
Then the assumption implies $e'_1=\alpha e_1+fv$ for some $\alpha\in A^\times$ and $v\in k^n$. Let $\bar e_1\in k^n$ denote the reduction of $e_1$ modulo $\mathfrak{m}$, then it remains to show $v\in k\bar e_1$. 
As $\phi(e_1)=\lambda_1 e_1$ and $\phi(e'_1)=\lambda_1 e'_1$ we deduce $\bar\phi(v)=\bar\lambda_1v$. The discussion of the case of an algebraically closed field $k$ above implies that it is enough to prove that $v\in \ker\bar N$.
However, we assume that $\Fcal_1$ and $\Fcal'_1$ are defined by points $(\phi,N,\Fcal_\bullet), (\phi,N,\Fcal'_\bullet)\in X_{\check G}^{\rm reg}(A)$. As $X_{\check G}^{\rm reg}$ is reduced by Lemma \ref{reducedparaboliccase} and as $N$ is nilpotent we deduce that $N(\Fcal_1)=N(\Fcal'_1)=0$ and hence $N(fv)=0$ which implies $\bar N(v)=0$. 
\end{proof}

We use the lemma to identify $\tilde X_{\check B}^{\rm reg}$ with a closed subscheme $Y_{\check G}^{\rm reg}$ of $X_{\check G}^{\rm reg}\times_{\check T/W}\check T$.
We denote by $$Y_{\check G}\subset X_{\check G}\times_{\check T/W}\check T=\Spec(A_{\check G}\otimes_{\mathfrak{Z}}A_{\check T})$$ the closure of $Y_{\check G}^{\rm reg}$ equipped with its canonical scheme structure (which is the reduced structure, as $\tilde X_{\check B}^{\rm reg}$ is reduced). Let us write $\tilde A_{\check G}$ for the corresponding quotient of $A_{\check G}\otimes_{\mathfrak{Z}}A_{\check T}$ and $\beta:Y_{\check G}\rightarrow X_{\check G}$ for the canonical projection. 

We can use Lemma \ref{HGestfree} to equip $$A_{\check G}\otimes_{\mathfrak{Z}}A_{\check T}=A_{\check G}\otimes_{\mathfrak{Z}}\Hcal_T\cong A_{\check G}\otimes_{\mathfrak{Z}}\Hcal_Te_{G,\rm st}=A_{\check G}\otimes_{\mathfrak{Z}}\Hcal_Ge_{G,\rm st}$$ with an $\Hcal_G$-module structure.
\begin{prop}\label{identifyMandOB}
(i) The kernel of the canonical morphism $A_{\check G}\otimes_{\mathfrak{Z}}A_{\check T}\rightarrow \tilde A_{\check G}$ is stable under the action of $\Hcal_G$. 
\\
(ii) There is a canonical isomorphism
\[\tilde{\mathcal{M}}_G\cong \beta_\ast \Ocal_{Y_{\check G}}\]
of $\check G$-equivariant $\Ocal_{X_{\check G}}\otimes_{\mathfrak{Z}}\Hcal_G$-modules.
\end{prop}
\begin{proof}

(i) By Lemma \ref{equidimparabolicregular} and Lemma \ref{reducedparaboliccase} the scheme $Y_{\check G}$ is reduced and every irreducible component of $Y_{\check G}$ dominates an irreducible component of $X_{\check G}$. 
In particular the canonical morphism
\[\tilde A_{\check G}\longrightarrow \prod_{\eta}\tilde A_{\check G}\otimes k(\eta)=\prod_{\eta}\Gamma(\tilde\beta_B^{-1}(\eta),\Ocal_{\tilde\beta_B^{-1}(\eta)})\]
is an injection. Here the product runs over all generic points $\eta$ of $X_{\check G}$.
It is therefore enough to prove that for all generic points $\eta$ of $X_{\check G}$ the kernel of the canonical map
\[k(\eta)\otimes_{\mathfrak{Z}}\Hcal_G e_{G, \rm st}=k(\eta)\otimes_{\mathfrak{Z}}\Hcal_T e_{G, \rm st}=k(\eta)\otimes_{\mathfrak{Z}}A_{\check T}\longrightarrow \Gamma((\tilde\beta_B^{-1}(\eta),\Ocal_{\tilde\beta_B^{-1}(\eta)})\]
is stable under the $\Hcal_G$-action. This follows from Corollary \ref{identifyquotofcindNG} applied to the generic point $\eta=(\phi_\eta,N_\eta)$.

\noindent (ii) Consider the diagram
\[
\begin{xy}
\xymatrix{
(\cind_N^G\psi)_{[T,1]}^I\otimes_{\mathfrak{Z}}A_{\check G}\ar@{->>}[r] \ar[d]^\cong&\Gamma(X_{\check G},\tilde{\mathcal{M}}_G)\ar@{^{(}->}[r] &\prod_{\eta}  {\rm LL}(\phi_\eta,N_\eta)^I\ar[d]^{\cong}\\
A_{\check T}\otimes_{\mathfrak{Z}}A_{\check G} \ar@{->>}[r] &\tilde A_{\check G} \ar@{^{(}->}[r] & \prod_{\eta} \Gamma(\tilde\beta_B^{-1}(\eta),\Ocal_{\tilde\beta_B^{-1}(\eta)}),
}
\end{xy}
\]
where the left vertical arrow comes from the identification of $$(\cind_N^G\psi)_{[T,1]}^I=\Hcal_G e_{G,\rm st}=\Hcal_T e_{G,\rm st}\cong A_{\check T}$$
and the right vertical arrow comes from the identification of the Jacquet-module of ${\rm LL}(\phi_\eta,N_\eta)$ in Corollary \ref{identifyquotofcindNG}. 
By construction the diagram is a commutative diagram of $A_{\check G}\otimes_{\mathfrak{Z}}\Hcal_{\check G}$-modules and moreover all morphisms are compatible with the $\check G$-action. Hence these morphisms induce a canonical isomorphism $$\Gamma(X_{\check G},\tilde{\mathcal{M}}_G)\cong \tilde A_{\check G}$$ as claimed.
\end{proof}

As a consequence we can easily deduce Conjecture \ref{Conjfibers} for regular semi-simple points.
\begin{cor}
Let $x=(\varphi,N)\in X_{\check G}$ with $\varphi$ regular semi-simple. Then
\[(\tilde{\mathcal{V}}_G\otimes k(x))^\vee\cong {\rm LL}^{\rm mod}((\phi,N)^\vee).\]
\end{cor}
\begin{proof}
It follows from the proof of Proposition \ref{identifyMandOB} that 
\[((\cind_N^G\psi)_{[T,1]}\otimes_\mathfrak{Z}k(x))^I\longrightarrow \tilde{\mathcal{M}}_G\otimes k(x)\cong\Gamma(\tilde\beta^{-1}(x),\Ocal_{\tilde\beta^{-1}(x)})\]
is a surjection of $\Hcal_T\otimes k(x)$-modules. 
The claim now follows from Corollary \ref{identifyquotofcindNG}.
\end{proof}

The module of $I$-invariants in the family of smooth representation over $\ell$-adic deformation rings proposed by Emerton and Helm \cite{EmertonHelm} (and constructed by Helm in \cite{Helm1}) is expected to have a close relation with patched modules in the Taylor-Wiles method (for Iwahori-level at $p$). 
In fact the Taylor-Wiles patching modules automatically produces maximal Cohen-Macaulay modules (which are in fact self-dual for Grothendieck-Serre duality).
The family defined in \cite{Helm1} is related to the family $\mathcal{V}_G$ by the twist with $|{\rm det}|^{-(n-1)/2}$ and by some flat base changes. 
This motivates the following conjecture which also would be a direct consequence of Conjecture \ref{GLnsurjection} and the self-duality statement in Conjecture \ref{conjdegzero}

\begin{conj} \label{ConjEHCM}
The Hecke module $\tilde{\mathcal{M}}_G=(\tilde{\mathcal{V}}_G)^I$ underlying the Emerton-Helm family $\tilde{\mathcal{V}}_G$ on $X_{\check G}$ is a Cohen-Macaulay module over $\mathcal{O}_{X_{\check G}}$. 
\end{conj}

\begin{rem}
One deduces easily from Proposition \ref{identifyMandOB} that $\tilde{\mathcal{M}}_G$ can not be flat as an $\Ocal_{X_{\check G}}$-module. 
On the other hand, as explained above the family $\tilde{\mathcal{M}}_G$ should have some relation with patching modules and hence it should satisfy some local-global compatibility with the cohomology of certain locally symmetric spaces. 
We do not give a very precise formulation of this here, but it would include the (derived) base change to a global Galois deformation ring. In the neighborhood of generic L-parameters there should be no obstruction for this base change to sit in a single cohomological degree. This motivates the following observation.
\end{rem}

\begin{cor}
Let $x=(\varphi,N)\in X_{\check G}^{\rm reg}$  such that the $\check G$-conjugacy class of $(\phi,N)$ is a generic $L$-parameter. Then $\tilde{\mathcal{M}}_G$ is locally free (as an $\Ocal_{X_{\check G}}$-module) in a neighborhood of $x$.
\end{cor}
\begin{proof}
As a maximal Cohen-Macaulay module over a regular local ring is automatically free, this follows from Remark \ref{remarkgenericLparam} and the identification of $\tilde{\mathcal{M}}_G|_{X_{\check G}^{\rm reg}}$ above.
\end{proof}

\subsection{The main conjecture in the regular case}
After restricting to the regular case we give a candidate for the functor $R_G^\psi$ in Conjecture \ref{mainconjecture}, as well as functors $R_M^{\psi_M}$ for all (standard) Levi subgroups, and prove compatibility with parabolic induction. 
As in the case of ${\rm GL}_n$ the choice of $(\mathbb{B},\psi)$ is unique up to conjugation, we will always omit the superscript $\psi$ from the notation.
By abuse of notation we will also use the symbols $\iota_P^G(-)$ and $\iota_{\overline{P}}^G(-)$ to denote the functors on Hecke modules corresponding to parabolic induction $(\ref{inductiontensor})$.

For a standard Levi subgroup $\prod_{i=1}^s \GL_{r_i}=\mathbb{M}\subset\mathbb{G}=\GL_n$ we write $\tilde{\mathcal{M}}_M$ for the tensor product of the pullbacks of the $\tilde{\Mcal}_{\GL_{r_i}(F)}$ on $X_{\GL_{r_i}}$ to $X_{\check M}=\prod_{i=1}^s X_{\GL_{r_i}}$. This is an $\check M$-equivariant sheaf of $\Ocal_{X_{\check M}}\otimes_{\mathfrak{Z}_M}\Hcal_M$-modules, and again we write $\Mcal_M$ for the sheaf on $[X_{\check M}/\check M]$ defined by $\tilde{\Mcal}_M$. We define the functor
\begin{equation}\label{EHfunctor}
\begin{aligned}
R_M:{\bf D}^+(\Hcal_M\text{-mod}) &\longrightarrow {\bf D}^+_{\rm QCoh}([X_{\check M}/\check M])\\
\pi^\bullet&\longmapsto {}^t\pi^\bullet\otimes^L_{\Hcal_M}\mathcal{M}_M.
\end{aligned}
\end{equation}
The derived tensor product in the formula can be defined for objects $\pi^\bullet\in {\bf D}^+(\Hcal_M\text{-mod})$ that are bounded above using finite projective resolutions (recall that $\Hcal_M$ has finite global dimension).
In general an object $\pi^\bullet\in {\bf D}^+(\Hcal_M\text{-mod})$ can be written as the direct limit of its truncations
\[\lim\limits_{\longrightarrow}\tau_{<m}(\pi^\bullet)\xrightarrow{\cong}\pi^\bullet,\]
and we can define 
\[{}^t\pi^\bullet \otimes^L_{\Hcal_M}\mathcal{M}_M=\lim\limits_{\longrightarrow}\big((\tau_{<m}({}^t\pi^\bullet)\otimes^L_{\Hcal_M}\mathcal{M}_M\big).\]
Note that by definition $R_M$ preserves the truncation $\tau_{<m}$.

We will write $R_M^{\rm reg}$ for the composition of $R_M$ with the restriction to the regular locus $[X_{\check M}^{\rm reg}/\check M]\subset [X_{\check M}/\check M]$.
Obviously the functors $R_M$ and $R_M^{\rm reg}$ are $\mathfrak{Z}_M$-linear. 

We restrict ourselves to the regular case. In order to have a compatible choice of the $\mathcal{M}_M$ (which are a priori only defined up to isomorphism) for various Levi subgroups of $\GL_n$, let us set
\[\mathcal{M}_M^{\rm reg}=\mathcal{M}_M|_{[X_{\check M}^{\rm reg}/\check M]}=R\beta_{B_M,\ast}\Ocal_{[X_{\check B_M}^{\rm reg}/\check B_M]},\]
where $B_M=B\cap M\subset M$ is a Borel and $\beta_{B_M}$ is the restriction of the canonical projection $[X_{\check B}/\check B]\rightarrow [X_{\check G}/\check G]$ to the regular locus. 
By abuse of notation we drop the restriction to the regular locus in the notation and just write $\mathcal{M}_M$ instead of $\mathcal{M}_M^{\rm reg}$. We now use Proposition \ref{identifyMandOB} to define the $\Hcal_M$-module structure on $\mathcal{M}_M$, i.e.~we let $\Hcal_M$ act on 
$$\Ocal_{X_{\check M}^{\rm reg}}\otimes_{\mathfrak{Z}_M}A_{\check T}\twoheadrightarrow \Ocal_{\tilde X_{\check B_M}^{\rm reg}}$$
by letting it act on $A_{\check T}=\Hcal_T\cong\Hcal_Me_{M,\rm st}$ (the fact that this $\Hcal_M$-action extends to the quotient is the content of Proposition \ref{identifyMandOB}).

Let $\mathbb{P}_1\subset\mathbb{P}_2$ be parabolic subgroups containing $\mathbb{B}$ with Levi quotients $\mathbb{M}_1$ and $\mathbb{M}_2$ and write $\mathbb{P}_{12}$ for the image of $\mathbb{P}_1$ in $\mathbb{M}_2$. 
We will define a natural $\mathfrak{Z}_{M_2}$-linear transformation 
\begin{equation}\label{defofnattrafo}
\xi_{P_{12}}^{M_2}:R_{M_2}^{\rm reg}\circ \iota_{\overline{P}_{12}}^{M_2}\longrightarrow (R\beta_{12,\ast}\circ L\alpha_{12}^\ast)\circ R_{M_1}^{\rm reg},
\end{equation}
where $\alpha_{12}$ and $\beta_{12}$ are the morphisms in the diagram
\begin{equation}\label{diagraminductionmorphism}
\begin{aligned}
\begin{xy}
\xymatrix{
[X_{\check M_2}^{\rm reg}/\check M_2] & [X_{\check P_{12}}^{\rm reg}/\check P_{12}] \ar[l]^{\beta_{12}} \ar[d]^{\alpha_{12}} &  [X_{\check B_{M_2}}^{\rm reg}/\check B_{M_2}] \ar[l]^{\beta} \ar[d]^\alpha \ar@/_ .5cm/[ll]_{\beta_{M_2}}\\
& [X_{\check M_1}^{\rm reg}/\check M_1] & \ar[l]_{\beta_{M_1}} [X_{\check B_{M_1}}^{\rm reg}/\check B_{M_1}].
}
\end{xy}\end{aligned}\end{equation}
Note that the square on the right hand side is cartesian and Tor-independent by Lemma \ref{fiberproducttorindep} and Corollary \ref{caresiansquareclassicalstacks}.

Let $\pi$ be a complex of $\Hcal_{M_1}$-module. Giving $\xi_{P_1}^{P_2}(\pi)$ is equivalent to defining its adjoint morphism
\[{}^t\xi_{P_{12}}^{M_2}(\pi): L\beta_{12}^\ast({}^t(\Hcal_{M_2}\otimes^L_{\Hcal_{M_1}}\pi)\otimes_{\Hcal_{M_2}}\mathcal{M}_{M_2})\longrightarrow L\alpha_{12}^\ast({}^t\pi\otimes^L_{\Hcal_{M_1}}\mathcal{M}_{M_1}).\]
Using Lemma \ref{lefttorightinduction} and compatibility of pullbacks with tensor products we need to define a morphism
\[{}^t\pi\otimes^L_{\Hcal_{M_1}}L\beta_{12}^\ast \mathcal{M}_{M_2}\longrightarrow {}^t\pi\otimes^L_{\Hcal_{M_1}}L\alpha_{12}^\ast\mathcal{M}_{M_1}\] that is, we need to define a morphism of $\Hcal_{M_1}\otimes_{\mathfrak{Z}_{M_2}}\Ocal_{[X_{\check P_{12}}^{\rm reg}/\check P_{12}]}$-modules
\[L\beta_{12}^\ast \mathcal{M}_{M_2}\longrightarrow L\alpha_{12}^\ast\mathcal{M}_{M_1}.\] 
Using the above identifications we can define this as the composition
\begin{equation}\label{inductionmorphism}
\begin{aligned}
L\beta_{12}^\ast \mathcal{M}_{M_2}=(L\beta_{12}^\ast\circ R\beta_{12,\ast}\circ R\beta_\ast)(\Ocal_{[X_{\check B_{M_2}}^{\rm reg}/\check B_{M_2}]})& 
\\\longrightarrow R\beta_\ast L\alpha^\ast(\Ocal_{[X_{\check B_{M_1}}^{\rm reg}/\check B_{M_1}]})\xrightarrow{\cong} L\alpha_{12}^\ast R\beta_{M_1,\ast}(\Ocal_{[X_{\check B_{M_1}}^{\rm reg}/\check B_{M_1}]})=  L\alpha_{12}^\ast\mathcal{M}_{M_1}&,
\end{aligned}
\end{equation}
where the first morphism is given by adjunction and the second morphism is given by the base change morphism in the cartesian square in $(\ref{diagraminductionmorphism})$. A priori this is only a morphism of $\Ocal_{[X^{\rm reg}_{\check P_{12}}/\check P_{12}]}$-modules.
\begin{lem}
The morphism $(\ref{inductionmorphism})$ is a morphism of $\Hcal_{M_1}$-modules.
\end{lem}
\begin{proof}
We prove the claim after pulling back to $\tilde X_{\check P_{12}}^{\rm reg}$ in $(\ref{diagraminductionmorphism})$. 

We write $\tilde\alpha$, $\tilde\beta$ etc.~for the corresponding morphisms of schemes. 
As all the maps $\tilde\beta$ (with various subscripts) are affine, all but the first object in $(\ref{inductionmorphism})$ are concentrated in degree 0. 
Moreover, all schemes are reduced, and hence it is enough to prove the claim after restricting to the dense open subscheme where $\phi$ is regular semi-simple. 
We denote these open subschemes by $\tilde X_{\check B_{M_2}}^{\rm reg\text{-}ss}$ etc..
Consider the diagram 
\[\begin{xy}
\xymatrix{
\tilde X_{\check B_{M_2}}^{\rm reg\text{-}ss} \ar[r]\ar[d]^{\tilde\alpha}& X_{\check M_2}^{\rm reg\text{-}ss}\times_{\check T/W_{M_2}}\check T \ar[r] & \check M_2^{\rm reg\text{-}ss}\times_{\check T/W_{M_2}}\check T \ar[r] \ar[d]^{(\ast)}& \check T\ar@{=}[d]\\
\tilde X_{\check B_{M_1}}^{\rm reg\text{-}ss} \ar[r] & X_{\check M_1}^{\rm reg\text{-}ss}\times_{\check T/W_{M_1}}\check T \ar[r] & \check M_1^{\rm reg\text{-}ss}\times_{\check T/W_{M_1}}\check T \ar[r] & \check T.
}
\end{xy}\]
Here, the vertical arrow $(\ast)$ on the right hand side is induced by the identification
 $$\check M_i^{\rm reg\text{-}ss}\times_{\check T/W_{M_i}}\check T\cong \big\{(\phi,g\check B_{M_i})\in \check M_i^{\rm reg\text{-}ss}\times \check M_i/\check B_{M_i}\mid \phi\in g^{-1}\check B_{M_i}g\big\}.$$ 
By definition the $\Hcal_{M_1}$-module structures on source and target of 
\[\tilde\beta_\ast\tilde\alpha^\ast\Ocal_{\tilde X^{\rm reg\text{-}ss}_{\check B_{M_1}}}\cong \tilde\alpha_{12}^\ast\tilde\beta_{M_1,\ast}\Ocal_{\tilde X_{\check B_{M_1}}^{\rm reg\text{-}ss}}\]
are induced by two (a priori maybe different) $\Hcal_{M_1}$-module structures of the structure sheaves $$\Ocal_{\tilde X^{\rm reg\text{-}ss}_{\check B_{M_2}}}\twoheadleftarrow \Ocal_{X_{\check {M_2}}^{\rm reg\text{-}ss}\times_{\check T/{W_{M_2}}}\check T}$$ which in turn are given by the pullback of an $\Hcal_{M_1}$-action on $A_{\check T}$. These $\Hcal_{M_1}$-actions are given by
\begin{itemize}
\item[-] the $\Hcal_{M_1}$ action on $A_{\check T}$ given by $A_{\check T}\cong\Hcal_Te_{M_1,\rm st}$,
\item[-] the restriction of the $\Hcal_{M_2}$ action on $A_{\check T}$ given by $A_{\check T}\cong\Hcal_Te_{M_2,\rm st}$.
\end{itemize}
By Lemma \ref{HGestfree} (ii) these actions coincide.
\end{proof}
We obtain the following first step towards Conjecture \ref{mainconjecture}.
\begin{theo}\label{theocompwparabinductionregular}
For each parabolic $\mathbb{B}\subset\mathbb{P}\subset\mathbb{G}$ with Levi $\mathbb{M}$ the restriction of $(\ref{EHfunctor})$ to the regular locus is a $\mathfrak{Z}_M$-linear functor
\[R_M^{\rm reg}:{\bf D}^+(\Hcal_M\text{-}{\rm mod}) \longrightarrow {\bf D}^+_{\rm QCoh}([X^{\rm reg}_{\check M}/\check M]).
\]
Moreover, for two parabolic subgroups $\mathbb{B}\subset \mathbb{P}_1\subset \mathbb{P}_2$ the natural transformation $\xi_{P_{12}}^{M_2}$ defined in $(\ref{defofnattrafo})$ is a $\mathfrak{Z}_{M_2}$-linear isomorphism. \\
For parabolic subgroups $\mathbb{P}_1\subset\mathbb{P}_2\subset\mathbb{P}_3$ let $\mathbb{M}_3$ denote the Levi quotient of $\mathbb{P}_3$ and $\mathbb{P}_{13}\subset\mathbb{P}_{23}$ denote the images of $\mathbb{P}_1\subset\mathbb{P}_2$ in $\mathbb{M}_3$. Then the diagram in Conjecture \ref{mainconjecture} (b), applied to $\mathbb{P}_{13}\subset\mathbb{P}_{23}\subset\mathbb{M}_3$, commutes. 
\end{theo}
\begin{proof}

We are left to prove that $\xi_{P_{12}}^{M_2}$ is an isomorphism and that the diagram in Conjecture \ref{mainconjecture} (b) commutes. 
Using truncations and resolutions by free modules it is enough to prove that
\begin{align*}\xi_{P_{12}}^{M_2}(\Hcal_{M_1}):&\ \mathcal{M}_{M_2}={}^t(\Hcal_{M_2}\otimes_{\Hcal_{M_1}}\Hcal_{M_1})\otimes_{\Hcal_{M_2}}\mathcal{M}_{M_2} \\ &\longrightarrow R\beta_{12,\ast}(L\alpha_{12}^\ast \mathcal{M}_{M_1})\cong R\beta_{12,\ast}R\beta_\ast \Ocal_{[X_{\check B_{M_2}}^{\rm reg}/\check B_{M_2}]}=\mathcal{M}_{M_2}
\end{align*}
is an isomorphism. However, this is a direct consequence of the construction of $\xi_{P_{12}}^{M_2}$ in $(\ref{inductionmorphism})$ using the base change isomorphism in the cartesian square of $(\ref{diagraminductionmorphism})$.

As $\xi_{P_{12}}^{M_2}$ is the composition of an adjunction morphism and a base change map, the commutativity of (b) in the conjecture is a consequence of standard compatibilities of base change morphisms and adjunctions.
\end{proof}

\begin{rem}\label{Remarkextendtononregular}
We point out that the arguments above directly extend from the regular locus to all of $X_{\check G}$ once Conjecture \ref{GLnsurjection} is known. 
\end{rem}

\subsection{Compactly induced representations}\label{sectioncind}
We describe the image of the functor $R_G$ defined in $(\ref{EHfunctor})$ on (the $I$-invariants in) the compactly induced representations $\cind_K^G\sigma_{\mathcal{P}}$. The result parallels, and is motivated by, results of Pyvovarov in \cite{Pyvovarov3}.

Recall from Proposition \ref{equidimreductive} (ii) that the irreducible components of $X_{\check G}$ are in bijection with the possible Jordan canonical forms of the nilpotent endomorphism $N$. 
For a partition $\mathcal{P}$ let $Z_{\check G,\mathcal{P}}$ denote the irreducible component of $X_{\check G,\mathcal{P}}$ such that the Jordan canonical form of $N$ at the generic point of $Z_{\check G,\mathcal{P}}$ is given by the partition $\mathcal{P}$. Then we set $$X_{\check G,\mathcal{P}}=\bigcup_{\mathcal{P}\preceq \mathcal{P}'}Z_{\check G,\mathcal{P}'}.$$
In particular we have $X_{\check G,\mathcal{P}_{\rm min}}=X_{\check G}$, and $X_{\check G,\mathcal{P}_{\rm max}}=Z_{\check G,\mathcal{P}_{\rm max}}$ is irreducible. We will sometimes write $X_{\check G,0}$ for this irreducible component, as it is the irreducible component defined by $N=0$. We write $\eta_\mathcal{P}$ for the generic point of the irreducible component $Z_{\check G,\mathcal{P}}$.

\begin{prop}\label{imagesofprojectives}
Let $\mathcal{P}$ be a partition. Then $R_G(\cind_K^G\sigma_\mathcal{P})$ is concentrated in degree $0$ and, viewed as a $\check G$-equivariant coherent sheaf on $X_{\check G}$, has support $X_{\check G,\mathcal{P}}$. Moreover,
\begin{align*}
R_G((\cind_K^G\mathbf{1}_K)^I)&=\Ocal_{X_{\check G,0}},\\
R_G((\cind_K^G{\rm st}_G)^I) &=\Ocal_{X_{\check G}},
\end{align*}
equipped with their canonical $\check G$-equivariant structures. In particular 
\[R_G((\cind_N^G\psi)_{[T,1]}^I)=\Ocal_{[X_{\check G}/\check G]}.\]
\end{prop}
\begin{proof}
We will rather calculate the images of $$\Hcal_G e_\mathcal{P}=(\cind_K^G\Sigma_P)^I\cong(\cind_K^G\sigma_\mathcal{P}^{\oplus m_\mathcal{P}})^I.$$ Recall that $m_{\mathcal{P}_{\rm min}}=m_{\mathcal{P}_{\rm max}}=1$.
Using Lemma \ref{dualofidempotents} we see that the $\check G$-equivariant coherent sheaf on $X_{\check G}$ defined by $R_G((\cind_K^G\Sigma_P)^I)$ is $e_{\mathcal{P}}\Hcal_{G}\otimes_{\Hcal_G}\tilde{\mathcal{M}}_G$.

Recall that by definition the sheaf $\tilde{\mathcal{M}}_G$ is the sheaf attached to the image of
\[A_{\check G}\otimes_\mathfrak{Z}\Hcal_Ge_{\rm st}\longrightarrow \prod_{\mathcal{P}'} {\rm LL}(\phi_{\eta_{\mathcal{P}'}},N_{\eta_{\mathcal{P}'}})^I\]
induced by the surjections $\Hcal_Ge_{\rm st}\otimes_\mathfrak{Z}k(\eta_{\mathcal{P}'})\rightarrow  {\rm LL}(\phi_{\eta_{\mathcal{P}'}},N_{\eta_{\mathcal{P}'}})^I$ of Lemma \ref{surjecttoLLmod}.
Consequently $e_{\mathcal{P}}\Hcal_{G}\otimes_{\Hcal_G}\tilde{\mathcal{M}}_G$ is the sheaf defined by the image of the morphism
\[A_{\check G}\otimes_{\mathfrak{Z}}e_\mathcal{P}\Hcal_Ge_{\rm st}\longrightarrow \prod_{\mathcal{P}'}e_{\mathcal{P}}{\rm LL^{mod}}(\phi_{\eta_{\mathcal{P}'}},N_{\eta_{\mathcal{P}'}})^I.\]

Note that $A_{\check G}\otimes_{\mathfrak{Z}}e_\mathcal{P}\Hcal_Ge_{\rm st}$ is a free $A_{\check G}$-module of rank $m_\mathcal{P}^2$, by Lemma \ref{technicalidentities} (iv). 
To show that the sheaf $R_G(\Hcal_Ge_\mathcal{P})$ has support $X_{\check G,\mathcal{P}}$ it remains to show that 
\[e_{\mathcal{P}}{\rm LL}(\phi_{\eta_{\mathcal{P}'}},N_{\eta_{\mathcal{P}'}})^I\neq 0 \Longleftrightarrow \mathcal{P}\preceq\mathcal{P}'.\]
The left hand side can be identified with 
\begin{align*}
{\rm Hom}_{\Hcal_G}(\Hcal_Ge_\mathcal{P},{\rm LL}(\phi_{\eta_{\mathcal{P}'}},N_{\eta_{\mathcal{P}'}})^I)&={\rm Hom}_G(\cind_K^G\Sigma_\mathcal{P},{\rm LL}(\phi_{\eta_{\mathcal{P}'}},N_{\eta_{\mathcal{P}'}}))\\ 
&={\rm Hom}_K(\sigma_{\mathcal{P}},{\rm LL}(\phi_{\eta_{\mathcal{P}'}},N_{\eta_{\mathcal{P}'}}))^{m_\mathcal{P}}.
\end{align*}
As ${\rm LL}(\phi_{\eta_{\mathcal{P}'}},N_{\eta_{\mathcal{P}'}})$ is absolutely irreducible and generic \cite[Theorem 3.7]{Shotton} implies the claim.

If $\mathcal{P}\in\{\mathcal{P}_{\rm min},\mathcal{P}_{\rm max}\}$, then $\Sigma_\mathcal{P}=\sigma_\mathcal{P}$ and $A_{\check G}\otimes_{\mathfrak{Z}}e_\mathcal{P}\Hcal_Ge_{\rm st}\cong A_{\check G}$.
In this case the above discussion shows that $R_G((\cind_K^G\sigma_\mathcal{P})^I)$ is the structure sheaf of the union of those irreducible components $Z_{\check G,\mathcal{P}'}$ such that 
${\rm Hom}_K(\sigma_\mathcal{P},{\rm LL}(\phi_{\eta_{\mathcal{P}'}},N_{\eta_{\mathcal{P}'}}))\neq 0$. 
If $\mathcal{P}=\mathcal{P}_{\rm max}$ this implies $\mathcal{P}'=\mathcal{P}$ as above. On the other hand, if $\mathcal{P}=\mathcal{P}_{\rm min}$, then ${\rm Hom}_K(\sigma_\mathcal{P},{\rm LL}(\phi_{\eta_{\mathcal{P}'}},N_{\eta_{\mathcal{P}'}}))\neq 0$ for all $\mathcal{P}'$ by \cite[Theorem 1.3]{Pyvovarov2}.
\end{proof}
\begin{rem}\label{remRGcind}
A closer analysis of the proof shows that $R_G(\cind_K^G\sigma_\mathcal{P})$ can never be (locally) free over its support $X_{\check G,\mathcal{P}}$ unless $m_\mathcal{P}=1$. 
Indeed, generically on $Z_{\check G,\mathcal{P}}$ the sheaf $R_G(\cind_K^G\sigma_\mathcal{P})$ is free of rank $1$, using \cite[Theorem 3.7 (ii)]{Shotton}.
On the other hand, generically on $X_{\check G,0}$ this sheaf is free of rank $m_{\mathcal{P}}$.
Indeed, let $L$ be the algebraic closure of $k(\eta_{\mathcal{P}_{\rm max}})$. Then 
${\rm LL}(\phi_{\mathcal{P}_{\rm max}},N_{\mathcal{P}_{\rm max}})\otimes_{k(\eta_{\mathcal{P}_{\rm max}})}L$ is an irreducible representation induced from the upper triangular Borel. On the other hand $\cind_K^G\sigma_\mathcal{P}\otimes_{\mathfrak{Z}}K$ is a direct sum of $m_{\mathcal{P}}$ copies of the same irreducible principal series representation by Corollary 6.1 and Lemma 6.4 of \cite{Pyvovarov1}.
\end{rem}

\subsection{Remarks about the relation of the various conjectures}\label{remarksaboutthevariousconjectures}
In this section we add a few remarks about the relation of Conjecture \ref{GLnsurjection} with Conjectures \ref{mainconjecture} and \ref{conjdegzero}.  This should give some evidence for Conjecture \ref{GLnsurjection} which in turn would imply that $R\tilde\beta_{B,\ast}\mathcal{O}_{\tilde{\mathbf{X}}_{\check B}}$ agrees with the $\mathcal{O}_{X_{\check G}}$-module underlying the Hecke module $\tilde{\mathcal{M}}_G$ defined by the Emerton-Helm family. 

\noindent (a) Before we come to this point, we mention that our expectations about $R\beta_{B,\ast}\mathcal{O}_{\tilde{\mathbf{X}}_{\check B}}$ imply that, for $G={\rm GL}_n$, the functor $R_G$ is uniquely determined by the requirements in Conjecture \ref{mainconjecture}. 
If $R_G$ is any functor satisfying the main conjecture, then 
\begin{align*}
R_G(\cind_I^G\mathbf{1}_I)&=R_G(\iota_{\overline B}^G (\cind_{T^\circ}\mathbf{1}_{T^\circ}))=R\beta_{B,\ast}(L\alpha_{B}^\ast R_T(\cind_{T^\circ}^T \mathbf{1}_{T^\circ}))\\
&=R\beta_{B,\ast}(L\alpha^\ast_B(\mathcal{O}_{[X_{\check T}/\check T]}))=R\beta_{B,\ast}\mathcal{O}_{\tilde{\mathbf{X}}_{\check B}}
\end{align*}
which is equipped with a faithful action of $\Hcal_G={\rm End}_{G}(\cind_I^G\mathbf{1}_I)$. Using the identification ${\rm Rep}_{[T,1]}G\cong \mathcal{H}_G\text{-}{\rm mod}$ and resolutions by free $\mathcal{H}_G$-modules, we deduce that the functor $R_G$ on the category $\mathcal{H}_G\text{-}{\rm mod}$ necessarily has to be of the form 
\[\pi^\bullet\longmapsto {}^{t}\pi^\bullet\otimes^L_{\mathcal{H}_G} R\beta_{B,\ast}\mathcal{O}_{\tilde{\mathbf{X}}_{\check B}}.\]
As we expect that $R\beta_{B,\ast}\mathcal{O}_{\tilde{\mathbf{X}}_{\check B}}$ is concentrated in one degree and is a maximal Cohen-Macaulay-module the $\mathcal{H}_G$-action (if it exists) is uniquely determined by its specialization at the generic points of $X_{\check G}$. At these points the Frobenius $\phi$ is regular semi-simple and we will see in Corollary \ref{Corollaryuniquelydetermined} below that in this case the completion of the functor (with respect to the corresponding character of the center $\mathfrak{Z}$) is uniquely determined by the conditions in Conjecture \ref{mainconjecture}. 

\noindent (b) One could hope that (if Conjecture \ref{conjdegzero} holds true) the Cohen-Macaulay property of $R\tilde \beta_{B,\ast}\mathcal{O}_{\tilde{\mathbf{X}}_{\check B}}$ implies that the $\mathcal{H}_G$-action extends from its restriction to the regular locus to all of $X_{\check G}$. Indeed this would be automatic if the complement of the open dense subset $X_{\check G}^{\rm reg}$ had codimension larger or equal to $2$ in $X_{\check G}$. Unfortunately this is not true in general, not even in the case $G={\rm GL}_n$. However, the generic points of $X_{\check G}\backslash X_{\check G}^{\rm reg}$ that are of codimension $1$ in $X_{\check G}$ can be described rather explicitly. In fact it turns out that they are generic L-parameters and hence $X_{\check G}$ is smooth at these points. It seems likely that Conjecture \ref{GLnsurjection} can be checked explicitly ("by hand") at these points. This would imply that the $\Hcal_G$-action on $R\tilde \beta_{B,\ast}\mathcal{O}_{\tilde{\mathbf{X}}_{\check B}}$  is well defined on an open subset $U$ whose complement has codimension $\geq 2$ and hence (using Conjecture \ref{conjdegzero}) the Hecke action extends to all of $R\tilde \beta_{B,\ast}\mathcal{O}_{\tilde{\mathbf{X}}_{\check B}}$.

\noindent (c) We assume Conjecture \ref{conjdegzero} and Conjecture \ref{mainconjecture} and explain that in this case Conjecture \ref{GLnsurjection} can be attacked using the functor $R_G$. The sheaf 
$$\tilde{\mathcal{M}}'_G=R\tilde{\beta}_{B,\ast}\mathcal{O}_{\tilde{\mathbf{X}}_{\check B}}$$ is an $\Ocal_{X_{\check G}}\otimes_{\mathfrak{Z}}\Hcal_G$-module that is Cohen-Macaulay as an $\Ocal_{X_{\check G}}$-module. 
In particular its direct summand $e_{G,\rm st}\tilde{\mathcal{M}}'_G$ is a Cohen-Macaulay module as well. If Conjecture \ref{GLnsurjection} can be checked on an open subset $U$ as in (b), then $e_{G,\rm st}\tilde{\mathcal{M}}'_G|_U\cong \mathcal{O}_U$ as in the proof of Proposition \ref{imagesofprojectives}. The Cohen-Macaulay property hence implies that $e_{G,\rm st}\tilde{\mathcal{M}}'_G\cong \mathcal{O}_{X_{\check G}}$ and hence by $(\ref{(n)adjointness})$ there is a canonical morphism
\[\mathcal{H}_Ge_{G,{\rm st}}\otimes_{\mathfrak{Z}}\mathcal{O}_{X_{\check G}}\longrightarrow \tilde{\mathcal{M}}'_G.\]
We expect that this morphism is a surjection, and that this surjection agrees with the surjection in Conjecture \ref{GLnsurjection}. In fact it seems that this can be checked using properties of the conjectured functor $R_G(-)={}^t(-)\otimes_{\mathcal{H}_G}^L\tilde{\mathcal{M}}'_G$ (compare (a) for the fact that $R_G$ necessarily is of this form):
We only need to check that for each point $x\in X_{\check G}$ the canonical map \begin{equation}\label{eqnsurjectiveatfibers}\mathcal{H}_Ge_{G,{\rm st}}\otimes_{\mathfrak{Z}}k(x)\longrightarrow \tilde{\mathcal{M}}'_G\otimes k(x)\end{equation}
is surjective. As $e_{G,\rm st}\tilde{\mathcal{M}}'_G\otimes k(x)=k(x)$ is one dimensional there is a unique generic Jordan-H\"older factor $\pi_x$ in  $\tilde{\mathcal{M}}'_G\otimes k(x)$ and $(\ref{eqnsurjectiveatfibers})$ is surjective if any only if $\pi_x$ is the cosocle of $\tilde{\mathcal{M}}'_G\otimes k(x)$. The center $\mathfrak{Z}$ acts on $\pi_x$ via a character $\chi_x:\mathfrak{Z}\rightarrow k(x)$ and there are only finitely many irreducible representations $\pi$ on which  $\mathfrak{Z}$ acts via $\chi_x$. Among those, $\pi_x$ is the unique generic one. 
We expect that for irreducible $\pi$ the complex of coherent sheaves $R_G(\pi)$ (which is supported in non-positive degrees) has cohomology in degree zero if and only if $\pi$ is generic. This implies in particular that 
${}^t\pi\otimes_{\Hcal_G}\tilde{\mathcal{M}}'_G\otimes k(x)=0$ unless $\pi$ is generic. Using the Tensor-Hom adjunction
\[{\rm Hom}_{k(x)}({}^t\pi\otimes_{\Hcal_G}\tilde{\mathcal{M}}'_G\otimes k(x),k(x))={\rm Hom}_{\mathcal{H}_G}({}^t\pi,{\rm Hom}_{k(x)}(\tilde{\mathcal{M}}'_G\otimes k(x),k(x)))\]
this expectation would imply that the cosocle of $\tilde{\mathcal{M}}'_G\otimes k(x)$ has to be generic and hence $(\ref{eqnsurjectiveatfibers})$ must be a surjection.

\subsection{Proof of the conjecture for ${\rm GL}_2$}
We prove Conjecture \ref{mainconjecture} in the two dimensional case.
In this subsection we use the notation $G=\GL_2(F)$ and $\check G$ is the algebraic group $\GL_2$ over $C$. 
In this case the $\check B$-action on $\Lie \check B\cap \mathcal{N}_{{\rm GL}_2}$ has two orbits and hence $X_{\check B}$ is a complete intersection and $\mathbf{X}_{\check B}=X_{\check B}$, see Remark \ref{XBfor GLsmalln}. 
Moreover, this remark implies that $X_{\check B}$ is reduced, and both of its two irreducible components are the closure of an irreducible component of $X_{\check B}^{\rm reg}$.
To simplify notations, we will write $X_0=X_{\check G,0}=Z_{\check G,\mathcal{P}_{\rm max}}\subset X_{\check G}$ for the component given by $N=0$ and $X=X_{\check G}$. Moreover, we sometimes write $X_1=Z_{\check G,\mathcal{P}_{\rm min}}\subset X$ for the component on which $N$ is generically non-trivial.

\begin{prop}
Conjecture \ref{GLnsurjection} is true for ${\rm GL}_2$. In particular we obtain an identification
\[\mathcal{M}_G=R\tilde\beta_\ast\Ocal_{\tilde X_{\check B}}.\]
\end{prop}
\begin{proof}
As already discussed above the claim holds over the open subset $X_{\check G}^{\rm reg}\subset X=X_{\check G}$.
On the other hand the closed complement of $X_{\check G}^{\rm reg}$ has the open neighborhood $X\backslash X_1$ which is an open subset of $\check G$. The claim now follows from the well known fact that 
\[h:\widetilde{\GL_2}=\{(\phi,gB)\in \check G\times \check G/\check B\mid \phi\in gBg^{-1}\}\longrightarrow \check G=\GL_2\]
has vanishing higher direct images, and its global sections are given by
\[\Gamma(\widetilde{\GL_2},\Ocal_{\widetilde{\GL_2}})=\Gamma(\GL_2\times_{\check T/W}\check T,\Ocal_{\GL_2\times_{\check T/W}\check T}).\]
\end{proof}


As a consequence we still can use $(\ref{inductionmorphism})$ to define a natural transformation 
\begin{equation}\label{nattrafoGL2}
\xi_B^G:R_G\circ\iota_{\overline{B}}^G\longrightarrow (R\beta_\ast L\alpha^\ast)\circ R_T,
\end{equation}
where $\alpha:[X_{\check B}/\check B]\rightarrow [X_{\check T}/\check T]$ and $\beta:[X_{\check B}/\check B]\rightarrow [X_{\check G}/\check G]$ are the canonical morphisms. The same computation as in the proof of Theorem \ref{theocompwparabinductionregular} again shows that this natural transformation is a $\mathfrak{Z}$-linear isomorphism, compare Remark \ref{Remarkextendtononregular}.

\begin{theo}\label{theoGL2}
Let $G={\rm GL}_2(F)$ and $T\subset B\subset G$ denote the standard maximal torus respectively the standard Borel. The functors $R_G$ and $R_T$ defined by $(\ref{EHfunctor})$ are fully faithful and the natural transformation $\xi_B^G$ defined by $(\ref{nattrafoGL2})$ is a $\mathfrak{Z}$-linear isomorphism. Moreover,
\[R_G((\cind_N^G\psi)^I_{[T,1]})\cong \Ocal_{[X_{\check G}/\check G]}\]
for a choice of a generic character $\psi:N\rightarrow C^\times$ of the unipotent radical $N$ of $B$.
\end{theo} 

By the above discussion, it remains to show that $R_G$ is fully faithful.
Let us write $f\in\mathfrak{Z}$ for the element corresponding to the characteristic polynomials of the form $(T-\lambda)(T-q\lambda)$ for some indeterminate $\lambda$.
Then the morphism
\[\begin{xy}\xymatrix{
\Ocal_{[X/\check G]}\ar[r]^{\cdot f} & \Ocal_{[X/\check G]}
}\end{xy}\]
factors through $\Ocal_{[X/\check G]}\twoheadrightarrow \Ocal_{[X_0/\check G]}$ and yields a morphism
\begin{equation}\label{basemorphismGL2}
\xymatrix{
\Ocal_{[X_0/\check G]}\ar[r]& \Ocal_{[X/\check G]}
}
\end{equation}
with image $f\Ocal_{[X/\check G]}$ and cokernel $\Ocal_{[X_1/\check G]}$.
\begin{prop}\label{extgroupcomputation1}
Let $\mathcal{F},\mathcal{G}\in\{\Ocal_{[X/\check G]},\Ocal_{[X_0/\check G]}\}$. Then
\[
{\rm Ext}^i_{[X/\check G]}(\mathcal{F},\mathcal{G})=\begin{cases}\mathfrak{Z},& i=0\\ 0,& i\neq 0.\end{cases}
\]
More precisely, a $\mathfrak{Z}$-basis of ${\rm Hom}={\rm Ext}^0$ is given by the identity if $\mathcal{F}=\mathcal{G}$. \\
If $\mathcal{F}=\Ocal_{[X/\check G]}$ and $\mathcal{G}=\Ocal_{[X_0/\check G]}$, then a $\mathfrak{Z}$-basis is given by the canonical projection,\\
and if $\mathcal{F}=\Ocal_{[X_0/\check G]}$ and $\mathcal{G}=\Ocal_{[X/\check G]}$, a $\mathfrak{Z}$-basis is given by the morphism $(\ref{basemorphismGL2})$.
\end{prop}
\begin{proof}
We can easily reduce to the case $C$ algebraically closed. 

Consider the canonical projection $$f:[X/\check G]\longrightarrow[\ast/\check G]={\rm B}\check G.$$
We need to compute $H^i({\rm B}\check G,Rf_\ast R\mathcal{H}om(\mathcal{F},\mathcal{G}))$. As $\check G$ is reductive and the base field $C$ has characteristic $0$ the category of $\check G$-representations is semi-simple and hence this vector space is given by
$H^i(Rf_\ast R\mathcal{H}om(\mathcal{F},\mathcal{G}))^{\check G}$, compare also \cite[Lemma 2.4.1]{DG}. Here we write $H^i(Rf_\ast R\mathcal{H}om(\mathcal{F},\mathcal{G}))$ for the $i$-th cohomology sheaf of the complex $Rf_\ast R\mathcal{H}om(\mathcal{F},\mathcal{G})$, which is a sheaf on ${\rm B}\check G$, and hence a $\check G$-representation.

Let us write $\tilde{\mathcal{F}}$ and $\tilde{\mathcal{G}}$ for the pullbacks of $\mathcal{F}$ and $\mathcal{G}$ to $X$. 
Then, by definition, giving the quasi-coherent sheaf $H^i(Rf_\ast R\mathcal{H}om(\mathcal{F},\mathcal{G}))$ on ${\rm B}\check G$ is the same as giving a $\check G$-equivariant structure on $H^i(R\mathcal{H}om(\tilde{\mathcal{F}},\tilde{\mathcal{G}}))$.
We conclude that 
\[{\rm Ext}^i_{[X/\check G]}(\mathcal{F},\mathcal{G})=\big({\rm Ext}^i_X(\tilde{\mathcal{F}},\tilde{\mathcal{G}})\big)^{\check G},\]
for the canonical $\check G$-representation on ${\rm Ext}^i_X(\tilde{\mathcal{F}},\tilde{\mathcal{G}})$ induced by the $\check G$-equivariant structures on $\tilde{\mathcal{F}}$ and $\tilde{\mathcal{G}}$. 

If $\mathcal{F}=\Ocal_{[X/\check G]}$, then 
\[{\rm Ext}^i_X(\tilde{\mathcal{F}},\tilde{\mathcal{G}})=\begin{cases}\Gamma(X,\tilde{\mathcal G}), & i=0\\ 0,&i\neq 0\end{cases}\]
and one easily computes $\Gamma(X,\tilde{\mathcal G})^{\check G}\cong\mathfrak{Z}$ in both cases. Moreover, a $\mathfrak{Z}$-basis is easily identified with the identity, respectively the canonical projection, as claimed.

Now assume that $\mathcal{F}=\Ocal_{[X_0/\check G]}$. We compute $${\rm Ext}^i_X(\tilde{\mathcal{F}},\tilde{\mathcal{G}})=\Gamma(X,\mathcal{E}xt^i_{\Ocal_X}(\tilde{\mathcal{F}},\tilde{\mathcal{G}})).$$ 
If $i\neq 0$, the sheaf $\mathcal{E}xt^i_{\Ocal_X}(\tilde{\mathcal{F}},\tilde{\mathcal{G}})$ clearly is supported on the intersection $X_0\cap X_1$. We first show that it is a locally free sheaf on $X_0\cap X_1$ (equipped with the reduced scheme structure). 
Let $X^{\text{reg-ss}}\subset X$ denote the Zariski open subset of $(\phi,N)$ with $\phi$ regular semi-simple. Then $X_0\cap X_1\subset X^{\text{reg-ss}}$. 

Moreover, let $X'\rightarrow X^{\text{reg-ss}}$ denote the scheme parametrizing a $\phi$-stable subspace. This is an \'etale Galois cover of degree two and the filtration by the universal $\phi$-stable subspace has a canonical $\phi$-stable splitting. Let us write $V_1$ and $V_2$ for these eigenspaces and let $Y\rightarrow \tilde X'$ denote the $\check T$-torsor trivializing $V_1$ and $V_2$.  Moreover, let 
 $$Z=\{(\lambda_1,\lambda_2,a,b)\in\check T\times \mathbb{A}^2\mid a(\lambda_2-q\lambda_1)=0=b(\lambda_1-q\lambda_2),\ \lambda_1\neq \lambda_2\}$$ equipped with the $\check T=\Spec C[s_1^{\pm1},s_2^{\pm1}]$-action that is trivial on $\check T$ and via multiplication with the character $\alpha:(s_1,s_2)\mapsto s_1s_2^{-1}$ on $a$, and via $\alpha^{-1}$ on $b$.
We consider the diagram
\[\begin{xy}
\xymatrix{
& Y\ar[dl]_{\beta}\ar[dr]^\gamma &\\
X^{\text{reg-ss}} && Z,
}
\end{xy}\]
where $\beta$ is the canonical projection, which is $\check G$-equivariant, and $\gamma$ is the $\check T$-equivariant $\check G$-torsor that is given by writing the matrices of $\phi$ and $N$ over $Y$ as
\[{\rm Mat}(\phi)=\begin{pmatrix}\lambda_1 & 0 \\ 0 & \lambda_2\end{pmatrix}\ \ \text{and} \ \ {\rm Mat}(N)=\begin{pmatrix} 0 & a \\ b & 0\end{pmatrix} \]
in the chosen basis of $V_1$ and $V_2$.

Let $x=(\phi_0,0)\in X_0\cap X_1$ be a $C$-valued point. Without loss of generality we may assume $\phi_0={\rm diag}(\lambda_0,q\lambda_0)$.
Let $y\in Y$ be a pre-image of $x$ and let $z$ denote its image in $Z$, such that $z=(\lambda_0,q\lambda_0,0,0)$.
Consider the closed subscheme $Z_0=V(a,b)\subset Z$ and write $\mathcal{F}_Z=\Ocal_{Z_0}$ and 
\[\mathcal{G}_Z=\begin{cases}\mathcal{O}_Z,& \text{if}\ \tilde{\mathcal{G}}=\Ocal_X\\ \mathcal{O}_{Z_0},& \text{if}\ \tilde{\mathcal{G}}=\Ocal_{X_0}.\end{cases}\]
Then $\mathcal{E}xt^i_{\Ocal_X}(\tilde{\mathcal{F}},\tilde{\mathcal{G}})$ is locally free on $X_0\cap X_1$ if and only if $\mathcal{E}xt^i_{\Ocal_Z}(\mathcal{F}_Z,\mathcal{G}_Z)$ is locally free on $Z_0$. 
Let $S=\hat{\Ocal}_{Z,z}\cong C[\![t_1,t_2,a]\!]/((t_1-t_2)a)$ be the complete local ring at $z$ with $\lambda_1=\lambda_0+t_1$ and $\lambda_2=q(\lambda_0+t_2)$, and consider the $\check T$-equivariant resolution of $\hat{\mathcal{F}}_{Z,z}=S/(a)$ given by
\[\cdots \rightarrow S(2)\xrightarrow{\cdot a}S(1)\xrightarrow{\cdot(t_1-t_2)}S(1)\xrightarrow{\cdot a}S.\]
Here $S(m)$ is the free $S$-module of rank $1$ with the $\check T$-action twisted by the multiplication with $\alpha^m$.
It follows that 
\[{\rm Ext}^i_S(S/(a),S)=\begin{cases}S/(a), & i=0\\ 0, & i\geq 1,\end{cases}\]
and 
\[{\rm Ext}^i_A(S/(a),S/(a))=\begin{cases}S/(a), & i=0\\ 0, & i\ \text{odd} \\ (S/(t_1-t_2,a))(-i/2), & i\geq 2\ \text{even}.\end{cases}\]
In particular $\mathcal{E}xt^i_{\Ocal_Z}(\mathcal{F}_Z,\Ocal_Z)$ vanishes for $i\neq 0$, and $\mathcal{E}xt^i_{\Ocal_Z}(\mathcal{F}_Z,\Ocal_{Z_0})$ vanishes for odd $i$ and is locally free of rank 1 over $Z_0$ for non-zero even $i$.
We deduce $$\mathcal{E}xt^i_{\Ocal_X}(\tilde{\mathcal{F}},\Ocal_X)=0\ \ \text{for}\ \ i\neq 0.$$
Moreover, it follows that $\mathcal{E}xt^i_{\Ocal_X}(\tilde{\mathcal{F}},\Ocal_{X_0})$ vanishes for odd $i$ and is locally free of rank $1$ on $X_0\cap X_1$ for non-zero even $i$.
In particular a $\check G$-invariant global section  $$h\in {\rm Ext}^i_X(\tilde{\mathcal{F}},\Ocal_{X_0})=\Gamma(X,\mathcal{E}xt^i_{\Ocal_X}(\tilde{\mathcal{F}},\Ocal_{X_0}))$$
vanishes if $$0=h(x')\in \mathcal{E}xt^i_{\Ocal_X}(\tilde{\mathcal{F}},\Ocal_{X_0})\otimes k(x')$$ for all $x'\in X_0\cap X_1$. Hence we have to show $h(x')=0$ for all $x'\in X_0\cap X_1$ for even $i\neq 0$.
Again it is enough to check this for our choice $x=(\phi_0,0)$. Then $\check T={\rm Stab}_{\GL_2}(\phi_0)$ acts on the fiber $\mathcal{E}xt^i_{\Ocal_X}(\tilde{\mathcal{F}},\Ocal_{X_0})\otimes k(x)$,
and $h(x)$ is $\check T$-invariant. By the above diagram the $\check T$-action on this fiber is the same as the $\check T$-action on 
\[\mathcal{E}xt^i_{\Ocal_Z}(\mathcal{F}_Z,\Ocal_{Z_0})\otimes k(z)=\begin{cases} 0,& i\ \text{odd}\\ C(-i/2),&  i\geq 2\ \text{even}.\end{cases}\]
Obviously, for $i\neq0$, there are no non-trivial $\check T$-invariants.

It remains to show that ${\rm Hom}_{[X/\check G]}(\Ocal_{[X_0/\check G]},{\mathcal{G}})\cong\mathfrak{Z}$, and to identify the basis vector. 
If $\mathcal{G}=\mathcal{O}_{[X_0/\check G]}$ this is clear, and a $\mathfrak{Z}$-basis is clearly given by the identity. 
If $\mathcal{G}=\Ocal_{[X/\check G]}$, one computes that the pull back of the morphism $(\ref{basemorphismGL2})$ to $Y$ specializes to the pullback of a basis vector of $\mathcal{H}om_{\Ocal_{Z}}(\Ocal_{Z_0},\Ocal_Z)\otimes k(\gamma(y))$ at every point of $y\in Y$. The claim easily follows from this.
\end{proof}

\begin{cor}\label{extgroupcomputation2}
Let $D_1,D_2\in \{\Hcal_G e_K,\Hcal_G e_{\rm st}\}$. The functor $R_G$ induces isomorphisms
\[{\rm Ext}^i_{\Hcal_G}(D_1,D_2)\longrightarrow {\rm Ext}^i_{[X/\check G]}(R_G(D_1),R_G(D_2)).\]
\end{cor}
\begin{proof}
Note that $\Hcal_G e_K$ and $\Hcal_G e_{\rm st}$ are projective and ${\rm Hom}_{\Hcal_G}(D_1,D_2)\cong\mathfrak{Z}$.
By Proposition \ref{extgroupcomputation1} the claim is true for $i\neq 0$ and we are left to show that in degree $0$ the canonical morphism identifies basis vectors. This is clear if $D_1=D_2$. 
Let us write $\gamma_1:\Hcal_G e_K\rightarrow \Hcal_G e_{\rm st}$ and $\gamma_2:\Hcal_G e_{\rm st}\rightarrow \Hcal_G e_K$ for choices of basis vectors and let $f\in\mathfrak{Z}$ as defined before Proposition \ref{extgroupcomputation1}. Then, up to scalars in $\mathfrak{Z}^\times$, we have
\begin{equation}\label{eqncomposition}\gamma_2\circ \gamma_1=f\cdot {\rm id}_{\Hcal_G e_K}\ \text{and}\ \gamma_1\circ \gamma_2=f\cdot {\rm id}_{\Hcal_G e_{\rm st}}.
\end{equation}
Writing $\delta_i=R_G(\gamma_i)$ one checks that the equalities $$\delta_2\circ \delta_1=f\cdot {\rm id}_{\Ocal_{[X_0/\check G]}}\ \text{and}\ \delta_1\circ \delta_2=f\cdot {\rm id}_{\Ocal_{[X/\check G]}}$$
enforce that
\begin{align*}
\delta_1&\in{\rm Hom}_{[X/\check G]}(\Ocal_{[X_0/\check G]},\Ocal_{[X/\check G]})\ \text{and} \\\delta_2&\in {\rm Hom}_{[X/\check G]}(\Ocal_{[X/\check G]},\Ocal_{[X_0/\check G]})
\end{align*}
are basis vectors. 
\end{proof}
\begin{proof}[Proof of Theorem \ref{theoGL2}]
We show that 
\[R_G:{\bf D}^b(\Hcal_G\text{-mod}_{\text{fg}})\longrightarrow {\bf D}^b_{\rm Coh}([X_{\check G}/\check G])\]
is fully faithful. The general case then follows from a limit argument as in Remark \ref{remaboutconj} (a).

By standard arguments the proof boils down to Corollary \ref{extgroupcomputation2}: let $D_1^\bullet, D_2^\bullet$ be complexes in ${\bf D}^b(\Hcal_G\text{-mod}_{\text{fg}})$. We may choose representatives of $D_i^\bullet$ consisting of bounded complexes whose entries are direct sums of copies of $\Hcal_G e_K$ and $\Hcal_G e_{\rm st}$.
Assume first $D_1^\bullet=\Hcal_G e_K$ or $\Hcal_G e_{\rm st}$ concentrated in degree $0$. We prove the claim by induction on the length of $D_2^\bullet$. 
By Corollary \ref{extgroupcomputation2} the claim is true if $D_2^\bullet$ has length $0$, i.e.~if $D_2^\bullet$ is concentrated in a single degree.
Assume the claim is true for all complexes of length $\leq m$ and let $D_2^\bullet$ be a complex in degrees $[r,r+m+1]$ for some $r\in\mathbb{Z}$. Then $D_2^\bullet$ can be identified with the mapping cone of a morphism of complexes 
\[D_2^r[-r]\longrightarrow \tilde{D}_2^\bullet\]
with $\tilde{D}_2^\bullet$ concentrated in degree $[r,r+m]$. The claim follows from the induction hypothesis, Corollary \ref{extgroupcomputation2} and the long exact cohomology sequence.

The general case follows by a similar induction on the length of $D_1^\bullet$.
\end{proof}

\begin{expl}\label{exampleSL2}
We finish this section with some remarks about the case $G={\rm SL}_2(F)$ and $\check G={\rm PGL}_2$. As already pointed out above, Conjecture \ref{GLnsurjection} fails for ${\rm SL}_2$. 

In this case $\check T=\mathbb{G}_m$ and $W=\mathcal{S}_2$. Let us write $$\widetilde{{\rm PGL}}_2\subset {\rm PGL}_2\times \check G/\check B={\rm PGL}_2\times\mathbb{P}^1$$
for the Grothendieck resolution of ${\rm PGL}_2$, that is for the closed subscheme of $(g,L)\in {\rm PGL}_2\times\mathbb{P}^1$ such that the line $L$ is stable under $g$. We consider the canonical diagram
\[\begin{xy}
\xymatrix{
&& \widetilde{\rm PGL}_2 \ar[dll]_f \ar[dl] \ar[ddl]\\
{\rm PGL}_2 \ar[d]& {\rm PGL}_2\times_{\mathbb{G}_m/\mathcal{S}_2} \mathbb{G}_m \ar[d]\ar[l]\\
\mathbb{G}_m/\mathcal{S}_2 & \mathbb{G}_m. \ar[l]
}
\end{xy}
\]
Again it is well known that $Rf_\ast\Ocal_{\widetilde{\rm PGL}_2}$ is concentrated in degree zero and in fact locally free (necessarily of rank $2$). However the canonical map
\begin{equation}\label{PGL2remarkcanmap}
\Ocal_{{\rm PGL}_2}^2=\Ocal_{{\rm PGL}_2}\otimes_{C[T]}C[X,X^{-1}]\longrightarrow f_\ast\Ocal_{\widetilde{\rm PGL}_2}=Rf_\ast\Ocal_{\widetilde{\rm PGL}_2},
\end{equation}
where $C[T]=\Gamma(\mathbb{G}_m/\mathcal{S}_2, \Ocal_{\mathbb{G}_m/\mathcal{S}_2})$ and $C[X,X^{-1}]=\Gamma(\mathbb{G}_m,\Ocal_{\mathbb{G}_m})$, is not surjective. 
Indeed, consider the point $x\in {\rm PGL}_2$ given by the class of the diagonal matrix ${\rm diag}(-1,1)$. Then $f^{-1}(x)=\{0,\infty\}$ is the disjoint union of two points and hence
$$f_\ast\Ocal_{\widetilde{\rm PGL}_2}\otimes k(x)=k(x)\times k(x)$$
as a $k(x)$-algebra. On the other hand the morphism $C[T]\rightarrow C[X,X^{-1}]$ is given by $T\mapsto X+X^{-1}$ and $x$ maps to the point $\{T=-2\}$ in $\Spec C[T]=\mathbb{G}_m/\mathcal{S}_2$. 
We find that the canonical map $(\ref{PGL2remarkcanmap})$ specializes to
\[k(x)\otimes_{C[T]}C[X,X^{-1}]=C[X]/(X+1)^2\longrightarrow C\times C=k(x)\times k(x)\]
which can not be surjective. In fact ${\rm Pic}({\rm PGL}_2)={\rm Hom}(\pi_1({\rm PGL_2}),\mathbb{G}_m)=\mathbb{Z}/2\mathbb{Z}$ and hence there is (up to isomorphism) a unique non-trivial line bundle $\mathcal{L}$ (which comes with a canonical ${\rm PGL}_2$-equivariant structure). In fact $\mathcal{L}$ can be identified with the ideal sheaf of the closed subscheme $\check G\cdot x$ and $\mathcal{L}^\vee\cong \mathcal{L}$. It is not hard to show that 
\begin{equation}\label{PGL2computation} f_\ast\Ocal_{\widetilde{\rm PGL}_2}=\mathcal{O}_{{\rm PGL}_2}\oplus \mathcal{L}.
\end{equation}
Let us also mention what $R\tilde \beta_{B,\ast}\Ocal_{{X}_{\check B}}$ looks like. 
As in the case of ${\rm GL}_2$ the scheme $X_{\check G}$ has to irreducible components $X_{\check G,0}$ and $X_{\check G,1}$, where $N=0$ on $X_{\check G,0}$ .Using $(\ref{PGL2computation})$ we obtain a decomposition
\begin{equation}\label{PGL2decomposition}R\tilde \beta_{B,\ast}\Ocal_{{X}_{\check B}}\cong \Ocal_{X_{\check G}}\oplus \mathcal{F}_0\cong \mathcal{F}\oplus \Ocal_{X_{\check G,0}},\end{equation}
where $\mathcal{F}$ respectively $\mathcal{F}_0$ is the twist of $\Ocal_{X_{\check G}}$ respectively $\Ocal_{X_{\check G,0}}$ by the pullback of the non-trivial line bundle $\mathcal{L}$ on $\check G={\rm PGL}_2$.
Indeed the same computation as for ${\rm GL}_2$ yields that 
\[{\rm Ext}^1_{[X_{\check G}/\check G]}(\mathcal{F}_0,\Ocal_{X_{\check G}})=0={\rm Ext}^1_{[X_{\check G}/\check G]}(\Ocal_{X_{\check G,0}},\mathcal{F})\]
and we can use $(\ref{PGL2computation})$ and the canonical morphism $\mathcal{O}_{X_{\check G}}\rightarrow R\tilde\beta_{B,\ast}\mathcal{O}_{\tilde X_{\check B}}$ to obtain a short exact sequence
\[0\longrightarrow \mathcal{O}_{X_{\check G}}\longrightarrow R\tilde\beta_{B,\ast}\mathcal{O}_{\tilde X_{\check B}}\longrightarrow \mathcal{F}_0\longrightarrow 0\]
which has to split using the computation of Ext-groups. Note that the pullback of $\mathcal{L}$ to $\widetilde{\rm PGL}_2$ is the trivial linde bundle, and hence twisting this sequence by the line bundle $\mathcal{F}$ and using the projection formula, we obtain a short exact sequence
\[0\longrightarrow \mathcal{F}\longrightarrow R\tilde\beta_{B,\ast}\mathcal{O}_{\tilde X_{\check B}}\longrightarrow \mathcal{O}_{X_{\check G,0}}\longrightarrow 0\]
which has to split as well. 

We point out that in light of the main conjecture the decompositions $(\ref{PGL2decomposition})$ correspond to the two decompositions
\[\cind_I^G \mathbf{1}_I=\cind_{K_1}^G \mathbf{1}_{K_1}\oplus  \cind_{K_1}^G {\rm st}_{K_1}=\cind_{K_1}^G \mathbf{1}_{K_2}\oplus  \cind_{K_2}^G {\rm st}_{K_2}\]
of representation of $G={\rm SL}_2(F)$, where $I$ is a choice of an Iwahori and $K_1$ and $K_2$ are the two non-conjugate hyperspecial subgroups of $G$ with $I=K_1\cap K_2$. Moreover $K_1\cong K_2\cong {\rm SL}_2(\Ocal_F)$ and ${\rm st}_{K_i}$ is the respective inflation of the finite dimensional Steinberg representation of ${\rm SL}_2(k_F)$.

Let us finally comment on the comparison with the Emerton-Helm construction (for a choice of a Whittaker datum $\psi$). Indeed, choosing $\psi$ we can again compute that $\Hcal_G e_{\psi}=(\cind_N^G\psi)^I$ for some idempotent element $e_\psi$ and this module is free of rank $1$ over $\Hcal_T$. We can define an "Emerton-Helm family" $\mathcal{V}_{G,\psi}$ as the quotient 
\[(\cind_N^G\psi)_{[T,1]}\otimes_{\mathfrak{Z}}\Ocal_{X_{\check G}}\longrightarrow \mathcal{V}_{G,\psi}\]
with prescribed fibers at the generic points of $X_{\check G}$. 
The fiber of $\mathcal{V}_{G,\psi}$ at the point $(x,0)\in X_{\check G}$, for $x\in{\rm PGL}_2$ as above, is the representation $\cind_N^G\psi\otimes_{\mathfrak{Z}}k(x)$ which is the unique non-split extension
$$0\longrightarrow \pi_{x}^{\psi'\text{-}{\rm gen}}\longrightarrow \mathcal{V}_{G,\psi}\otimes k((x,0)) \longrightarrow \pi_{x}^{\psi\text{-}{\rm gen}}\longrightarrow 0$$
where $\pi_{x}^{\psi\text{-}{\rm gen}}$ is the $\psi$-generic representation on the L-packet defined by $(x,0)$ and $\psi'$ is (a choice of) the Whittaker datum not conjugate to $\psi$.  

On the other hand we obtain a diagram
\[\begin{xy}
\xymatrix{
(\cind_N^G\psi)^I\otimes_{\mathfrak{Z}}\mathcal{O}_{X_{\check G}}=\mathcal{O}_{X_{\check G}}\otimes_{\mathfrak{Z}}\mathcal{O}_{\check T} \ar[r] \ar[dr]& \mathcal{M}_{G,\psi}=(\mathcal{V}_{G,\psi})^I \ar@{.>}[d]^g \\
& \tilde\beta_{B,\ast}\mathcal{O}_{\tilde{X}_{\check B}} 
}
\end{xy}\]
of $\mathcal{O}_{X_{\check G}}$-modules. The morphism $g$ is an injection and induces an isomorphism on the open complement $U$ of the Cartier divisor $\check G\cdot(x,0)\subset X_{\check G}$, and 
\[\tilde\beta_{B,\ast}\mathcal{O}_{\tilde{X}_{\check B}} \subset j_{U,\ast}(\tilde\beta_{B,\ast}\mathcal{O}_{\tilde{X}_{\check B}}|_U)=j_{U,\ast}(\mathcal{M}_{G,\psi}|_U)\]
is stable under the action of $\mathcal{H}_G$, where $j_U:U\hookrightarrow X_{\check G}$ is the canonical embedding.  

This construction equips $\tilde\beta_{B,\ast}\mathcal{O}_{\tilde{X}_{\check B}} =R\tilde\beta_{B,\ast}\mathcal{O}_{\tilde{X}_{\check B}}$ with an action of $\mathcal{H}_G$, depending on the choice of the Whittaker datum $\psi$, such that $$R_G^\psi(-)={}^t(-)\otimes^L_{\Hcal_G}\tilde\beta_{B,\ast}\mathcal{O}_{\tilde{X}_{\check B}}$$ is the desired functor in the case of $G={\rm SL}_2$. Computing the fibers of $\tilde\beta_{B,\ast}\mathcal{O}_{\tilde{X}_{\check B}}$ we find
\[\tilde\beta_{B,\ast}\mathcal{O}_{\tilde{X}_{\check B}}\otimes k((x,0))=(\pi_{x}^{\psi\text{-}{\rm gen}})^I\oplus (\pi_{x}^{\psi'\text{-}{\rm gen}})^I,\]
and the action of the centralizer of $x$ is trivial on $(\pi_{x}^{\psi\text{-}{\rm gen}})^I$ and non-trivial on $(\pi_{x}^{\psi'\text{-}{\rm gen}})^I$.
\end{expl}

\subsection{Calculation of examples}
We finish by computing the image of some special representations under the functor $R_G$ defined in $(\ref{EHfunctor})$. In particular we are in the situation $G={\rm GL}_n(F)$ and $\check G$ is the algebraic group $\GL_n$ over $C$. For simplicity we assume that $C$ is algebraically closed. We fix the choice of the diagonal torus $\mathbb{T}$ and the upper triangular Borel subgroup $\mathbb{B}$.

Let $x=(\phi,N)\in X_{\check G}$. In the examples calculated in this section we will assume that $\phi$ is regular semi-simple. 
As in Remark \ref{remaboutconj} (b) we write $$X_{\check G,[\phi,N]}^\circ=\check G\cdot x$$ for the $\check G$-orbit of $(\phi,N)$ and $X_{\check G,[\phi,N]}$ for its closure. 

\begin{theo}\label{computeRGLLmod}
Let $(\phi,N)\in X_{\check G}(C)$ and assume that $\phi$ is regular semi-simple. Then
\[R_G({\rm LL^{mod}}(\phi,N))=\Ocal_{[X_{\check G,[\phi,N]}/\check G]}.\]
\end{theo}

To prove this, we will use compatibility with parabolic induction. Hence the main step will be to calculate the image of the generalized Steinberg representations. 

Let $\chi:\mathfrak{Z}\rightarrow C$ be the character defined by the characteristic polynomial of $\phi$. 
We write $\hat{\mathfrak{Z}}_\chi$ for the completion of $\mathfrak{Z}$ with respect to the kernel $\mathfrak{m}_\chi$ of $\chi$ and $$\hat \Hcal_{G,\chi}=\Hcal_G\otimes_{\mathfrak{Z}}\hat{\mathfrak{Z}}_\chi$$ for the $\mathfrak{m}_\chi$-adic completion of $\Hcal_G$. Similarly, if $M\subset G$ is a Levi subgroup, we write $\hat\Hcal_{M,\chi}$ for the corresponding completion of $\Hcal_M$. 

Assume that $\phi={\rm diag}(\phi_1,\dots,\phi_n)$. For $w\in W=\mathcal{S}_n$ we write $w\phi$ for the diagonal matrix ${\rm diag}(\phi_{w(1)},\dots,\phi_{w(n)})$. 
We use the notation $\delta_w$ to denote the $\Hcal_T$-module defined by the unramified character ${\rm unr}_{w\phi}$ (i.e.~the residue field at the point $w\phi\in{\rm Spec}\,\Hcal_T$), and $\hat \delta_w$ to denote the completion of $\Hcal_T$ at the point $w\phi\in{\Spec}\, \Hcal_T$. 
Then $\delta_w$ and $\hat\delta_w$ are $\hat \Hcal_{T,\chi}$-modules. 

We recall intertwining operators for parabolic induction:
Let $\mathbb{P},\mathbb{P'}\subset \mathbb{G}$ be parabolic subgroups (containing $\mathbb{T}$) with Levi subgroups $\mathbb{M}$ and $\mathbb{M}'$ and let $w\in W$ such that $M'=wMw^{-1}$. 
Let $\pi$ and $\pi'$ be smooth representations of $M$ respectively $M'$ and let $f:\pi^w\rightarrow \pi'$ be a morphism of $M'$-representations. Then there is a canonical morphism of $G$-representations 
\[F(w,f)=F_G(w,f):\iota_{\overline P}^G\pi\longrightarrow \iota_{\overline{P}'}^G\pi'\]
associated to $f$ (and similarly for $\iota_{P}^G$ and $\iota_{P'}^G$). Moreover, this construction extends to (morphisms of) complexes of $M$- respectively $M'$-representations. 
We also note that the formation of these intertwining operators is transitive in the following sense:
Let $\mathbb{P}_1,\mathbb{P}'_1\subset\mathbb{P}\subset\mathbb{G}$ be parabolic subgroups. Let $\mathbb{M}_1,\mathbb{M}'_1$ and $\mathbb{M}$ denote the corresponding Levi quotients and let $\mathbb{P}_{1,M}$ be the image of $\mathbb{P}_1$ in $\mathbb{M}$ (and similarly $\mathbb{P}'_{1,M}$).
Let $w\in W_M\subset W$ be a Weyl group element such that $wM_1w^{-1}=M'_1$ (as subgroups of $G$, and hence also as subgroups of $M$). Moreover, let $\pi$ be a representation of $M_1$ and $\pi'$ be a representation of $M'_1$ and $f:\pi^w\rightarrow \pi'$ be a morphism of $M'_1$-representations. Then, under the canonical identifications $$\iota_{\overline{P}}^G\iota_{\overline{P}_{1,M}}^M(\pi)=\iota_{\overline{P}_1}^G\pi\ \ \text{and}\ \ \iota_{\overline{P}}^G\iota_{\overline{P}'_{1,M}}^M(\pi')=\iota_{\overline{P}'_1}^G\pi',$$ the morphism $\iota_{\overline{P}}^G(F_M(w,f))$ is identified with $F_G(w,f)$.

Now fix $\lambda\in C^\times$ and let $\phi={\rm diag}(\lambda, q^{-1}\lambda,\dots, q^{-(n-1)}\lambda)$. 
For $w,w'\in W$ the identity of $\delta_{w'}$ respectively $\hat\delta_{w'}$ induces  intertwining operators 
\begin{equation}\label{intertwining}
\begin{aligned}
f(w,w'):\iota_{\overline{B}}^G(\delta_w)&\longrightarrow \iota_{\overline{B}}^G(\delta_{w'})\\
\hat f(w,w'):\iota_{\overline{B}}^G(\hat\delta_w)&\longrightarrow \hat \iota_{\overline{B}}^G(\hat\delta_{w'}).
\end{aligned}
\end{equation}
Note that these morphisms are isomorphisms (with inverse $f(w',w)$ respectively $\hat f(w',w)$) if and only if for each $i$ the entries $q^{-i}\lambda$ and $q^{-(i+1)}\lambda$ appear in the same order in $w\phi$ and $w'\phi$.
Moreover, 
\begin{equation}\label{Homofinducedrepen}
{\rm Hom}_{\hat{\mathfrak{Z}}_\chi[G]}(\iota_{\overline{B}}^G\hat\delta_w,\iota_{\overline{B}}^G\hat\delta_{w'})=\hat{\mathfrak{Z}}_\chi\hat f(w,w')
\end{equation}
is a free $\hat{\mathfrak{Z}}_\chi$-module of rank $1$.
We define the (universal) deformation of the generalized Steinberg representation 
$$\hat{\rm St}(\lambda,r)=\iota_B^G\big(\delta_B^{-1/2}\otimes \widehat{\rm unr}_\lambda|-|^{(n-1)/2}\big)\big/\sum_{B\subsetneq P\subseteq G}\iota_P^G\big(\delta_P^{-1/2}\otimes\widehat{\rm unr}_\lambda|-|^{(n-1)/2}\big).$$
Here we write $\widehat{\rm unr}_\lambda$ for the universal (unramified) deformation of the character ${\rm unr}_{\lambda}$ and denote the target of $\widehat{\rm unr}_\lambda$ by $C[\![t]\!]$. Then
\begin{align*}
\widehat{\rm unr}_\lambda\otimes_{C[\![t]\!]}C[\![t]\!]/(t)&={\rm unr}_\lambda,\\
\hat{\rm St}(\lambda,r)\otimes_{C[\![t]\!]}C[\![t]\!]/(t)&={\rm St}(\lambda,r).
\end{align*}
Note that by definition $\hat{\rm St}(\lambda,n)$ is a quotient of $\iota_{\overline{B}}^G\hat\delta_{w_0}$, where $w_0\in W$ is the longest element.

By abuse of notation we will also write ${\rm St}(\lambda,n)$ and $\hat{\rm St}(\lambda,n)$ for the $\Hcal_G$- respectively $\hat\Hcal_{G,\chi}$-module given by the $I$-invariants in the respective representations. 
Similarly, we will continue to write $\iota_{\overline{B}}^G\delta_w$ etc.~for the Hecke modules defined by these representations. In the following we will only work with Hecke modules, hence no confusion should arise. 

We construct a projective resolution $\hat C^\bullet_{n,\lambda}$ of the $\hat\Hcal_{G,\chi}$-module $\hat {\rm St}(\lambda,n)$ concentrated in (cohomological) degrees $[-(n-1),0]$ such that all objects in the complex are direct sums of induced representations ${\iota}_{\overline{B}}^G\hat\delta_w$ and the differentials are given by combinations of the intertwining morphisms $(\ref{intertwining})$. We construct the complex by induction. 

\noindent If $n=2$, then $\phi={\rm diag}(\lambda, q^{-1}\lambda)$ and we consider the complex
$$ \xymatrix{\hat C^\bullet_{2,\lambda}:&  \hat C^{-1}_{2,\lambda}=\iota_{\overline{B}}^G(\hat\delta_1)  \ar[r]^{\hat f(1,s)}  & \hat C^{0}_{2,\lambda}=\iota_{\overline{B}}^G(\hat\delta_s), }$$
where $s\in \mathcal{S}_2$ is the unique non trivial element. It can easily be checked that the morphism $\hat f(1,s)$ is injective and that its cokernel is $\hat{\rm St}(\lambda,2)$.


\noindent Assume we have constructed $\hat C^\bullet_{n-1,\lambda}$. For $i=1,2$ consider the upper triangular block parabolic subgroup $\mathbb{P}_i\subset \mathbb{G}$ with Levi subgroup $\mathbb{M}_i$ such that $\mathbb{M}_1$ has block sizes $(n-1,1)$ and $\mathbb{M}_2$ has block sizes $(1,n-1)$. 
We consider $$D_1^\bullet=\hat C^\bullet_{n-1,\lambda}\widehat{\otimes}\ \widehat{\rm unr}_{q^{-(n-1)}\lambda}$$ as a complex of $\hat\Hcal_{M_1,\chi}$-modules and $$D_2^\bullet=\widehat{\rm unr}_{q^{-(n-1)}\lambda}\widehat{\otimes}\ \hat C_{n-1,\lambda}^\bullet$$ as a complex of $\hat\Hcal_{M_2,\chi}$-modules. Let $\sigma\in \mathcal{S}_n$ be the cycle $(12\dots n)$.  Then $M_1$ and $M_2$ satisfy $\sigma M_1\sigma^{-1}=M_2$ and the identity $(D_1^\bullet)^\sigma\rightarrow D_2^\bullet$, as a morphism of complexes of $\hat\Hcal_{M_2,\chi}$-modules, induces a morphism of complexes
\[\iota_{\overline{P}_1}^G D_1^\bullet\longrightarrow\iota_{\overline{P}_2}^G D_2^\bullet.\]
We define $\hat C_{n,\lambda}^\bullet$ as the mapping cone of this complex.
 Then $\hat C_{n,\lambda}^\bullet$ obviously is a complex in degree $[-(n-1),0]$ whose entries are (by transitivity of parabolic induction) direct sums of $\iota_{\overline{B}}^G(\hat\delta_w)$ for some $w\in W$ (each isomorphism class appearing exactly once) and the differentials are given by intertwining operators (by transitivity of intertwining operators).
\begin{lem}
The complex $\hat C^\bullet_{n,\lambda}$ is exact in negative degrees and
\[H^0(\hat C^\bullet_{n,\lambda})\cong \hat{\rm St}(\lambda,n).\]
\end{lem}
\begin{proof}
We proceed by induction. If $n=2$ this was already remarked above. 
Assume that $\hat C^\bullet_{n-1,\lambda}$ is quasi-isomorphic to $\hat{\rm St}(\lambda,n-1)$. 
Then $\hat C^\bullet_{n,\lambda}$ is quasi-isomorphic to the complex
\[\iota_{\overline{P}_1}^G\big(\hat{\rm St}(\lambda,n-1)\widehat{\otimes}\ \widehat{\rm unr}_{q^{-(n-1)}\lambda}\big)\longrightarrow\iota_{\overline{P}_2}^G\big(\widehat{\rm unr}_{q^{-(n-1)}\lambda}\widehat{\otimes}\ \hat{\rm St}(\lambda,n-1)\big)\]
in degrees $-1$ and $0$, where the morphism is given by the obvious intertwining map. One can easily check that this morphism is injective and its cokernel is $\hat{\rm St}(\lambda,n)$.
\end{proof}

Similarly to the definition of $\hat\Hcal_{G,\chi}$ we define $\mathfrak{m}_\chi$-adic completions on the side of stacks of L-parameters: let $\hat X_{\check G,\chi}$ denote the completion of $X_{\check G}$ along the pre-image of $\chi\in{\rm Spec}\,\mathfrak{Z}=\check T/W$ under the canonical morphism $X_{\check G}\rightarrow \check T/W$. This formal scheme is still equipped with an action of $\check G$ and we can form the stack quotient $[\hat X_{\check G,\chi}/\check G]$.
Similarly we write $\hat X_{\check P,\chi}$ and $\hat X_{\check M,\chi}$ for the corresponding completions of $X_{\check P}$ and $X_{\check M}$. 
The functor $R_G$ defined in $(\ref{EHfunctor})$ naturally extends to a functor 
\[\hat R_{G,\chi}:{\bf D}^+(\hat\Hcal_{G,\chi}\text{-mod}) \longrightarrow {\bf D}^+_{\rm QCoh}([\hat X_{\check G,\chi}/\check G]).\]
As a consequence of Theorem \ref{theocompwparabinductionregular} the functor $\hat R_{G,\chi}$ also satisfies compatibility with parabolic induction similarly to Conjecture \ref{mainconjecture} (ii), but for the induced morphism between the formal completions of the stacks involved.

Let us build a more explicit model of these stacks. We consider the closed formal subscheme 
\begin{equation}\label{formalschemeY}
\hat Y={\rm Spf}\ \big(C[u_1,\dots,u_{n-1}][\![t_1,\dots,t_n]\!]/((t_{i+1}-t_i)u_i)\big)\subset \hat X_{\check G,\chi},
\end{equation}
where $C[u_1,\dots,u_{n-1}][\![t_1,\dots,t_n]\!]/((t_{i+1}-t_i)u_i)$ is equipped with the $(t_1,\dots,t_n)$-adic topology. The embedding into $\hat X_{\check G,\chi}$ is defined by the $(\phi,N)$-module 
\begin{align*}
\phi_{\hat Y}&={\rm diag}(\lambda+t_1,q^{-1}(\lambda+t_2),\dots, q^{-(n-1)}(\lambda+t_n))\\
N_{\hat Y}(e_i)&=\begin{cases}u_ie_{i+1}, &i<n-1\\ 0, &i=n-1\end{cases}
\end{align*}
over $\hat Y$.
This formal scheme comes equipped with a canonical $\check T$-action (which is trivial on the $t_i$ and via the adjoint action on the $u_i$) such that $[\hat Y/\check T]=[\hat X_{\check G,\chi}/\check G]$. 

For $w\in W$ we define a closed $\check T$-equivariant formal subscheme $\hat Y(w)$ by adding the equation $u_i=0$ if $q^{-(i-1)}\lambda$ precedes $q^{-i}\lambda$ in $w\phi$. In particular $\hat Y({w_0})=\hat Y$ if $w_0\in W$ is the longest element.
We denote by $\hat X_{\check G,\chi}(w)$ the corresponding $\check G$-equivariant closed formal subscheme of $\hat X_{\check G,\chi}$.
\begin{lem}
There is an isomorphism
\[\hat R_{G,\chi}(\iota_{\overline{B}}^G\hat\delta_w)\cong \Ocal_{\hat{X}_{\check G,\chi}(w)},\]
where we view $\hat R_{G,\chi}(\iota_{\overline{B}}^G\hat\delta_w)$ as a $\check G$-equivariant sheaf on $\hat X_{\check G,\chi}$.
\end{lem}
\begin{proof}
This is a straight forward calculation using the compatibility of $\hat R_{G,\chi}$ with parabolic induction. 
\end{proof}
The lemma identifies the images of parabolically induced representations under $\hat R_{G,\chi}$. Next we identify the images of intertwining operators. For $w,w'\in W$ there is a canonical $\check T$-equivariant morphism 
\[\hat g(w,w'):\Ocal_{\hat Y(w)}\longrightarrow \Ocal_{\hat Y(w')}\]
defined as follows: let $I_w=\{i=1,\dots, n-1|\ q^{-(i-1)}\lambda\ \text{precedes}\ q^{-i}\lambda\ \text{in}\ w\phi\},$ i.e.~
\[\hat Y(w)={\rm Spf}\,C[u_1,\dots, u_{n-1}][\![t_1,\dots,t_n]\!]/(u_i,\,i\in I_w, (t_{i+1}-t_i)u_i,\,i\notin I_w)\]
and let us write
\[\hat Y(w,w')={\rm Spf}\,C[u_1,\dots, u_{n-1}][\![t_1,\dots,t_n]\!]/(u_i,\,i\in I_w\cap I_{w'}, (t_{i+1}-t_i)u_i,\,i\notin I_w\cap I_{w'})\] 
for the moment. Then, similarly to $(\ref{basemorphismGL2})$, multiplication by $\prod_{i\in I_w\backslash I_{w'}} (t_{i+1}-t_i)$ induces a morphism $\Ocal_{\hat Y(w)}\rightarrow \Ocal_{\hat Y(w,w')}$ and we define $\hat g(w,w')$ to be its composition with the canonical projection to $\Ocal_{\hat Y(w')}$.
\begin{lem}
For $w,w'\in W$ the $\hat{\mathfrak{Z}}_\chi$-module
\[{\rm Hom}_{[\hat X_{\check G,\chi}/\check G]}(\Ocal_{[\hat X_{\check G,\chi}(w)/\check G]},\Ocal_{[\hat X_{\check G,\chi}(w')/\check G]})={\rm Hom}_{[\hat Y/\check T]}(\Ocal_{[\hat Y(w)/\check T]},\Ocal_{[\hat Y(w')/\check T]})\]
is free of rank one with basis $\hat g(w,w')$.
\end{lem}
\begin{proof}
This is a straight forward computation.
\end{proof}
By the following theorem the images of the intertwining operators $\hat R_{G,\chi}(\hat f(w,w'))$ can be identified (up to isomorphism) with the morphisms $\hat g(w,w')$ just constructed.
\begin{theo}\label{theouniquenesshatR}
Let $\phi={\rm diag}(\lambda,q^{-1}\lambda,\dots, q^{-(n-1)}\lambda)\in\check T(C)$ and $\chi:\mathfrak{Z}\rightarrow C$ the character defined by the image of $\phi$ in $\check T/W$. The set of functors 
\[\hat R_{M,\chi}:\mathbf{D}^{+}(\hat \Hcal_{M,\chi}\text{-}{\rm mod})\longrightarrow \mathbf{D}^{+}_{\rm QCoh}([\hat X_{\check M,\chi}/\check M])\]
for standard Levi subgroups $M\subset G$, is uniquely determined (up to isomorphism) by requiring that they are $\hat{\mathfrak{Z}}_{M,\chi}$-linear, compatible with parabolic induction, and that $\hat R_{T,\chi}$ is induced by the identification 
\[\hat\Hcal_{T,\chi}\text{-}{\rm mod}\xrightarrow{\cong}{\rm QCoh}(\hat X_{\check T,\chi}).\]
More precisely, let $\hat R'_{G,\chi}$ be any functor satisfying these conditions. Then for each $w\in W$, there are isomorphisms $$\alpha_w:\hat R'_{G,\chi}(\iota_{\overline{B}}^G\hat\delta_w)\xrightarrow{\cong}\Ocal_{\hat X_{\check G,\chi}(w)}$$ such that for $w,w'\in W$ the diagram
\begin{equation}\label{checkcommutativity}
\begin{aligned}
\xymatrix{
\hat R'_{G,\chi}(\iota_{\overline{B}}^G\hat\delta_w) \ar[r]^{\alpha_w}\ar[d]_{\hat R'_{G,\chi}(\hat f(w,w'))}& \Ocal_{\hat X_{\check G,\chi}(w)}\ar[d]^{\hat g(w,w')}\\
\hat R'_{G,\chi}(\iota_{\overline{B}}^G\hat\delta_{w'}) \ar[r]^{\alpha_{w'}}& \Ocal_{\hat X_{\check G,\chi}(w')}
}
\end{aligned}
\end{equation}
commutes.
\end{theo}
\begin{rem}
(a) Note that we do not need to add the requirement $$\hat R_{G,\chi}((\cind_N^G\psi)^I_{[T,1]}\otimes_{\mathfrak{Z}}\hat{\mathfrak{Z}}_\chi)\cong \Ocal_{[\hat X_{\check G,\chi}/\check G]}$$
which also would be a consequence of the requirements in Conjecture \ref{mainconjecture}. In the situation considered here, there is an isomorphism
$$(\cind_N^G\psi)_{[T,1]}\otimes_{\mathfrak{Z}}\hat{\mathfrak{Z}}_\chi\cong \iota_{\overline{B}}^G\hat\delta_{w_0}$$ and hence the above isomorphism is automatic.\\
(b) It seems possible to compute that 
\[{\rm Ext}^i_{[\hat Y/\check T]}(\Ocal_{[\hat Y(w)/\check T]},\Ocal_{[\hat Y(w')/\check T]})=0\]
for $w,w'\in W$ and $i\neq 0$ by a similar explicit computation as in Proposition \ref{extgroupcomputation1}. This would imply the conjectured fully faithfulness of $\hat R_{G,\chi}$.
\end{rem}
\begin{proof}
Let us first justify that the second assertion implies the first. Note that 
\[\hat\Hcal_{M,\chi}=\hat\Hcal_{M,\chi}\otimes_{\hat \Hcal_{T,\chi}}\hat\Hcal_{T,\chi}=\hat\Hcal_{M,\chi}\otimes_{\hat \Hcal_{T,\chi}}\big(\bigoplus_{w\in W_M} \hat\delta_w\big)=\bigoplus_{w\in W_M}\iota_{\overline{B}_M}^M\hat\delta_w.\]
Using free resolutions of bounded above objects in $\mathbf{D}^+(\hat\Hcal_{M,\chi}\text{-mod})$ it is hence enough control the images of parabolically induced representations and the images of the intertwining operators. Then a limit argument deals with the general case.

Given $\hat R'_{G,\chi}$ as in the formulation of the theorem, compatibility with parabolic induction forces the existence of isomorphisms $\alpha_w$. Note that $\alpha_w$ is unique up to a unit in $\hat{\mathfrak{Z}}_\chi$. 
We claim that we can choose the isomorphisms such that the diagrams $(\ref{checkcommutativity})$ are commutative. 
In order to do so, we proceed by induction. By assumption the claim is true for $n=1$. We also make $n=2$ explicit. In this case we can identify 
\[\iota_{\overline{B}}^G\hat \delta_1=\hat\Hcal_{G,\chi}e_K\ \ \text{and}\ \ \iota_{\overline{B}}^G\hat \delta_s=\hat\Hcal_{G,\chi}e_{\rm st}.\]
One calculates that the intertwining operators $\hat f(1,s)$ and $\hat f(s,1)$ are identified with a $\hat{\mathfrak{Z}}_\chi$-basis of $${\rm Hom}_{\hat \Hcal_{G,\chi}}(\hat\Hcal_{G,\chi}e_K,\hat\Hcal_{G,\chi}e_{\rm st})\ \ \text{resp.}\ \ {\rm Hom}_{\hat \Hcal_{G,\chi}}(\hat\Hcal_{G,\chi}e_{\rm st},\hat\Hcal_{G,\chi}e_{K}).$$
Moreover, the compositions $\hat f(1,s)\circ \hat f(s,1)$ and $\hat f(s,1)\circ \hat f(1,s)$ are the multiplications with $f\in\hat{\mathfrak{Z}}^\times_\chi$, with $f$ as defined just before Proposition \ref{extgroupcomputation1}.
The calculation in the rank $2$ case, Proposition \ref{extgroupcomputation1}, yields the claim.

Assume now that the claim is true for $n-1$ and view $\mathcal{S}_{n-1}$ as the subgroup of $W=\mathcal{S}_n$ permuting the elements $1,\dots, n-1$.
Recall the parabolic subgroups $P_1$ and $P_2$ from the inductive construction of the complex $\hat C^\bullet_{n,\lambda}$. Using parabolic induction $\iota_{\overline{P}_1}^G$ and the induction hypothesis we may assume that we have constructed $\alpha_w$ for all $w\in\mathcal{S}_{n-1}\subset W$ such that the diagram $(\ref{checkcommutativity})$ commutes for all $w,w'\in\mathcal{S}_{n-1}$.
Let $\sigma=(12\dots n)$ as above. We first show that we can choose
\[\alpha_{\sigma w\sigma^{-1}}:\hat R'_{G,\chi}(\iota_{\overline{B}}^G\hat \delta_{\sigma w\sigma^{-1}})\xrightarrow{\cong}\Ocal_{\hat Y(\sigma w\sigma^{-1})} \]
such that $(\ref{checkcommutativity})$ commutes for the pairs $w,\sigma w\sigma^{-1}$ and $\sigma w \sigma^{-1},w$.
Let $\tau_{i,i+1}$ denote the transposition of $i$ and $i+1$. Inductively we define $w_1=\tau_{n,n-1}w\tau_{n,n-1}$ and $w_i=\tau_{n-i,n-i-1}w_{i-1}\tau_{n-i,n-i-1}$. Then the composition of intertwining operators
\begin{align*}
\iota_{\overline B}^G\hat \delta_w\longrightarrow \iota_{\overline B}^G\hat \delta_{w_1}\longrightarrow \dots \longrightarrow \iota_{\overline B}^G\hat \delta_{w_{n-1}}=\iota_{\overline B}^G\hat \delta_{\sigma w\sigma^{-1}}
\end{align*}
is identified with $\beta \hat f(w,\sigma w \sigma^{-1})$ for some $\beta\in\hat{\mathfrak{Z}}_\chi^\times$. Similarly, the composition 
$$\iota_{\overline B}^G\hat \delta_{\sigma w\sigma^{-1}}\longrightarrow \iota_{\overline B}^G\hat \delta_{w_{n-1}}\longrightarrow \dots \longrightarrow \iota_{\overline B}^G\hat \delta_{w_{1}}=\iota_{\overline B}^G\hat \delta_{w}$$
is identified with $\beta \hat f(\sigma w \sigma^{-1},w)$ for the same unit $\beta$. 
In this composition all the intertwining maps are isomorphisms, except for the morphisms
\[\iota_{\overline B}^G\hat \delta_{w_i}\longrightarrow\iota_{\overline B}^G\hat \delta_{w_{i+1}}\ \ \text{and}\ \ \iota_{\overline B}^G\hat \delta_{w_{i+1}}\longrightarrow\iota_{\overline B}^G\hat \delta_{w_{i}}\]
where the position of $n$ and $n-1$ in $(w_i(1),\dots, w_i(n))$ and $(w_{i+1}(1),\dots, w_{i+1}(n))$ is interchanged. By the computation in the two dimensional case and compatibility with parabolic induction this intertwining morphism is given by the multiplication with $\beta'(t_n-t_{n-1})$ for some unit $\beta'\in \hat{\mathfrak{Z}}_\chi^\times$ respectively by canonical projection multiplied with $\beta'$.
Modifying $\alpha_{\sigma w\sigma^{-1}}$ by $(\beta\beta')^{-1}$ we deduce the commutativity of the diagrams $(\ref{checkcommutativity})$ for the pairs $w,\sigma w\sigma^{-1}$ and $\sigma w \sigma^{-1},w$.

Now consider the general case. Note that for any $w\in W$ there exists $\tilde w\in\mathcal{S}_{n-1}$ such that 
\[\iota_{\overline{B}}^G\hat\delta_w\xrightarrow{\cong}\iota_{\overline{B}}^G\hat\delta_{\tilde w}\ \ \text{or} \ \ \iota_{\overline{B}}^G\hat\delta_w\xrightarrow{\cong}\iota_{\overline{B}}^G\hat\delta_{\sigma \tilde w\sigma^{-1}}.\]
Hence we can choose $\alpha_w$ such that all the diagrams $(\ref{checkcommutativity})$ commute, provided we can check commutativity of these diagrams for $w,w'\in \mathcal{S}_{n-1}\cup \sigma \mathcal{S}_{n-1}\sigma^{-1}$. If both elements $w,w'$ lie in $\mathcal{S}_{n-1}$ this follows from the induction hypothesis.
Let us check the claim for $w,w''\in\mathcal{S}_{n-1}$ and $w'=\sigma w''\sigma^{-1}$ (the argument in the other cases being similar). By $\hat{\mathfrak{Z}}_{\chi}$-linearity it is enough to check that 
\[
\xymatrix{
\hat R'_{G,\chi}(\iota_{\overline{B}}^G\hat\delta_w) \ar[r]^{\alpha_w}\ar[d]_{\hat R'_{G,\chi}(\gamma\hat f(w,w'))}& \Ocal_{\hat X_{\check G,\chi}(w)}\ar[d]^{\gamma\hat g(w,w')}\\
\hat R'_{G,\chi}(\iota_{\overline{B}}^G\hat\delta_{w'}) \ar[r]^{\alpha_{w'}}& \Ocal_{\hat X_{\check G,\chi}(w')}
}
\]
commutes for any choice of $0\neq \gamma\in\hat{\mathfrak{Z}}_\chi$. In particular we may check it for the element $\gamma$ defined by 
\[\hat f(w'',\sigma w''\sigma^{-1})\circ \hat f(w,w'')=\gamma \hat f(w,\sigma w''\sigma^{-1})=\gamma \hat f(w,w').\]
This follows from functoriality and the cases already treated above.
\end{proof}

We now continue to calculate the image $\hat R_{G,\chi}(\hat{\rm St}(\lambda,n))$ of the deformed Steinberg representation.
Let us write $$\hat Y^{\rm St}\cong{\rm Spf}\, C[u_1,\dots, u_{n-1}][\![t]\!] \subset \hat Y$$ for the formal subscheme defined by $t:=t_1=\dots=t_n$. We write $\hat X_{\check G,\chi}^{\rm St}$ for the corresponding $\check G$-equivariant scheme.

We inductively construct a 
$\check T$-equivariant resolution $\hat E^\bullet_{n,\lambda}$ of $\Ocal_{\hat Y^{\rm St}}$.\\
\noindent If $n=2$ we set
$$ \xymatrix{\hat E^\bullet_{2,\lambda}:&  \hat E^{-1}_{2,\lambda}=\Ocal_{\hat Y(1)}  \ar[r]^{\cdot (t_2-t_1)}  & \hat E^{0}_{2,\lambda}=\Ocal_{\hat Y(s)},   }$$
where again $s\in \mathcal{S}_2$ is the unique non trivial element. It can easily be checked that this morphism is injective and its cokernel is $\Ocal_{\hat Y^{\rm St}}$

\noindent Assume that $\hat E_{n-1,\lambda}^\bullet$ is constructed, then consider the morphism of complexes 
\[\hat E^\bullet_{n-1,\lambda}[u_{n-1}][\![t_n]\!]/(u_{n-1})\xrightarrow{\cdot (t_n-t_{n-1})}\hat E^\bullet_{n-1,\lambda}[u_{n-1}][\![t_n]\!]/((t_n-t_{n-1})u_{n-1})\]
on $\hat Y$ and define $\hat E^\bullet_{n,\lambda}$ to be its mapping cone. Here we write
\begin{align*}
\hat E^\bullet_{n-1,\lambda}[u_{n-1}][\![t_n]\!]/(u_{n-1})&=\hat E^\bullet_{n-1,\lambda}\otimes_{A_{n-1}}A_{n}/(u_{n-1}), \ \ \text{and}\\
\hat E^\bullet_{n-1,\lambda}[u_{n-1}][\![t_n]\!]/((t_n-t_{n-1})u_{n-1})&=\hat E^\bullet_{n-1,\lambda}\otimes_{A_{n-1}} A_n,
\end{align*}
by slight abuse of notation, where we write $$A_j=C[u_1,\dots, u_{j-1}]\![[t_1,,\dots, t_{j}]\!]/((t_{i+1}-t_i)u_i).$$
\begin{lem}
The complex  $\hat E^\bullet_{n,\lambda}$ is exact in negative degrees and 
\[H^0(\hat E^\bullet_{n,\lambda})=\Ocal_{\hat Y^{\rm St}}.\]
\end{lem}
\begin{proof}
We proceed by induction. For $n=2$ the claim is clear. Assume the claim is true for $n-1$, then the long exact cohomology sequence implies that $\hat E^\bullet_{n,\lambda}$ is quasi-isomorphic to the complex
\[C[u_1,\dots, u_{n-2}][\![t,t_n]\!]\longrightarrow C[u_1,\dots, u_{n-2},u_{n-1}][\![t,t_n]\!]/((t_n-t)u_{n-1})\]
sending $1$ to $(t_n-t)$. The claim follows from this.
\end{proof}
Let us denote by $\hat{\mathcal{E}}_{n,\lambda}^\bullet$ the $\check G$-equivariant complex on $\hat X_{\check G,\chi}$ corresponding to the $\check T$-equivariant complex $\hat E^\bullet_{n,\lambda}$ under the identification $[\hat Y/\check T]=[\hat X_{\check G,\chi}/\check G]$.





%

\begin{cor}\label{lemidentifycomplexes}
There is an isomorphism of complexes
\begin{equation}\label{isoofcomplexes}
\hat R_{G,\chi}(\hat C^\bullet_{n,\lambda})\cong \hat{\mathcal{E}}^\bullet_{n,\lambda}.
\end{equation}
\end{cor}
\begin{proof}
We prove this using the inductive construction of both complexes. The case $n=1$ is trivial.
Assume now that $(\ref{isoofcomplexes})$ is true for $n-1$. 
Recall the parabolic subgroups $P_1$ and $P_2$ from the inductive construction of $\hat C^\bullet_{n,\lambda}$.\\
Let us write $G_{n-1}=\GL_{n-1}(F)$ and $B_{n-1}\subset G_{n-1}$ for the upper triangular Borel. Further let $\phi'={\rm diag}(\lambda,q^{-1}\lambda,\dots, q^{-(n-2)}\lambda)$. Similarly to the definition of $\delta_w$ and $\hat\delta_w$ using $w\phi$ we define $\delta'_w$ and $\hat\delta'_w$ using $w\phi'$ for $w\in \mathcal{S}_{n-1}$.
Then 
\begin{align*}
\iota_{\overline{P}_1}^G\big(\iota_{\overline{B}_{n-1}}^{G_{n-1}} \hat\delta'_w\widehat\otimes \ \widehat{\rm unr}_{q^{-(n-1)}\lambda}\big)&=\iota_{\overline{B}}^G\hat\delta_w\\
\iota_{\overline{P}_2}^G\big(\widehat{\rm unr}_{q^{-(n-1)}}\widehat\otimes \ \iota_{\overline{B}_{n-1}}^{G_{n-1}} \hat\delta'_w\lambda\big)&=\iota_{\overline{B}}^G\hat\delta_{\sigma w\sigma^{-1}},
\end{align*}
and the intertwining operator between the representations on the right hand side translates to the intertwining operator $\hat f(w,\sigma w\sigma^{-1})$ under this identification.\\
By the same inductive construction, we assume that each entry of $\hat C^\bullet_{n-1,\lambda}$ is a direct sum of representations $\iota_{\overline{B}_{n-1}}^{G_{n-1}}\hat\delta'_w$ for $w\in \mathcal{S}_{n-1}$.
By Theorem \ref{theouniquenesshatR} the morphism $$\hat R_{G,\chi}(\iota_{\overline{P}_1}^G D_1^\bullet)\longrightarrow\hat R_{G,\chi}(\iota_{\overline{P}_2}^G D_2^\bullet)$$ is (up to a unit) identified with the multiplication by $(t_n-t_{n-1})$. The inductive construction of $\hat E^\bullet_{n,\lambda}$ hence implies the claim.
\end{proof}

\begin{cor}\label{corimageofSteinberg}
Let $\lambda\in C^\times$ and let $(\phi,N)\in X_{\check G}(C)$ be the L-parameter defined by $(C^n,\phi,N)={\rm Sp}(\lambda,n)$. Then
\[R_G({\rm St}(\lambda,n))\cong \Ocal_{X_{\check G,[\phi,N]}},\]
where ${\rm St}(\lambda,n)={\rm LL}(\phi,N)={\rm LL^{mod}}(\phi,N)$ is the generalized Steinberg representation.
\end{cor}
\begin{proof}
The corollary above implies $$\hat R_{G,\chi}(\hat{\rm St}(\lambda,n))=\Ocal_{\hat X_{\check G,\chi}^{\rm St}}$$
as $\check G$-equivariant sheaves.
Moreover, we have 
\begin{align*}
\hat{\rm St}(\lambda,n)\otimes^L_{C[\![t]\!]}C[\![t]\!]/(t)&=\hat{\rm St}(\lambda,n)\otimes_{C[\![t]\!]}C[\![t]\!]/(t)={\rm St}(\lambda,n),\\
\Ocal_{\hat X_{\check G,\chi}^{\rm St}}\otimes^L_{C[\![t]\!]}C[\![t]\!]/(t)&=\Ocal_{\hat X_{\check G,\chi}^{\rm St}}\otimes_{C[\![t]\!]}C[\![t]\!]/(t)=\Ocal_{X_{\check G,[\phi,N]}}.
\end{align*}
The center $\hat{\mathfrak{Z}}_\chi$ acts on $\hat{\rm St}(\lambda,n)$ and $\Ocal_{\hat X_{\check G,\chi}^{\rm St}}$ via a surjection $$\hat{\mathfrak{Z}}_\chi\longrightarrow C[\![t]\!].$$
Choosing a pre-image $g$ of $t$ we obtain isomorphisms
\begin{align*}
\hat{\rm St}(\lambda,n)\otimes^L_{C[\![t]\!]}C[\![t]\!]/(t)&=\hat{\rm St}(\lambda,n)\otimes^L_{\hat{\mathfrak{Z}}_\chi}\hat{\mathfrak{Z}}_\chi/(g)\\
\Ocal_{\hat X_{\check G,\chi}^{\rm St}}\otimes^L_{C[\![t]\!]}C[\![t]\!]/(t)&=\Ocal_{\hat X_{\check G,\chi}^{\rm St}}\otimes^L_{\hat{\mathfrak{Z}}_\chi}\hat{\mathfrak{Z}}_\chi/(g).
\end{align*}
The claim now follows from $\hat{\mathfrak{Z}}_\chi$-linearity of $\hat R_{G,\chi}$.
\end{proof}
\begin{rem}\label{remimageofLL}
With some extra effort one can use a similar strategy  to compute the images of ${\rm LL}(\phi,N)$, where $\phi={\rm diag}(\lambda,q^{-1}\lambda,\dots, q^{-(n-1)}\lambda)$ and $N$ is an arbitrary endomorphism such that $(\phi,N)\in X_{\check G}$.
Recall that ${\rm LL}(\phi,N)$ is the unique simple quotient of ${\rm LL^{mod}}(\phi,N)$. One needs to build a complex similar to $\hat C^\bullet_{n,\lambda}$ which is a resolution of ${\rm LL}(\phi,N)$. We omit the technical computation, and only describe the result.

Let us choose such $y=(\phi,N)\in Y\subset X_{\check G}$, where $Y\subset \hat Y$ is the closed subscheme $t_1=\dots=t_n=0$. We denote by $L(y)$ the sheaf of ideals defining the closed subscheme $$\bigcup_{\{i\mid u_i(y)=0\}} \{u_i=0\}\subset Y.$$
Obviously this is a $\check T$-equivariant line bundle and we write $\mathcal{L}(y)$ for the corresponding $\check G$-equivariant line bundle on $X_{\check G}$. Let us denote the number of $i\in\{1,\dots, n-1\}$ such that $u_i(y)=0$ by $l_y$. Then
\[R_G({\rm LL}(\phi,N))=\mathcal{L}(y)[l_y]\]
is the equivariant line bundle $\mathcal{L}(y)$ shifted to (cohomological) degree $-l_y$.
\end{rem}

\begin{proof}[Proof of Theorem \ref{computeRGLLmod}]
We assume that $\phi$ is an arbitrary regular semi-simple element and choose a decomposition 
\[(C^n,\phi,N)=\bigoplus_{i=1}^s {\rm Sp}(\lambda_i,r_i)\]
as in subsection \ref{modifiedLL}. Then 
\begin{equation*}
{\rm LL^{mod}}(\phi,N)=\iota_P^G \big({\rm St}(\lambda_1,r_1)\otimes\dots\otimes{\rm St}(\lambda_{s},r_s)\big)=\iota_{\overline{P}'}^G \big({\rm St}(\lambda_s,r_s)\otimes\dots\otimes{\rm St}(\lambda_{1},r_1)\big)
\end{equation*}
with the ordering of $(\ref{defnmodifiedLL})$. Here $P$ is a block upper triangular parabolic with Levi $M$ and we set $P'$ to be the block upper triangular parabolic with Levi $M'=w_0Mw_0^{-1}$, where $w_0\in W$ is the longest element. 
Write $\check M'=\GL_{r_s}\times\dots \times\GL_{r_1} $ and consider the morphisms  
\[\begin{aligned}
\alpha:{X}_{\check P'}&\longrightarrow X_{\check M'},\\
\beta:\tilde{X}_{\check P'}&\longrightarrow X_{\check G}.
\end{aligned}
\]
The choice of $\check M'\hookrightarrow \check P'$ defines an embedding $\iota:X_{\check M'}\hookrightarrow X_{\check P'}$. We will write $(x_s,\dots,x_1)\in X_{\check M'}$ for the point defined by 
\[{\rm Sp}(\lambda_s,r_s)\oplus \dots \oplus {\rm Sp}(\lambda_1,r_1),\]
and write $Z_{\check M'}(x_1,\dots,x_s)$ for the Zariski-closure of its $\check M'$-orbit $\check M'\cdot(x_s,\dots,x_1)$.
Then one easily checks that the choice of ordering of $\lambda_1,\dots,\lambda_s$ implies that $$\alpha^{-1}(Z_{\check M'}(x_s,\dots,x_1))=:Z_{\check P'}(x_s,\dots,x_1)$$ is the Zariski-closure of the $\check P'$-orbit of $\iota(x_s,\dots, x_1)$. Moreover, the choice of ordering implies that $\alpha$ is smooth along this pre-image. In particular 
\[L\alpha^\ast \Ocal_{Z_{\check M'}(x_s,\dots,x_1)}=\Ocal_{Z_{\check P'}(x_s,\dots,x_1)}.\]
Let $Z_{\check G}(x_s,\dots,x_1)\subset \tilde X_{\check P'}$ denote the $\check G$-invariant closed subscheme of $\tilde X_{\check P'}$ corresponding to the $\check P'$-invariant closed subscheme $Z_{\check P'}(x_s,\dots,x_1)\subset X_{\check P'}$. 
Using Corollary \ref{corimageofSteinberg} and compatibility of $R_G$ with parabolic induction, we are left to show that $$R\beta_\ast(\Ocal_{Z_{\check G}}(x_s,\dots,x_1))=\Ocal_{X_{\check G,[\phi,N]}}.$$
This follows, as the construction implies that $\beta$ maps $Z_{\check G}(x_s,\dots,x_1)$ isomorphically onto the Zariski-closure $X_{\check G,[\phi,N]}$ of the $\check G$-orbit $\check G\cdot (\phi,N)=\check G\cdot \iota(x_s\dots,x_1)$.
%
\end{proof}
We also remark that Theorem \ref{theouniquenesshatR} is true for all regular semi-simple elements $\phi$.
\begin{cor}\label{Corollaryuniquelydetermined}
Let $\phi\in\check T(C)$ be regular semi-simple and $\chi:\mathfrak{Z}\rightarrow C$ the character defined by the image of $\phi$ in $\check T/W$. The set of functors 
\[\hat R_{M,\chi}:\mathbf{D}^{+}(\hat \Hcal_{M,\chi}\text{-}{\rm mod})\longrightarrow \mathbf{D}^{+}_{\rm QCoh}([\hat X_{\check M,\chi}/\check M])\]
for standard Levi subgroups $M\subset G$, is uniquely determined (up to isomorphism) by requiring that they are $\hat{\mathfrak{Z}}_{M,\chi}$-linear, compatible with parabolic induction, and that $\hat R_{T,\chi}$ is induced by the identification 
\[\hat\Hcal_{T,\chi}\text{-}{\rm mod}={\rm QCoh}(\hat X_{\check T,\chi}).\]
\end{cor}
\begin{proof}
As in the proof of Theorem \ref{theouniquenesshatR} the images of $\hat R_{G,\chi}(\hat\delta_w)$ are uniquely determined up to isomorphism and it is enough to prove that the same is true for the images of intertwining operators. 
Without loss of generality we may assume
\[\phi={\rm diag}(\lambda_1,q^{-1}\lambda_1,\dots,q^{-(r_1-1)}\lambda_1,\dots \lambda_s,q^{-1}\lambda_s,\dots q^{-(r_s-1)}\lambda_s)\]
with $q^{-a}\lambda_i\neq q^{-b}\lambda_j$ for $i\neq j,\, a=0,\dots,r_i-1,\, b=0,\dots, r_j-1$, and $\lambda_i\neq q^{-r_j}\lambda_j$.
Let $M={\rm GL}_{r_1}(F)\times\dots\times{\rm GL}_{r_s}(F)$ be the block diagonal Levi subgroup with block sizes $(r_1,\dots, r_s)$ and $P$ the corresponding block upper triangular parabolic subgroup. 
Further let $\phi_i={\rm diag}(\lambda_i,q^{-1}\lambda_i,\dots,q^{-(r_i-1)}\lambda_i)\in{\rm GL}_{r_i}(F)$. For $w_i\in\mathcal{S}_{r_i}$ we write $\hat \delta^{(i)}_{w_i}$ for the universal unramified deformation of the character defined by $w_i\phi_i$.
Then, by means of an intertwining operator, every $\iota_{\overline{B}}^G\hat\delta_w$ is isomorphic to $$\iota_{\overline{P}}^G(\iota_{\overline{B}_M}^M(\hat\delta^{(1)}_{w_1}\otimes\dots\otimes\hat\delta^{(s)}_{w_s}))$$ for some $(w_1,\dots,w_s)\in \mathcal{S}_{r_1}\times\dots\times\mathcal{S}_{r_s}=W_M$.
As in the proof of Theorem \ref{theouniquenesshatR} we deduce that, given two functors $\hat R_{G,\chi}$ and $\hat R'_{G,\chi}$ satisfying the assumptions, it is enough to show that for all $w\in W_M$ there are isomorphisms
\[\alpha_w:\hat R_{G,\chi}(\iota_{\overline{B}}^G\hat\delta_w)\longrightarrow \hat R'_{G,\chi}(\iota_{\overline{B}}^G\hat\delta_w)\]
such that the diagrams
\[\xymatrix{
\hat R_{G,\chi}(\iota_{\overline{B}}^G\hat\delta_w) \ar[r]^{\alpha_w}\ar[d]_{\hat R_{G,\chi}(\hat f(w,w'))}& \hat R'_{G,\chi}(\iota_{\overline{B}}^G\hat\delta_w) \ar[d]^{\hat R'_{G,\chi}(\hat f(w,w'))}\\
\hat R_{G,\chi}(\iota_{\overline{B}}^G\hat\delta_{w'}) \ar[r]^{\alpha_{w'}}& \hat R'_{G,\chi}(\iota_{\overline{B}}^G\hat\delta_{w'})
}
\]
commute for all $w,w'\in W_M$. This follows from the statement of Theorem \ref{theouniquenesshatR} and transitivity of intertwining operators under parabolic induction.
\end{proof}

We finish by giving more details on the behavior of $R_G(\cind_K^G\sigma_\mathcal{P})$ in the three-dimensional case.
\begin{expl}\label{GL3example}
In the case $n=3$ there are three partitions $\mathcal{P}_{\rm min},\mathcal{P}_0,\mathcal{P}_{\rm max}$ of $n=3$. We have 
\begin{align*}
&m_{\mathcal{P}_{\rm min}}=m_{\mathcal{P}_{\rm max}}=1,\\
&m_{\mathcal{P}_0}=2,
\end{align*}
where the multiplicities are defined as in $(\ref{SchneiderZinkdecompo})$. The sheaves $R_G((\cind_K^G\sigma_{\rm min})^I)$ and $R_G((\cind_K^G\sigma_{\rm max})^I)$ are determined in Proposition \ref{imagesofprojectives}. 
Let us give a closer description of 
$$\mathcal{F}=R_G((\cind_K^G\sigma_{\mathcal{P}_0})^I).$$
As discussed in the Remark \ref{remRGcind}  the generic rank of $\mathcal{F}$ on $Z_{\check G,\mathcal{P}'}$ is $0$ if $\mathcal{P}'=\mathcal{P}_{\rm min}$, it is $1$ if $\mathcal{P}'=\mathcal{P}_0$ and it is $2$ if $\mathcal{P}'=\mathcal{P}_{\rm max}$. 

We describe the completed stalks $\hat{\mathcal{F}}_x$ as modules over the complete local rings $\hat{\Ocal}_{X_{\check G},x}$ for $C$-valued points $x=(\phi,N)\in X_{\check G}$. To simplify the exposition we restrict ourselves to regular semi-simple $\phi$. 
Recall that $X_{\check G,\mathcal{P}_0}=Z_{\check G,\mathcal{P}_0}\cup Z_{\check G,\mathcal{P}_{\rm max}}$ is a union of two irreducible components in this case. Moreover, recall that we write $X_{\check G,0}=Z_{\check G,\mathcal{P}_{\rm max}}$ for the irreducible component defined by $N=0$.

\noindent (a) Assume $x\in Z_{\check G,\mathcal{P}_{\rm min}}\backslash X_{\check G,\mathcal{P}_0}$, then, $\hat{\mathcal{F}}_x=0$.

\noindent (b) Assume $x\in Z_{\check G,\mathcal{P}_0}\backslash Z_{\check G,\mathcal{P}_{\rm max}}$, then $\hat{\mathcal{F}}_x\cong \hat{\Ocal}_{X_{\check G},x}$.

\noindent (c) Assume $x\in  Z_{\check G,\mathcal{P}_{\rm max}}\backslash Z_{\check G,\mathcal{P}_0}$, then $\hat{\mathcal{F}}_x\cong \hat{\Ocal}_{X_{\check G},x}^2$. 

\noindent (d) Assume $x\in Z_{\check G,\mathcal{P}_{\rm max}}\cap Z_{\check G,\mathcal{P}_0}$. Without loss of generality we may assume $\phi={\rm diag}(\lambda_1,\lambda_2,\lambda_3)$. As before we write $\chi:\mathfrak{Z}\rightarrow C$ for the character defined by the characteristic polynomial of $\phi$. 
Up to renumbering, we have to distinguish two cases:
\begin{enumerate}
\item[(d1)] $\lambda_2=q^{-1}\lambda_1$ and $\lambda_3\notin\{q^{-1}\lambda_2,q\lambda_1\}$. In this case $Z_{\check G,\mathcal{P}_0}$ and $Z_{\check G,\mathcal{P}_{\rm max}}$ are smooth at $x$. 
Moreover (using the notations introduced above) $$\cind_K^G\sigma_{\mathcal{P}_0}\otimes_\mathfrak{Z}\hat{\mathfrak{Z}}_{\chi}\cong\iota_{\overline{B}}^G\hat\delta_{w_1}\oplus \iota_{\overline{B}}^G\hat\delta_{w_2}$$ 
for some $w_1,w_2\in W$ and (with appropriate numeration)
$$\cind_K^G\sigma_{\mathcal{P}_{\rm min}}\otimes_\mathfrak{Z}\hat{\mathfrak{Z}}_{\chi}\cong\iota_{\overline{B}}^G\hat\delta_{w_1}\not\cong \iota_{\overline{B}}^G\hat\delta_{w_2}\cong \cind_K^G\sigma_{\mathcal{P}_{\rm max}}\otimes_\mathfrak{Z}\hat{\mathfrak{Z}}_{\chi}.$$
We then can use compatibility of $R_G$ with parabolic induction to deduce that $$\hat{\mathcal{F}}_x\cong \hat{\Ocal}_{X_{\check G,\mathcal{P}_0},x}\oplus \hat{\Ocal}_{X_{\check G,0},x}.$$
\item[(d2)] $\lambda_3=q^{-1}\lambda_2=q^{-2}\lambda_1$.  In this case $Z_{\check G,\mathcal{P}_0}$ is no longer smooth at $x$, but has a self intersection as can be seen from the description of the complete local ring: 
using a local presentation as in $(\ref{formalschemeY})$ we can compute that the complete local ring of $\hat{\Ocal}_{X_{\check G},x}$ is smoothly equivalent to 
\[C[\![t_1,t_2,t_3,u_1,u_2]\!]/((t_1-t_2)u_1,(t_2-t_3)u_2).\]
With these coordinates the completion of $Z_{\check G,\mathcal{P}_{\rm min}}$ at $x$ is given by the vanishing locus $V(t_1-t_2,t_2-t_3)$ and the completion of $Z_{\check G,{\mathcal{P}_{\rm max}}}$ is given by $V(u_1,u_2)$. Moreover, both are smooth at $x$. 
However, the completion of $Z_{\check G,\mathcal{P}_0}$ is given by $V(t_1-t_2,u_2)\cup V(u_1,t_2-t_3)$, i.e.~it decomposes into two components, say $\hat{Z}_1$ and $\hat{Z}_2$.
Note that this computation implies that $Z_{\check G,\mathcal{P}_0}$ can not be Cohen-Macaulay at $x$, as it has a self intersection in codimension $2$.
We can compute the completions of the compactly induced representation:
\begin{align*}
\cind_K^G\sigma_{\mathcal{P}_{\rm max}}\otimes_{\mathfrak{Z}}\hat{\mathfrak{Z}}_\chi&\cong \iota_{\overline{B}}^G \hat\delta_{1}\\ 
\cind_K^G\sigma_{\mathcal{P}_0}\otimes_{\mathfrak{Z}}\hat{\mathfrak{Z}}_\chi&\cong \iota_{\overline{B}}^G \hat\delta_{w_1}\oplus \iota_{\overline{B}}^G \hat\delta_{w_2}\\ 
\cind_K^G\sigma_{\mathcal{P}_{\rm min}}\otimes_{\mathfrak{Z}}\hat{\mathfrak{Z}}_\chi&\cong \iota_{\overline{B}}^G \hat\delta_{w_0},
\end{align*}
where $w_0\in W$ is the longest element and $w_1,w_2\in W\backslash\{1,w_0\}$. 
Here the elements $w_1,w_2$ are chosen such that $$\{\iota_{\overline{B}}^G \hat\delta_{1},\iota_{\overline{B}}^G \hat\delta_{w_1},\iota_{\overline{B}}^G \hat\delta_{w_2},\iota_{\overline{B}}^G \hat\delta_{w_0}\}= \{\iota_{\overline{B}}^G \hat\delta_{w},\ w\in W\}$$
is the set (consisting of four pairwise non-isomorphic elements) of induced representations of the form $\iota_{\overline{B}}^G \hat\delta_{w}$.
Using compatibility with parabolic induction we deduce that (in the coordinates introduced above) 
\begin{align*}\hat{\mathcal{F}}_x=&C[\![t_1,t_2,t_3,u_1,u_2]\!]/((t_1-t_2)u_1,u_2)\\ &\oplus C[\![t_1,t_2,t_3,u_1,u_2]\!]/(u_1,(t_2-t_3)u_2).\end{align*}
In other words the completion $\hat{\mathcal{F}}_x$ is the direct sum of the structure sheaves of $\hat X_{\check G,0,x}\cup \hat Z_1$ and $\hat X_{\check G,0,x}\cup \hat Z_2$.
\end{enumerate}
\end{expl}


\begin{thebibliography}{5mm}
\bibitem{BLGGT}
T.~Barnet-Lamb, T.~Gee, D.~Geraghty, R.~Taylor, \emph{Potential automorphy and change of weight}. Ann.~Math.~{\bf 179}, pp.~501-609 (2014).
\bibitem{BZCHN}
D.~Ben-Zvi, H.~Chen, D.~Helm, D.~Nadler, \emph{Coherent Springer theory and the categorical Deligne-Langlands correspondence}, preprint, arXiv:2010.02321v3 [math.RT].
\bibitem{Bernstein}
J.~Bernstein, \emph{Representations of $p$-adic groups}, course notes, \url{http://www.math.tau.ac.il/~bernstei/Publication_list/publication_texts/Bernst_Lecture_p-adic_repr.pdf}.
\bibitem{BernsteinZele}
J.~Bernstein, A.V.~Zelevinsky, \emph{Induced representations of reductive $\mathfrak{p}$-adic groups, I}, Annales scientifiques de l'\'Ecole Normale Sup\'erieure, S\'erie 4, {\bf 10} (1977) no.~4, pp.~441-472.
\bibitem{BoosM}
M.~Boos, M.~Bulois, \emph{Parabolic Conjugation and Commuting Varieties} 
Transformation Groups {\bf 24} (2019), pp.~951-986. 
\bibitem{BreuilSchneider}
C.~Breuil, P.~Schneider, \emph{First steps towards a $p$-adic Langlands functoriality}, J.~Reine Angew.~Math.~{\bf 610} (2007), pp.~149-180.
\bibitem{BushnellHenniart}
C.J.~Bushnell, G.~Henniart, \emph{Generalized Whittaker Models and the Bernstein center}, Amer.~J.~Math.~{\bf 125} (2003), no.~3, pp.~513-547.
\bibitem{BushnellKutzko}
C.J.~Bushnell, P.C.~Kutzko, \emph{Semisimple Types in $\GL_n$}, Compositio Math.~{\bf 119} (1999), pp. 53-97.
\bibitem{ChrissGinz}
N.~Chriss, V.~Ginzburg, \emph{Representation Theory and Complex Geometry}, Birkh\"auser (1997). 
\bibitem{DHKM}
J.-F.~Dat, D.~Helm, R.~Kurinczuk, G.~Moss, \emph{Moduli of Langlands Parameters}, preprint, arXiv:2009.06708 [math.NT].
\bibitem{DG}
V.~Drinfeld, D.~Gaitsgory, \emph{On some finiteness questions for algebraic stacks}, Geom.~Funct.~Anal. {\bf 23} (2013).
\bibitem{EGA4}
A.~Grothendieck, J.~Dieudonn\'e, \emph{\'El\'ements de g\'eom\'etrie alg\'ebrique IV: \'Etude locale des sch\'emas et des morphismes de sch\'emas (premi\'ere partie)}, Publ.~Math.~I.H.\'E.S. {\bf 20} (1964).
\bibitem{EmertonHelm}
M.~Emerton, D.~Helm,\emph{The local Langlands correspondence for $\GL_n$ in families}, Annales scientifiques de l'\'Ecole Normale Sup\'erieure, S\'erie 4,  {\bf 47} no.~4 (2014), pp.~655-722.
\bibitem{FS}
L.~Fargues, P.~Scholze, \emph{Geometrization of the local Langlands correspondence}, preprint, arXiv:2102.13459 [math.RT].
\bibitem{Gaitsgory}
D.~Gaitsgory, \emph{Outline of the proof of the geometric Langlands conjecture for $\GL(2)$}, Ast\'erisque {\bf 370} (2015), pp.~1-112.
\bibitem{Haines}
T.~Haines, \emph{The stable Bernstein center and test functions for Shimura varieties}, in: F.~Diamond, P.~Kassaei and M.~Kim (Eds.), Automorphic Forms and Galois Representations (London Math.~Soc.~Lecture Note Series, pp.~118-186) Cambridge University Press. 
\bibitem{HainesKottwitzPrasad}
T.~Haines, R.~Kottwitz, A.~Prasad, \emph{Iwahori-Hecke algebras}, Journal of the Ramanujan Mathematical Society, {\bf 25} (2010), pp.~113-145.
\bibitem{HartlHellmann}
U.~Hartl, E.~Hellmann, \emph{The universal family of semi-stable $p$-adic Galois representations}, Algebra and Number Theory {\bf 14} (2020) Issue 5, pp.~1055-1121.
\bibitem{Helm1}
D.~Helm, \emph{Whittaker models and the integral Bernstein center for ${\rm GL}_n$}, Duke Math.~J.~{\bf 165}, no.~9 (2016), pp.~1597-1628.
\bibitem{Helm2}
D.~Helm, \emph{Curtis homomorphisms and the integral Bernstein center for ${\rm GL}_n$}, Algebra and Number Theory {\bf 14} (2020) Issue 9, pp.~2607-2645.
\bibitem{HelmMoss}
D.~Helm, G.~Moss, \emph{Converse theorems and the local Langlands correspondence in families}, Inventiones math.~{\bf 214} (2018), pp.~999-1022.
\bibitem{KL}
D.~Kazhdan, G.~Lusztig, \emph{Proof of the Deligne-Langlands conjecture for Hecke algebras}, Inventiones math.~{\bf 87} (1987), pp.~153-215.
\bibitem{Kudla}
S.~Kudla, \emph{The local Langlands correspondence: the non-Archimedean case}, in: Motives (Seattle, 1991), Proc.~Symp.~Pure Math.~{\bf 55} no.2 (1994), pp.~365-392.
\bibitem{LM}
G.~Laumon, L.~Moret-Bailly, \emph{Champs alg\'ebriques}, Ergebnisse der Mathematik und ihrer Grenzgebiete (3. Folge), {\bf 39} (2000). Springer.
\bibitem{Matsumura}
H.~Matsumura, \emph{Commutative ring theory}, Cambridge studies in advanced mathematics {\bf 8}. Cambridge University Press (1989).
\bibitem{Pyvovarov1}
A.~Pyvovarov, \emph{The endomorphism ring of projectives and the Bernstein centre}, Journal de Th\'eorie de Nombres Bordeaux {\bf 32} no.1 (2020), pp.~49-71.
\bibitem{Pyvovarov2}
A.~Pyvovarov, \emph{Generic smooth representations}, Documenta Math.~{\bf 25} (2020), pp.~2473-2485.
\bibitem{Pyvovarov3}
A.~Pyvovarov, \emph{On the Breuil-Schneider conjecture: generic case}, Algebra and Number Theory {\bf 15} (2020), Issue 2, pp.~309-339.
\bibitem{SchneiderZink}
P.~Schneider, E.W.~Zink, \emph{$K$-types for the tempered components of a $p$-adic general linear group}, J.~Reine Angew.~Math.~{\bf 517} (1999), pp.~161-208.
\bibitem{Shotton}
J.~Shotton, \emph{The Breuil-M\'ezard conjecture when $\ell\neq p$}, Duke Math.~J.~{\bf 167} no.~4 (2018), pp.~603-678.
\bibitem{stacksproject}
The Stacks Project Authors, \emph{Stacks Project}, \url{http://stacks.math.columbia.edu}. 
\bibitem{Zele}
A.V.~Zelevinsky, \emph{Induced representations of reductive $\mathfrak{p}$-adic groups, II. On irreducible representations of ${\rm GL}(n)$},
Annales scientifiques de l'\'Ecole Normale Sup\'erieure, S\'erie 4, {\bf 13} (1980) no.~2, pp.~165-210.
\bibitem{Zhu}
X.~Zhu, \emph{Coherent sheaves on the stack of Langlands parameters}, perprint, arXiv:2008.02998v2 [math.AG].
\end{thebibliography}
\end{document}